\documentclass[final]{siamltex}
\pdfoutput=1

\usepackage{amsmath,amsfonts}
\usepackage{graphicx}
\usepackage{url}
\usepackage{subfig}
\usepackage{float}
\usepackage{xspace}

\usepackage{kbordermatrix}

\graphicspath{{./}{./figures/}}
\usepackage{epstopdf}

  \def\clap#1{\hbox to 0pt{\hss#1\hss}}
  \def\mathclap{\mathpalette\mathclapinternal}
  \def\mathclapinternal#1#2{%
    \clap{$\mathsurround=0pt#1{#2}$}}

\providecommand{\mat}[1]{\boldsymbol{#1}}%
\renewcommand{\vec}[1]{\mathbf{#1}}

\providecommand{\mC}{\ensuremath{\mat{C}}}

\providecommand{\mF}{\ensuremath{\mat{F}}}

\providecommand{\mQ}{\ensuremath{\mat{Q}}}
\providecommand{\mR}{\ensuremath{\mat{R}}}

\providecommand{\mU}{\ensuremath{\mat{U}}}
\providecommand{\mV}{\ensuremath{\mat{V}}}

\providecommand{\mSigma}{\ensuremath{\mat{\Sigma}}}

\providecommand{\vf}{\ensuremath{\vec{f}}}

\providecommand{\vs}{\ensuremath{\vec{s}}}

\providecommand{\vu}{\ensuremath{\vec{u}}}
\providecommand{\vv}{\ensuremath{\vec{v}}}

\newcommand{\sE}{\mathcal{E}}
\newcommand{\sI}{\mathcal{I}}

\newcommand{\sS}{\mathcal{S}}
\newcommand{\sX}{\mathcal{X}}

\newcommand{\Exp}[1]{\mathbb{E}\left[#1\right]}
\newcommand{\Cov}[2]{\operatorname{Cov}\left[#1,\,#2\right]}

\newcommand{\Var}[1]{\operatorname{Var}\left[#1\right]}

\newcommand{\bmat}[1]{\begin{bmatrix}#1\end{bmatrix}}

\newcommand{\Map}{\textit{map}\xspace}
\newcommand{\Reduce}{\textit{reduce}\xspace}
\newcommand{\Shuffle}{\textit{shuffle}\xspace}

\usepackage{color}

\newtheorem{assump}{Assumption}
\newtheorem{prop}{Proposition}

\title{Model Reduction With MapReduce-enabled Tall and Skinny
  Singular Value Decomposition}

\author{ Paul G.~Constantine\thanks{Colorado School of Mines Applied
    Mathematics and Statistics, Golden, Colorado 80401 ({\tt
      paul.constantine@mines.edu}).}  \and David
  F.~Gleich\thanks{Purdue CS, West Lafayette, Indiana 47907 ({\tt
      dgleich@purdue.edu}).}  \and Yangyang Hou\thanks{Purdue CS, West
    Lafayette, Indiana 47907 ({\tt hou13@purdue.edu}).}  \and \newline Jeremy
  Templeton\thanks{Sandia National Laboratories, Livermore, California
    ({\tt jatempl@sandia.gov}). Sandia National Laboratories is a
    multi-program laboratory managed and operated by Sandia
    Corporation, a wholly owned subsidiary of Lockheed Martin
    Corporation, for the U.S. Department of Energy's National Nuclear
    Security Administration under contract DE-AC04-94AL85000.}  }

\begin{document}

\maketitle

\begin{abstract}
We present a method for computing reduced-order models of
parameterized partial differential equation solutions. The key
analytical tool is the singular value expansion of the parameterized
solution, which we approximate with a singular value decomposition of
a parameter snapshot matrix. To evaluate the reduced-order model at a
new parameter, we interpolate a subset of the right singular vectors
to generate the reduced-order model's coefficients. We employ a novel
method to select this subset that uses the parameter gradient of the
right singular vectors to split the terms in the expansion yielding a
mean prediction and a prediction covariance---similar to a Gaussian
process approximation. The covariance serves as a confidence measure
for the reduce order model.
 
We demonstrate the efficacy of the reduced-order model using a
parameter study of heat transfer in random media. The high-fidelity
simulations produce more than 4TB of data; we compute the singular
value decomposition and evaluate the reduced-order model using
scalable MapReduce/Hadoop implementations.  We compare the accuracy of
our method with a scalar response surface on a set of temperature
profile measurements and find that our model better captures sharp,
local features in the parameter space.
\end{abstract}


\pagestyle{myheadings}
\thispagestyle{plain}
\markboth{P.~G. CONSTANTINE, D.~F. GLEICH, Y. HOU, AND J. TEMPLETON}{MAPREDUCE-ENABLED
  MODEL REDUCTION}

\section{Introduction \& motivation}
\label{sec:intro}

High-fidelity simulations of partial differential equations are
typically too expensive for design optimization and uncertainty
quantification, where many independent runs are necessary. Cheaper
reduced-order models (ROMs) that approximate the map from simulation
inputs to quantities of interest may replace expensive simulations to
enable such parameter studies. These ROMs are constructed with a
relatively small set of high-fidelity runs chosen to cover a range of
input parameter values. Each evaluation of the ROM is a linear
combination of basis functions derived from the outputs of the
high-fidelity runs; ROM constructions differ in their choice of basis
functions and method for computing the coefficients of the linear
combination. Projection-based methods project the residual of the
governing equations (i.e., a Galerkin projection) to create a
relatively small system of equations for the coefficients; see the
recent preprint~\cite{Benner2013} for a survey of projection-based
techniques. Alternatively, one may derive a closely related
optimization problem to compute the
coefficients~\cite{Bui2008,carlberg2013gnat,Constantine2012}. These
two formulations can provide a measure of confidence along with the
ROM. However, they are often difficult to implement in existing
solvers since they need access to the equation's operators or
residual.

To bypass the implementation difficulties, one may use response
surfaces---e.g., collocation~\cite{Babuska2007,Xiu2005} or Gaussian
process regression~\cite{Rasmussen2006}---which are essentially
interpolation methods applied to the high-fidelity outputs. They do
not need access to the differential operators or residuals and are
therefore relatively easy to implement. However, measures of confidence
are more difficult to formulate and compute. Several works have
explored using interpolation instead of projection or optimization to
compute the ROM coefficients. This approach is justified when the
governing equations are unknown or only a set of PDE solutions are
available~\cite{ly2001modeling}. It is also useful when nonlinearities
in the governing equations prohibit a theoretically sound Galerkin
projection~\cite{audouze2009reduced,audouze2013nonintrusive}. For
these reasons, it has been applied to several problems in aerospace
engineering~\cite{Bui2003,lee2011reduced,goss2008inlet,ostrowski2005estimation}.

In this paper, we extend these ideas by equipping an
interpolation-based ROM with a novel parameter-dependent confidence
measure. We first view the ROM from the perspective of the singular
value expansion (SVE) of the parameterized PDE solution, where the
left singular functions depend on space and time, and the right
singular functions depend on the input parameters.  We approximate
these functions with the components of a singular value decomposition
(SVD) of a tall, dense matrix of solutions computed at a set of input
parameter values, i.e., parameter \emph{snapshots}. Many reduced basis
methods use the left singular vectors of the snapshot matrix for the
ROM basis, where each snapshot represents a spatially varying
solution.  In contrast, each column of our snapshot matrix contains
the full spatio-temporal solution for a given parameter.  We have
observed interesting behavior in the right singular vectors in several
applications: as the index of the singular vector increases, its
components---when viewed as evaluations of a parameter-dependent
function---become more oscillatory. This is consistent with the common
use of the phrase ``higher order modes'' to describe the singular
vectors with large indices. More importantly, the rate that the
singular vectors become oscillatory depends on the parameter. In
particular, the singular vectors become oscillatory faster in regions
of the parameter space where the PDE solution changes rapidly with
small parameter perturbations.

We exploit this observation to devise a heuristic for choosing the
subset of the left singular vectors comprising ROM; instead of
selecting all left singular vectors whose corresponding singular value
is above a chosen threshold, we examine the gradients of the right
singular vectors at the parameter value where we wish to evaluate the
ROM.  After some critical index, the right singular vectors are too
irregular to safely interpolate.  For each ROM evaluation, the right
singular vectors are divided into two categories: (i) those that are
sufficiently smooth for interpolation, and (ii) those that oscillate
too rapidly. The first category identifies the left singular vectors
used in the ROM, and the coefficients are computed with interpolation.
The remaining left singular vectors are used to compute a measure of
confidence similar to the prediction variance in Gaussian process
regression. The number of left singular vectors in each category may
be different for different ROM evaluations depending on the
irregularity of the right singular vectors at the interpolation point;
we explore this in the numerical examples.  The heuristics we employ
to categorize the right singular vectors are based on the work of
Hansen in the context of ill-posed inverse problems~\cite{Hansen2010}. We 
describe the ROM methodology in Section \ref{sec:rom}. 

In Section \ref{sec:exp}, we demonstrate the ROM and its confidence
measure with a parameter study of heat transfer in random media. A
brick is heated on one side, and we measure how much heat transfers to
the opposite side given the parameterized thermal conductivity of the
material. The high-fidelity simulations use Sandia National
Laboratories' finite element production code Aria~\cite{sand07-2734}
on its capacity cluster with a mesh containing 4.2M elements. The
study uses 8192 simulations, which produce approximately 4 terabytes
of data. We study the effectiveness of the reduced-order model and
compare its predictions with a response surface on two relevant scalar
quantities of interest computed from the full temperature
distribution.

Given the data-intensive computing requirements of this and similar
applications, we have chosen to implement the ROM in the popular
Hadoop distribution~\cite{Hadoop2012-0202-cdh3} of
MapReduce~\cite{Dean2004-MapReduce}. By expressing each step of 
the ROM construction in the MapReduce framework, we take advantage
of its built-in parallelism and fault tolerance. MapReduce's scalability
enables us to compute the SVD of a snapshot matrix with 64 columns and
roughly five billion rows---approximately 2.3 terabytes of data---without
the custom hard disk I/O that would be necessary to use standard
parallel routines such as Scalapack~\cite{blackford1997scalapack}. 
The algorithm we use for the SVD in MapReduce is based on the 
communication-avoiding QR decomposition~\cite{Demmel-2012-CAQR} as
described in our previous work~\cite{Constantine2011}. We present
the implementation of the ROM in Section \ref{sec:implement}. 

\section{A reduced-order modeling approach}
\label{sec:rom}

Let $f=f(x,s)$ be the solution of a partial differential equation
(PDE), where $x\in\sX\subset\mathbb{R}^4$ are the space and time
coordinates (three spatial dimensions and a temporal dimension), and
$s\in\sS\subset\mathbb{R}$ is an input parameter.  We restrict our
attention to models with a single scalar parameter to keep the
presentation simple. The ROM employs both interpolation and
approximations of derivatives in the parameter space $\sS$.  While
these operations are possible with more than one parameter, challenges
arise in higher dimensions---e.g., the choice of interpolation nodes
and the accuracy of derivatives---that we will avoid. If one is
willing to address the difficulties of interpolation and approximating
derivatives in multiple dimensions, then our approach can be extended.

The interpretation of $f$ as the solution of a PDE is important for
two reasons.  First, the solution of many PDE models can be shown to
be a smooth function of both the space/time variables and the
parameter, and we restrict our attention to sufficiently smooth
solutions. Second, computational tools will compute $f$ at all values
of $x$ given an input $s$. In other words, we cannot evaluate $f$ at a
specific $x$ without evaluating it for every $x$. We will not discuss
refinement of the ROM---i.e., which parameter values to run a new set
of high-fidelity runs to best improve an initial ROM---but our
heuristic offers a natural criterion for such a selection.  We assume
that computing $f(x,s)$ for a particular $s$ is computationally
expensive.  We want to use the outputs from a few expensive
computations at a chosen set of input parameters to approximate
$f(x,s)$ at some other $s$ in a manner that is less computationally
expensive than solving the differential equation.

We assume that $f$ is continuous and square-integrable ($\int_\sX
\int_\sS f^2 \,ds\,dx \;<\;\infty$). In practice, the techniques we
use will perform better if $f$ is smooth, e.g., admits continuous
derivatives up to some order.
Since $f$ is continuous, it admits a uniformly convergent series
representation known as the \emph{singular value expansion} (SVE);
see~\cite{Hansen1988} for more details on the SVE:
\begin{equation}
\label{eq:sve}
f(x,s) \;=\; \sum_{k=1}^\infty \mu_k\,u_k(x)\,v_k(s).
\end{equation}
The singular functions $u_k(x)$ and $v_k(s)$ are continuous and
orthonormal,
\begin{equation}
\int_{\sX} u_{k_1}\,u_{k_2}\,dx \;=\; \int_{\sS} v_{k_1}\,v_{k_2}\,ds
\;=\; \delta_{k_1,k_2}.
\end{equation}
The singular values are positive and ordered in decreasing order,
\begin{equation}
\mu_1 \;\geq\; \mu_2 \;\geq\; \cdots \;\geq\; 0.
\end{equation}
Hansen discusses methods for approximating the factors of the SVE
using the singular value decomposition. We will employ his
construction~\cite[Section 5]{Hansen1988}, which ultimately uses
point evaluations of the function $f$ to construct a matrix suited
for the SVD.

Let $x_1,\dots,x_M$ with $x_i\in\sX$ be the points of a discretization
of the spatio-temporal domain. A run of the PDE solver produces an
approximate solution at these points in the domain for a given input
$s$. We assume that the spatio-temporal discretization is sufficient
to produce an accurate approximation of the PDE solution for all
values of $s$; in practice, such an assumption should be verified. Let
$s_1,\dots,s_N$ with $s_j\in\sS$ be a set of input parameters where
the approximate PDE solution will be computed; we call these the
\emph{training} runs. 
The number $N$ is the budget of simulations, and we expect that $N\ll
M$ for most cases. In other words, we assume that the number of nodes
in the spatio-temporal discretization is much larger than the budget
of simulations.

From these approximate solutions, we construct the tall, dense matrix
\begin{equation}
\label{eq:fmat}
\mF\;=\;
\bmat{
f(x_1,s_1) & \cdots & f(x_1,s_N) \\
\vdots & \ddots & \vdots \\
f(x_M,s_1) & \cdots & f(x_M,s_N)
}.
\end{equation}
Next we compute the thin SVD,
\begin{equation}
\mF \;=\; \mU\mSigma\mV^T,\qquad
\mSigma\;=\;\mathrm{diag}\,(\sigma_1,\dots,\sigma_N),
\end{equation}
where, following~\cite{Hansen1988}, we treat $\sigma_k\approx\mu_k$
and
\begin{equation}
\label{eq:singvecs}
\mU 
\;\approx\;
\bmat{
u_1(x_1) & \cdots & u_N(x_1) \\
\vdots & \ddots & \vdots \\
u_1(x_M) & \cdots & u_N(x_M)
},\qquad
\mV
\;\approx\;
\bmat{
v_1(s_1) & \cdots & v_N(s_1) \\
\vdots & \ddots & \vdots \\
v_1(s_N) & \cdots & v_N(s_N)
}
\end{equation}
In other words, we treat the entries of the left and right singular
vectors as evaluations of the singular functions at the points $x_i$
and $s_j$, respectively.

\subsection{Oscillations in the singular vectors}
\label{sec:oscillations}
We will leverage the work of Hansen \cite{Hansen2006,Hansen2010} on
computational methods for linear, ill-posed inverse problems to
develop the heuristics for the ROM.  He observed that, for a broad
class of integral equation kernels found in practice, the singular
functions become more oscillatory (i.e., cross zero more frequently)
as the index $k$ increases. Rigorous proofs of this observation are
available for some special cases. However, it is easy to construct
kernels whose singular functions do not behave this way. For example,
take a kernel whose singular functions become more oscillatory with
increasing $k$ and shuffle the singular functions. Such
counterexamples offer evidence that a general statement is difficult
to formulate.

We have observed similar phenomena for functions $f=f(x,s)$ coming
from parameterized partial differential equations. This observation is
corroborated by many studies in coherent structures based on the
closely related proper orthogonal
decomposition~\cite{lumley2012turbulence}.  Additionally, we have
observed that these oscillations may not increase uniformly over
the parameter domain. In particular, the rate of 
increasing oscillations may be greater in regions of the parameter
space where $f$ has large parameter gradients; we provide two 
illustrative examples below. 

The components of the singular vectors inherit the observed
oscillating behavior of the singular functions. In particular, the
oscillations increase as the index $k$ increases, and they increase
more rapidly in regions corresponding to large differences in the
elements of the data matrix $\mF$.  These rapid oscillations manifest
as an increase with $k$ in the magnitude of the difference between
entries of the singular vectors corresponding to evaluations of the
singular functions that are nearby in parameter space. For example,
the difference between $\mV_{j,k+1}\approx v_{k+1}(s_j)$ and
$\mV_{j',k+1}\approx v_{k+1}(s_{j'})$ with $s_j$ neighboring $s_{j'}$
will be greater than the difference between $\mV_{j,k}\approx
v_{k}(s_j)$ and $\mV_{j',k}\approx v_{k}(s_{j'})$. (Note that when the
model contains more than one parameter, the notion of
\emph{neighboring} becomes more complicated.) However, since there is
finite resolution in the parameter space, there is typically some $k$
after which the discretization is insufficient to represent the
oscillations, and this pattern breaks down.  The phenomenon is similar
to approximating a sequence of sine waves with increasing frequency
using the same grid.
We formalize this notion in the following assumption.
\begin{assump}
\label{bigass}
Let $\sS = [s_1,s_N]$ be a closed interval with a discretization $s_i
= s_1+(i-1)\Delta s$, where $\Delta s = (s_N-s_1)/(N-1)$, and let
$\mV$ be defined as in \eqref{eq:singvecs}. There is an $R=R_i\leq N$
such that the sequence of difference magnitudes between neighboring
right singular vector entries will increase for $k$ from 1 to $R$,
i.e.,
\begin{equation}
|\mV_{i+1,k+1} - \mV_{i,k+1}| \;>\; |\mV_{i+1,k} - \mV_{i,k}|.
\end{equation}
For $k>R$, the relationship becomes unpredictable due to the finite
resolution in the parameter space $\sS$.
\end{assump}

Note our restriction to a single parameter and a uniform
discretization of the parameter space. These restrictions can be
relaxed with appropriate discretizations of a multivariate space.  We
will use Assumption \ref{bigass} to justify a heuristic that
distinguishes between singular functions that can be resolved and
those that cannot given the discretization.

Next we give two concrete examples of the observed behavior in the
right singular vectors. The first is a steady state
advection-diffusion type boundary value problem,
\begin{equation}
\label{eq:ad}
\frac{df}{dx} + s\frac{d^2f}{dx^2} = -1, 
\qquad x\in[-10,10],\quad s\in[2,20],
\end{equation}
with homogeneous boundary conditions. The solution is given by
\begin{equation}
f(x,s) \;=\;
\frac{\exp(20s^{-1})(x-10)+20\exp(s^{-1}(10-x))-x-10}{1-\exp(20s^{-1})}
\end{equation}
The parameter $s$ represents the ratio of diffusion to advection.
Figure \ref{fig:svecs0} shows the results the SVD approximation to the
SVE factors for a overresolved model (1999 points in the
discretization of the parameter space) and an underresolved model (15
points in the parameter space). Observe how the first seven singular
functions scaled by their respective singular values become
oscillatory at different rates in different regions of the parameter
space. In particular, the singular functions oscillate more rapidly in
regions of the parameter space corresponding to more advection. Also
note how the underresolved approximations deviate from the
overresolved approximations in regions of high oscillations.

\begin{figure}[ht]
\centering
\subfloat[Surface]{
\includegraphics[width=0.3\linewidth]{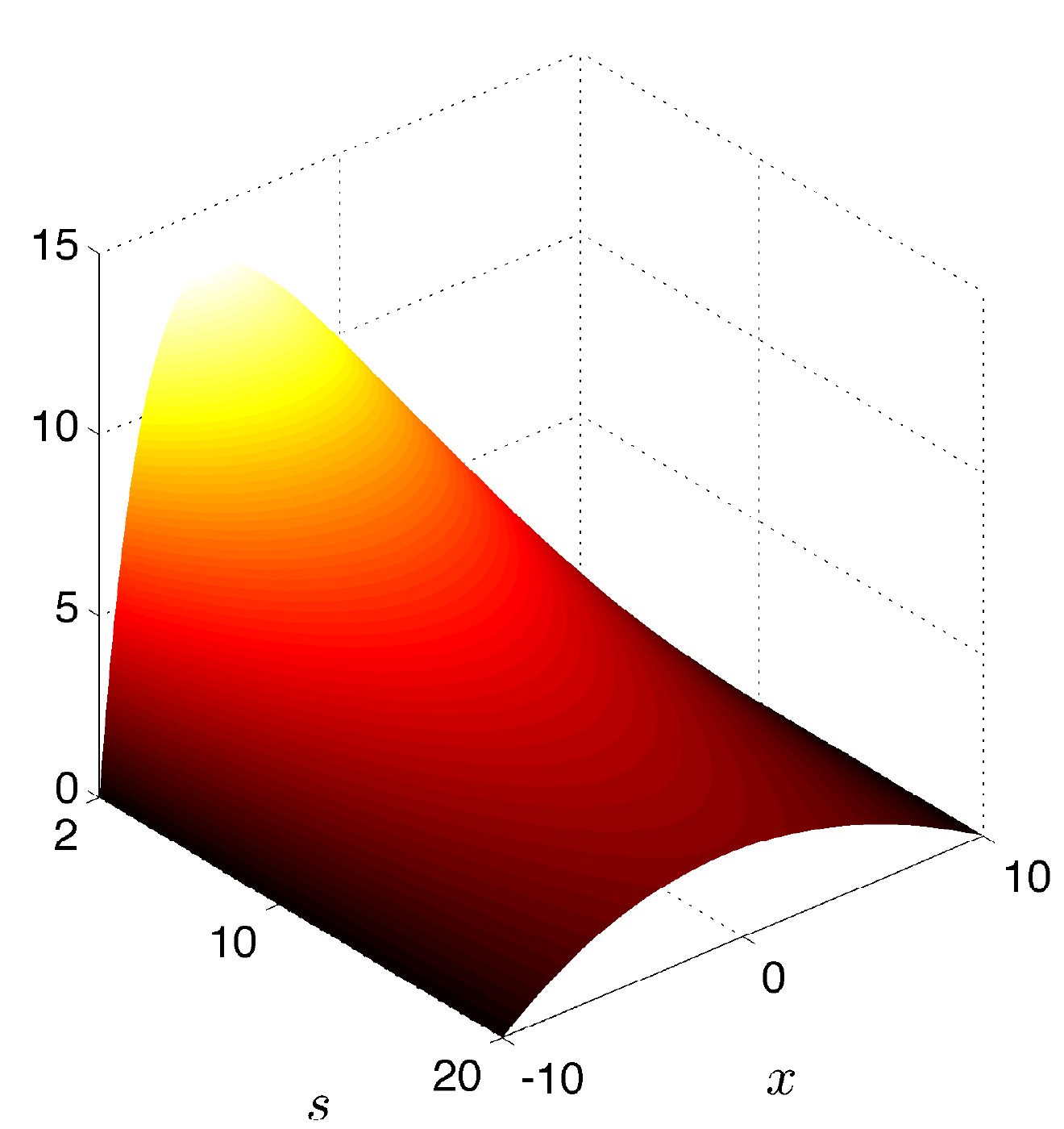}
}\;
\subfloat[Singular values]{
\includegraphics[width=0.3\linewidth]{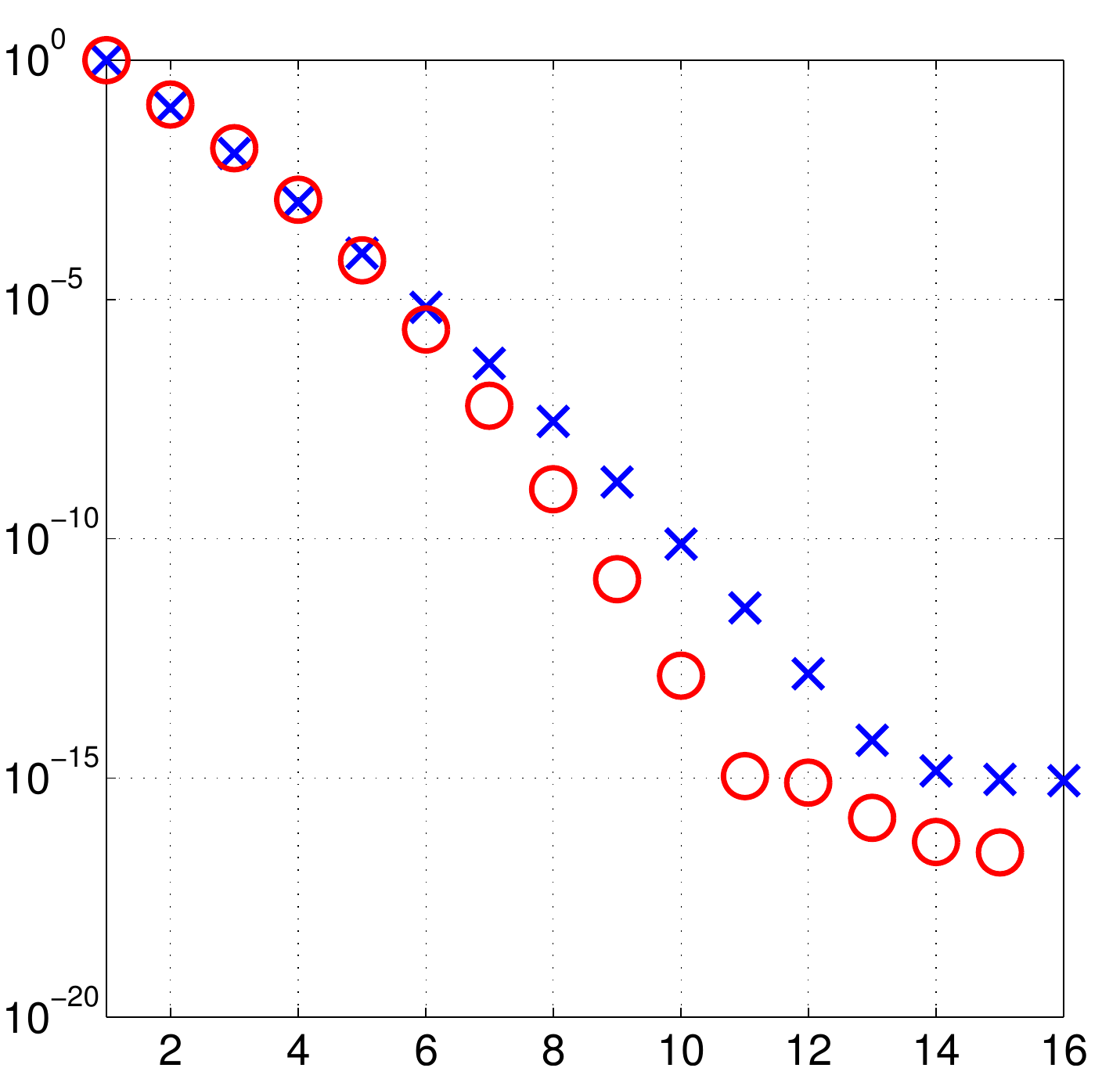}
}\;
\subfloat[$\sigma_1 v_1(s)$]{
\includegraphics[width=0.3\linewidth]{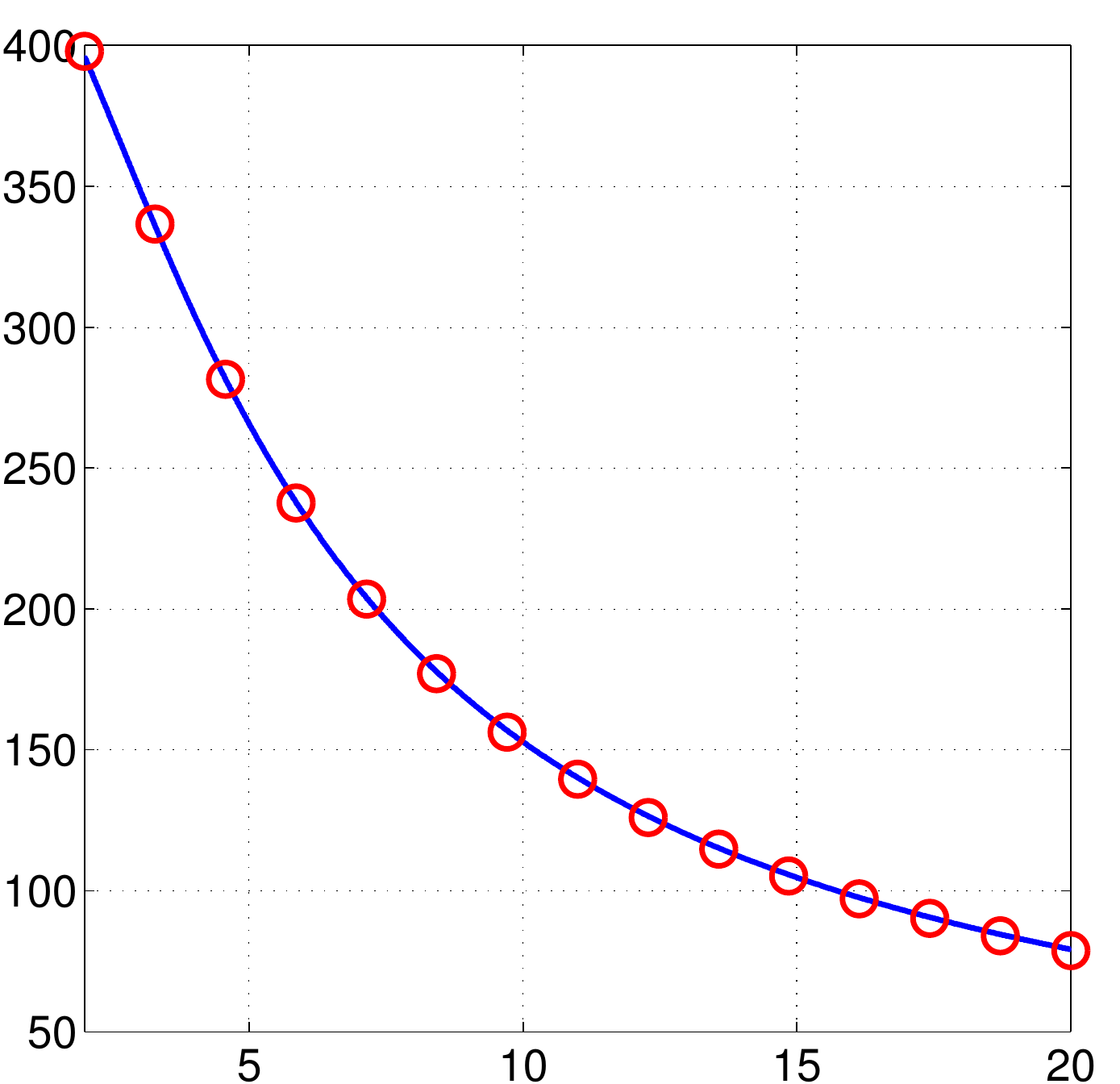}
}\\
\subfloat[$\sigma_2v_2(s)$]{
\includegraphics[width=0.3\linewidth]{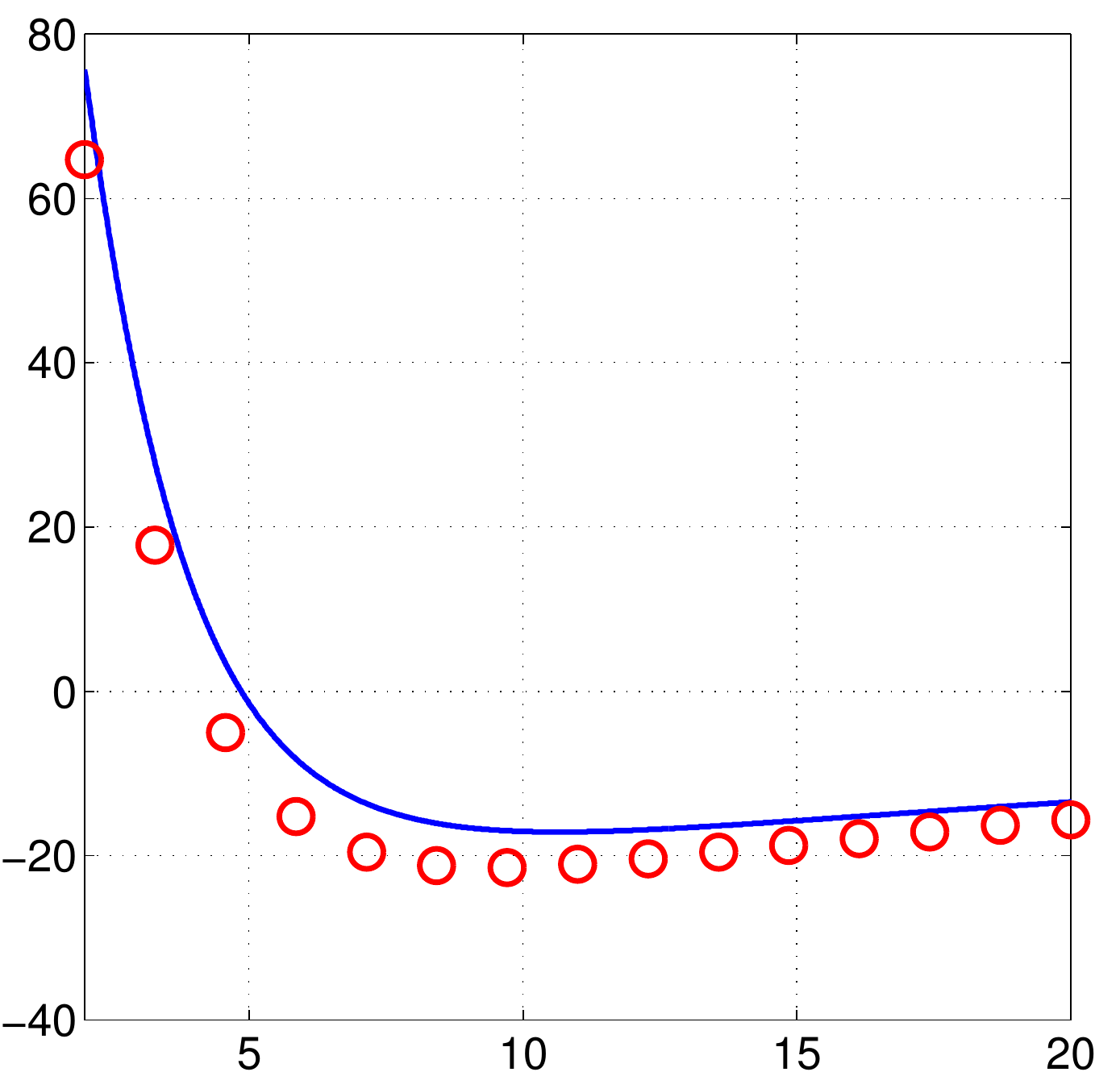}
}\;
\subfloat[$\sigma_3v_3(s)$]{
\includegraphics[width=0.3\linewidth]{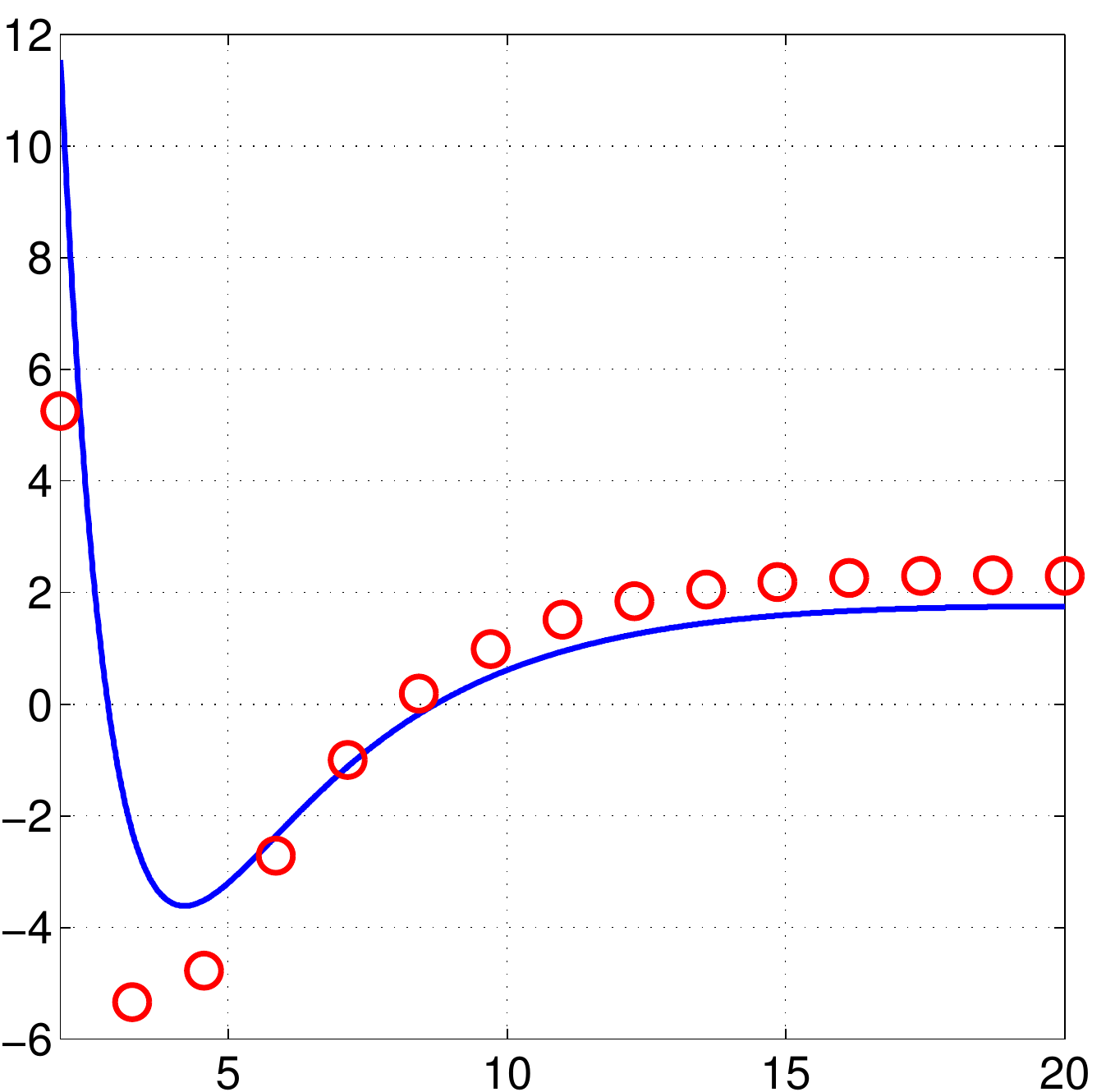}
}\;
\subfloat[$\sigma_4v_4(s)$]{
\includegraphics[width=0.3\linewidth]{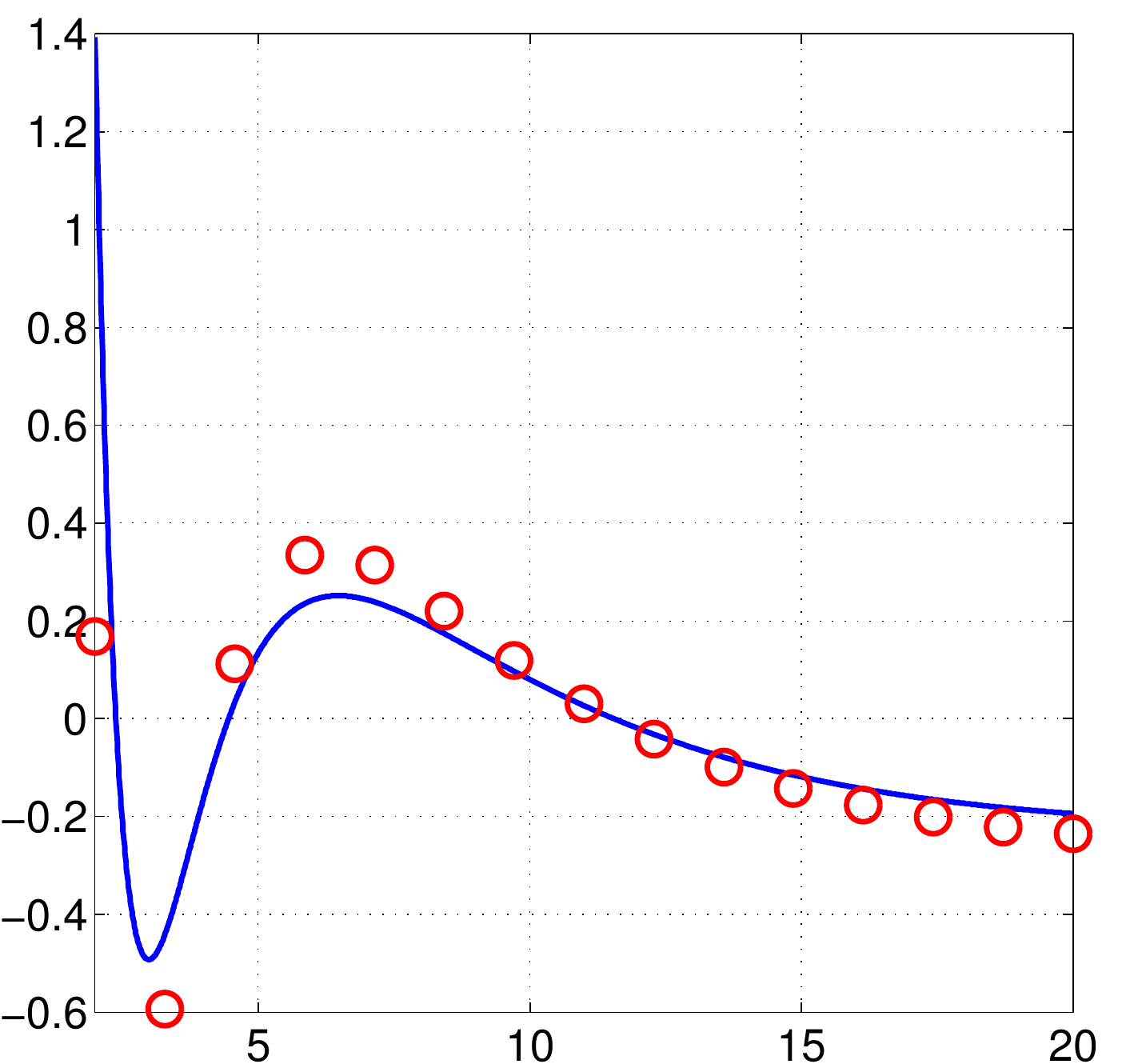}
}\\
\subfloat[$\sigma_5v_5(s)$]{
\includegraphics[width=0.3\linewidth]{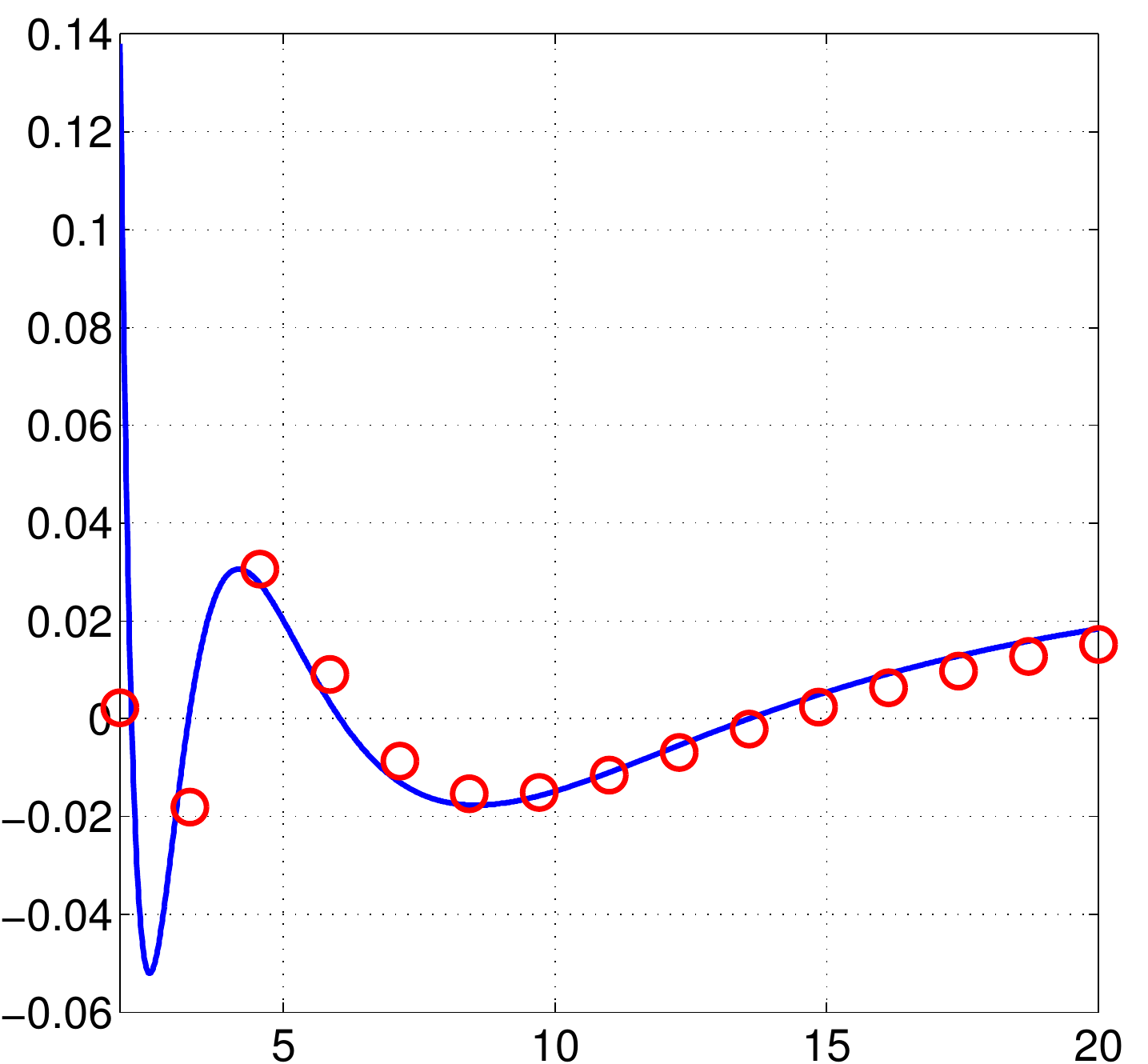}
}\;
\subfloat[$\sigma_6v_6(s)$]{
\includegraphics[width=0.3\linewidth]{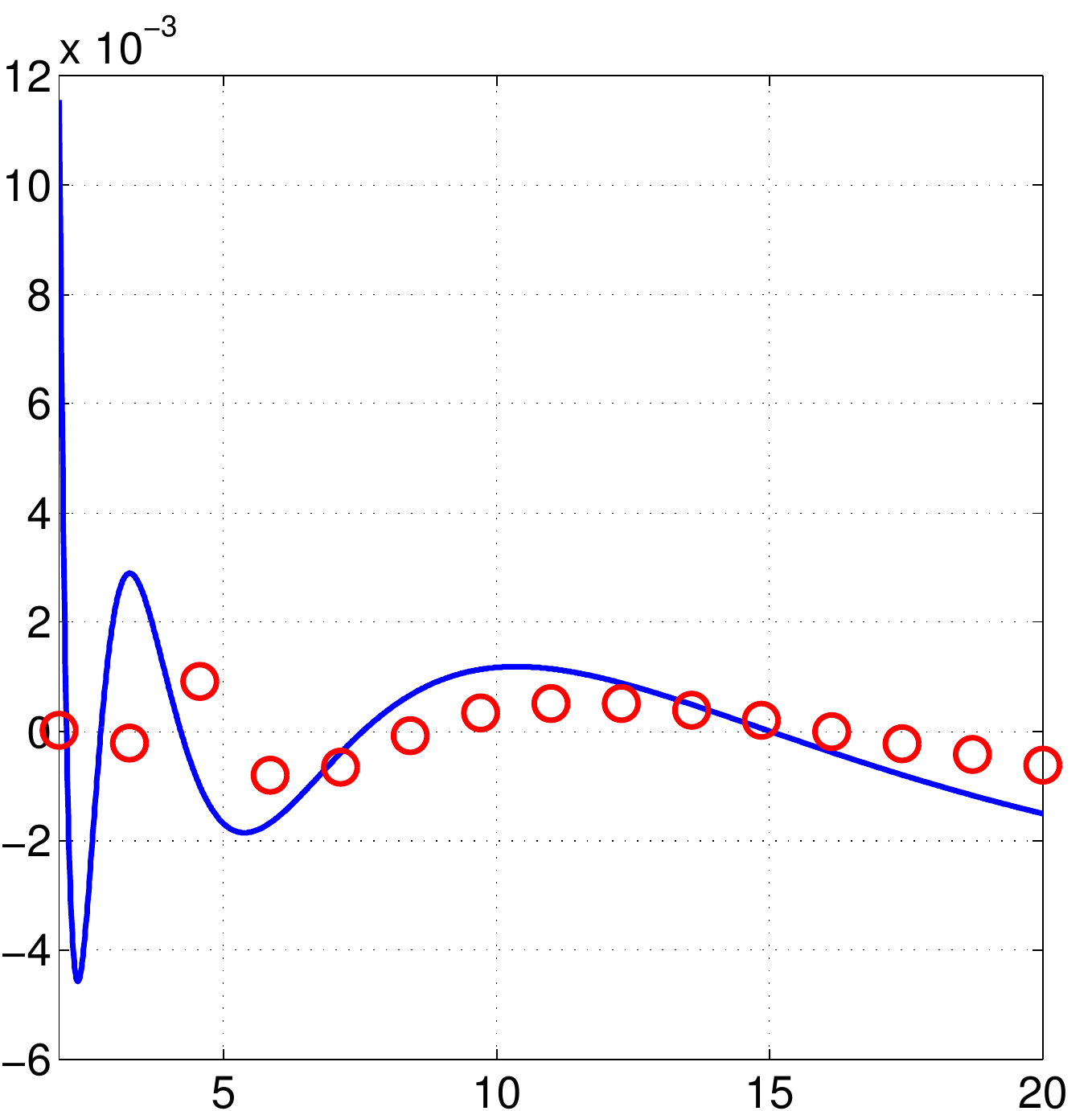}
}\;
\subfloat[$\sigma_7v_7(s)$]{
\includegraphics[width=0.3\linewidth]{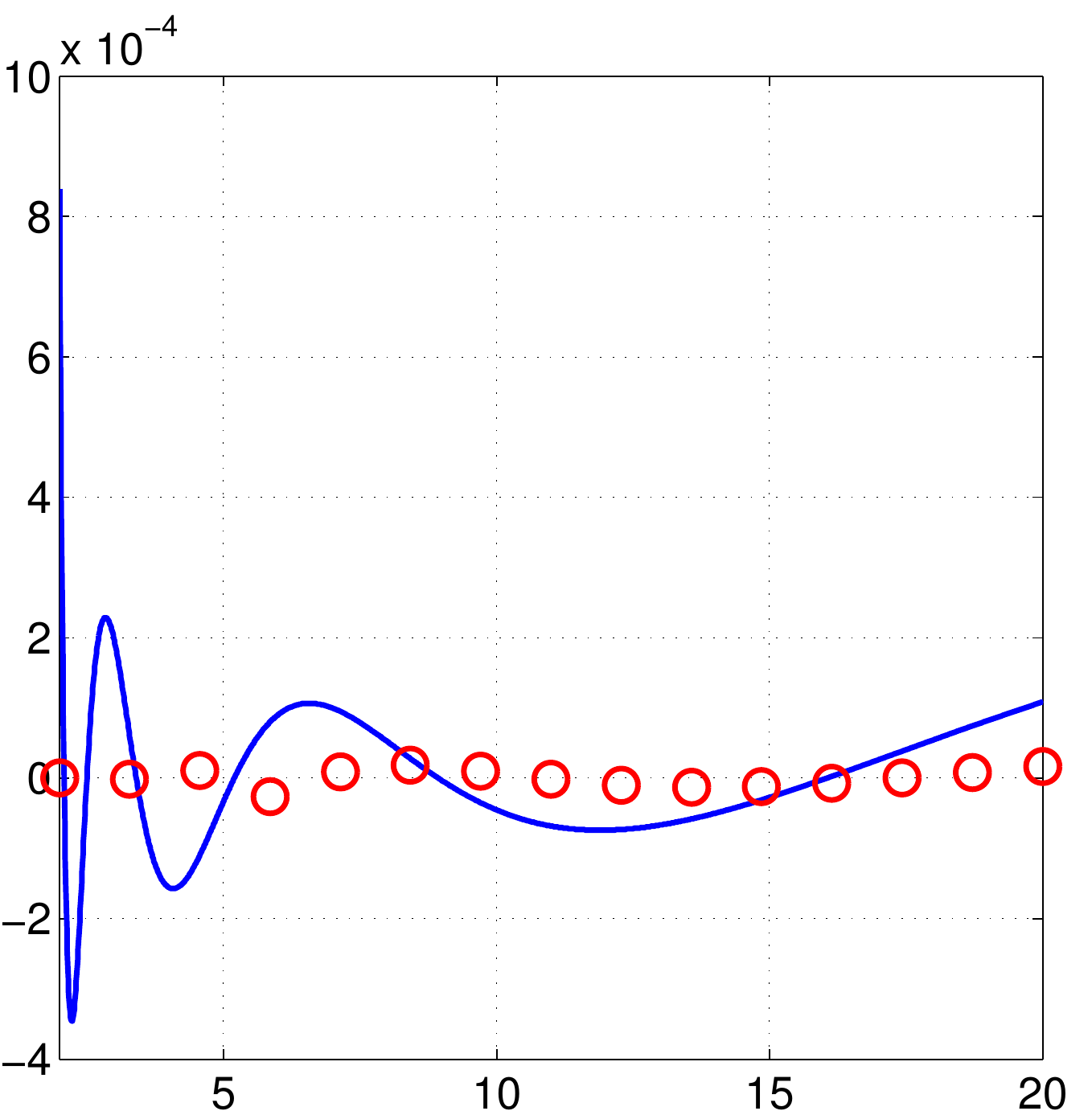}
}
\caption{The top left figure shows the solution $f(x,s)$ to the
  advection-diffusion equation \eqref{eq:ad}. The figure to its right
  shows the singular values of a finely resolved $\mF$ with 1999
  columns (blue x's) and a coarsely resolved $\mF$ with 15 columns
  (red o's). Each are scaled by the maximum singular value from each
  set. The remaining figures show the approximations of the singular
  functions $v_1(s)$ through $v_7(s)$ scaled by the respective
  singular values for $\mF$ with 1999 columns (blue lines) and 15
  columns (red o's). (Colors are visible in the electronic version.) }
\label{fig:svecs0}
\end{figure}

The second example is another second order boundary value problem with
spatially varying coefficients,
\begin{equation}
\label{eq:bvp}
-\frac{d}{dx} \left(a\,\frac{df}{dx} \right) \;=\; 1,\qquad x\in[0,1],
\end{equation}
with homogeneous boundary condtions, and
\begin{equation}
a \;=\; a(x,s) \;=\; 1 + 4s(x^2-x),\qquad s\in[0.1,0.9],
\end{equation}
The solution is
\begin{equation}
\label{eq:toysoln}
f(x,s) \;=\; -\frac{1}{8s}\log\left( 1+ 4s(x^2-x)\right),
\end{equation}
which is plotted in the top left of Figure \ref{fig:svecs1}.  Outside
the domain, $f$ has a singularity at $(x=0.5,s=1)$, which causes $f$
to grow rapidly along the line $x=0.5$ near the boundary $s=0.9$. This
local feature of the solution results in more rapid oscillations of
the singular functions $v_k(s)$ near the boundary $s=0.9$. The first
seven singular functions, scaled by the singular values, are plotted
in Figure \ref{fig:svecs1}. The rapid oscillations near the parameter
boundary $s=0.9$ are clearly visible.  In those same figures, we plot
the components of the corresponding singular vectors, scaled by the
singular values, for a data matrix $\mF$ with columns computed at
eleven equally spaced parameter values in the interval
$[0.1,0.9]$. Notice how the components of the singular vectors deviate
from the singular functions as $k$ increases, particularly in the
regions of rapid oscillations.

\begin{figure}[ht]
\centering
\subfloat[Surface]{
\includegraphics[width=0.3\linewidth]{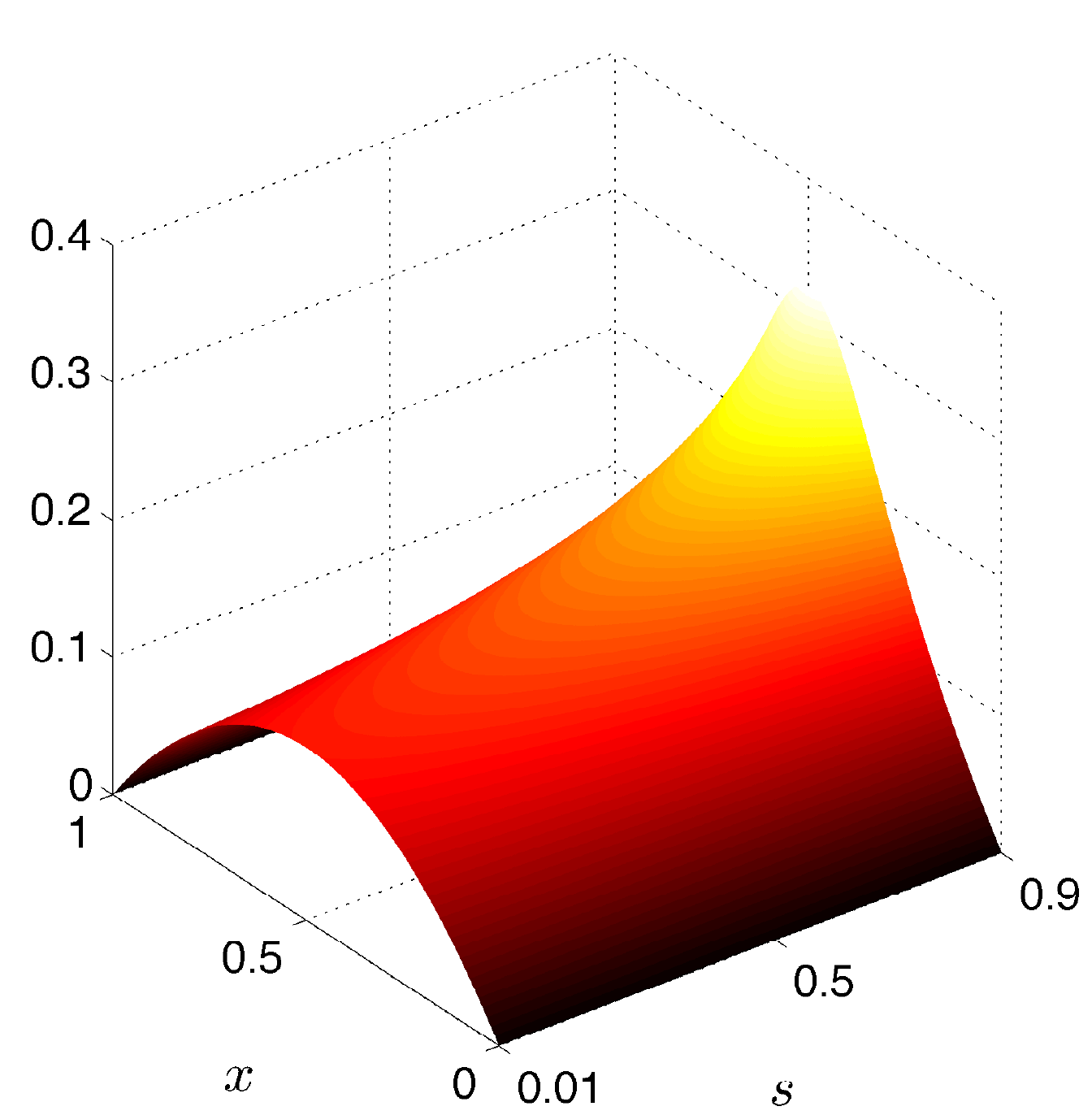}
}\;
\subfloat[Singular values]{
\includegraphics[width=0.3\linewidth]{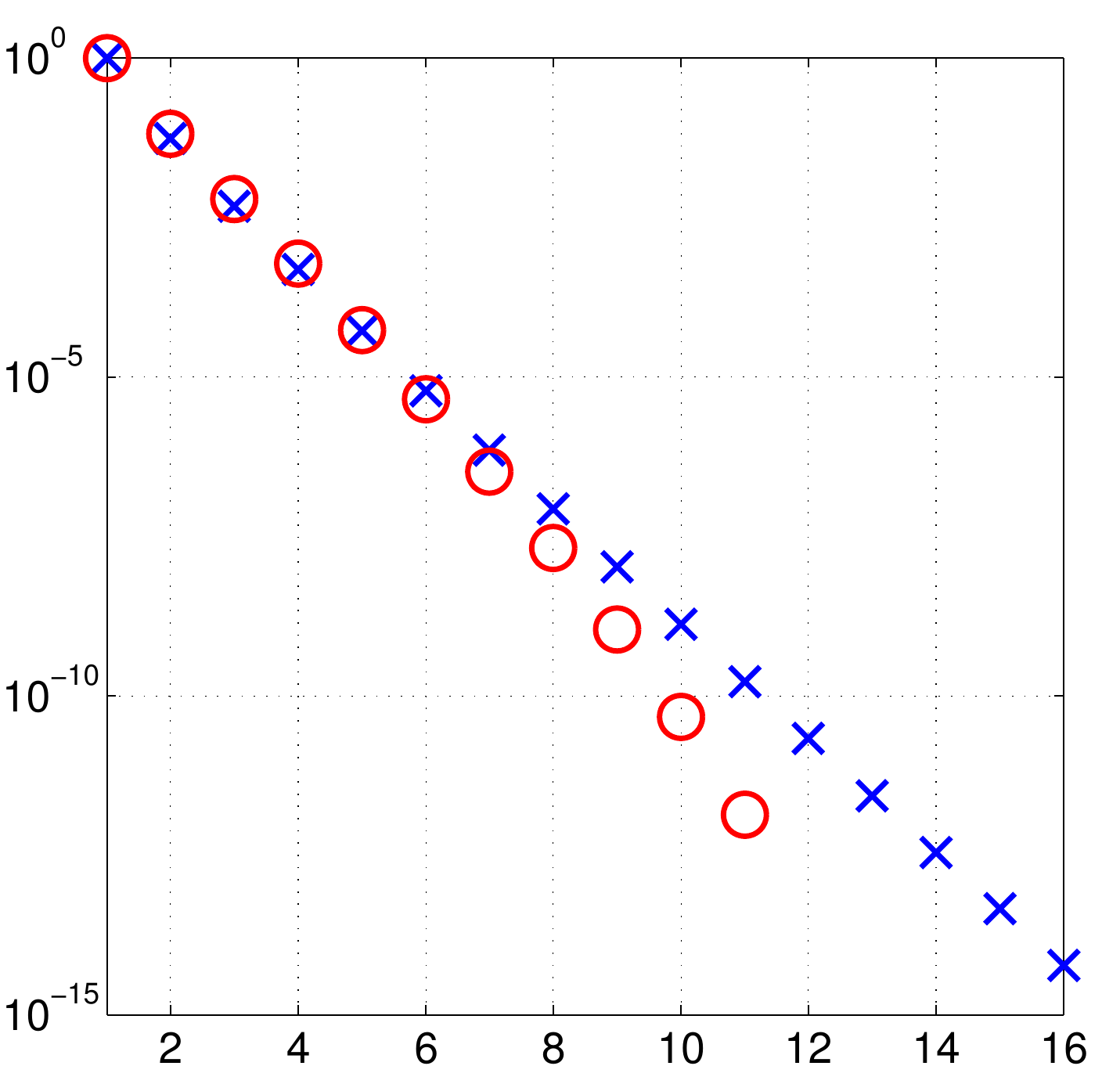}
}\;
\subfloat[$\sigma_1 v_1(s)$]{
\includegraphics[width=0.3\linewidth]{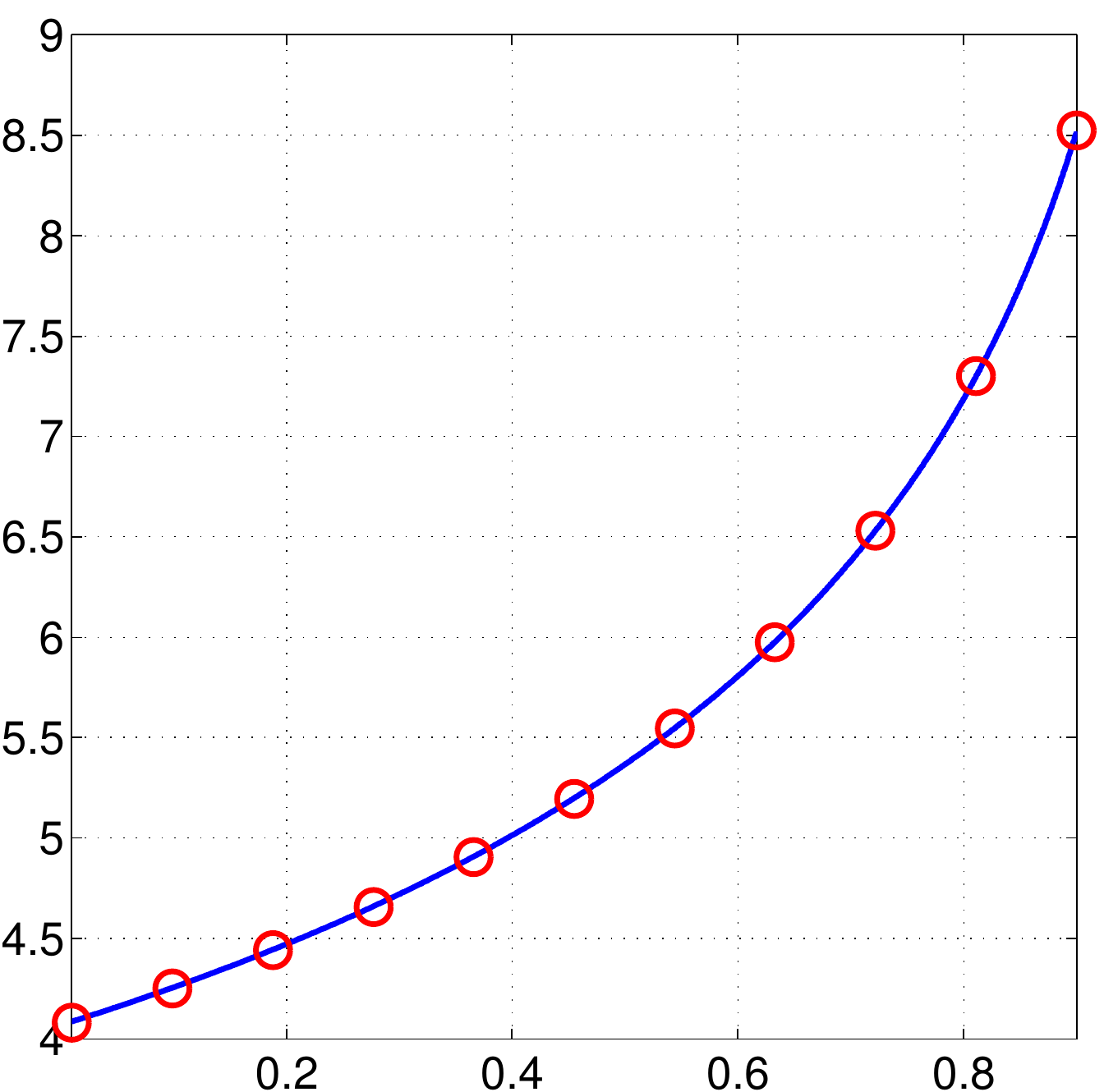}
}\\
\subfloat[$\sigma_2v_2(s)$]{
\includegraphics[width=0.3\linewidth]{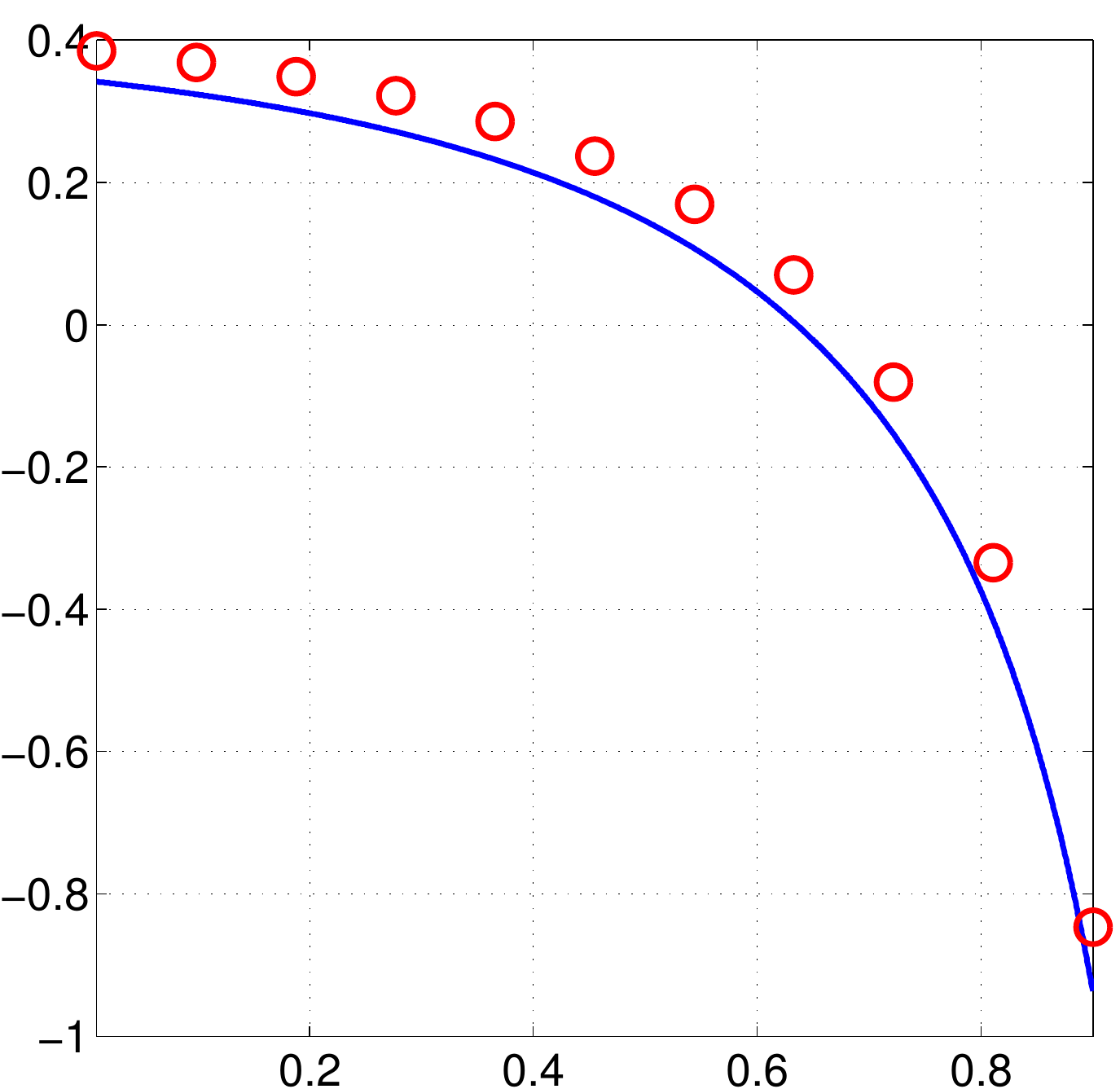}
}\;
\subfloat[$\sigma_3v_3(s)$]{
\includegraphics[width=0.3\linewidth]{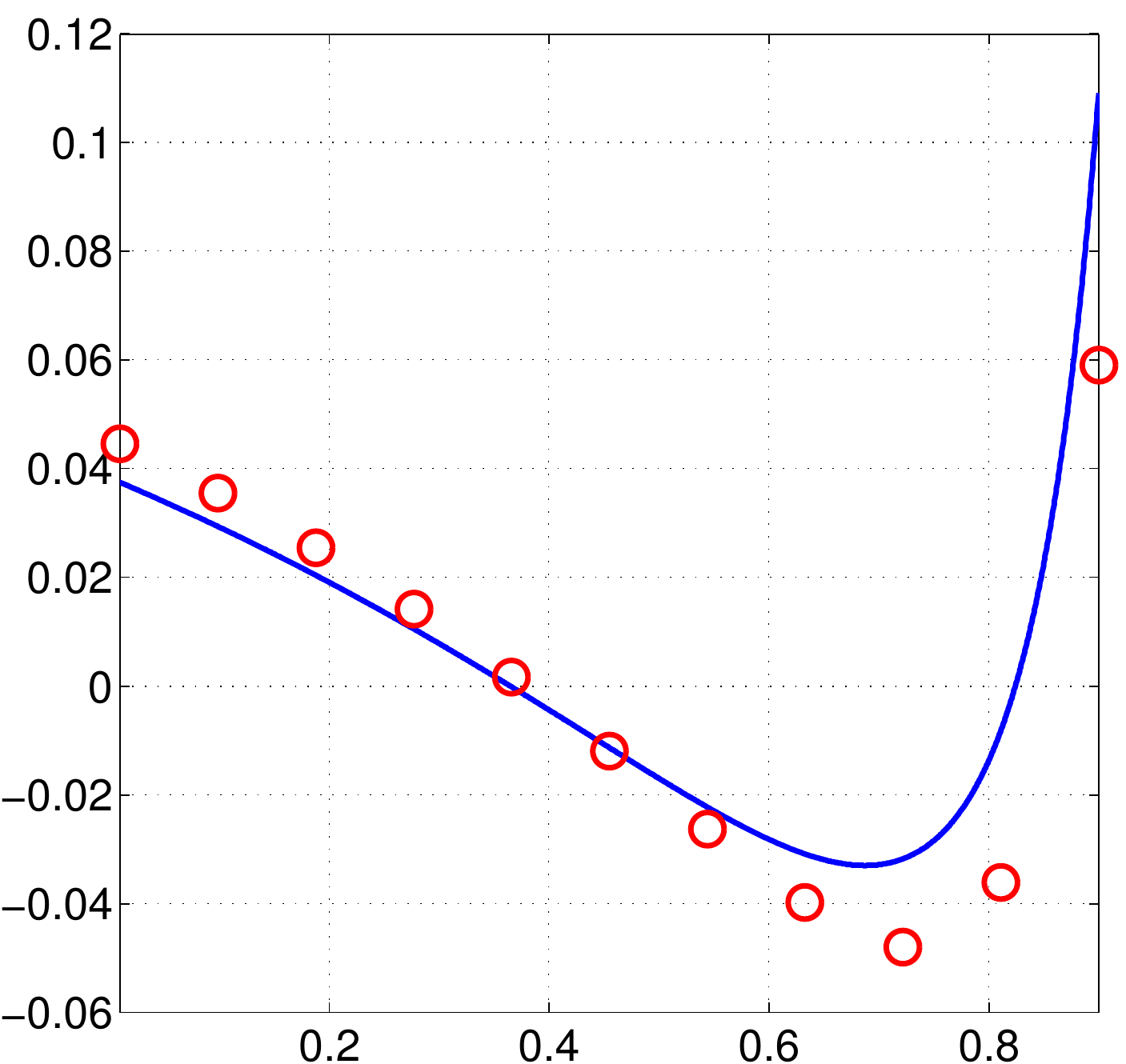}
}\;
\subfloat[$\sigma_4v_4(s)$]{
\includegraphics[width=0.3\linewidth]{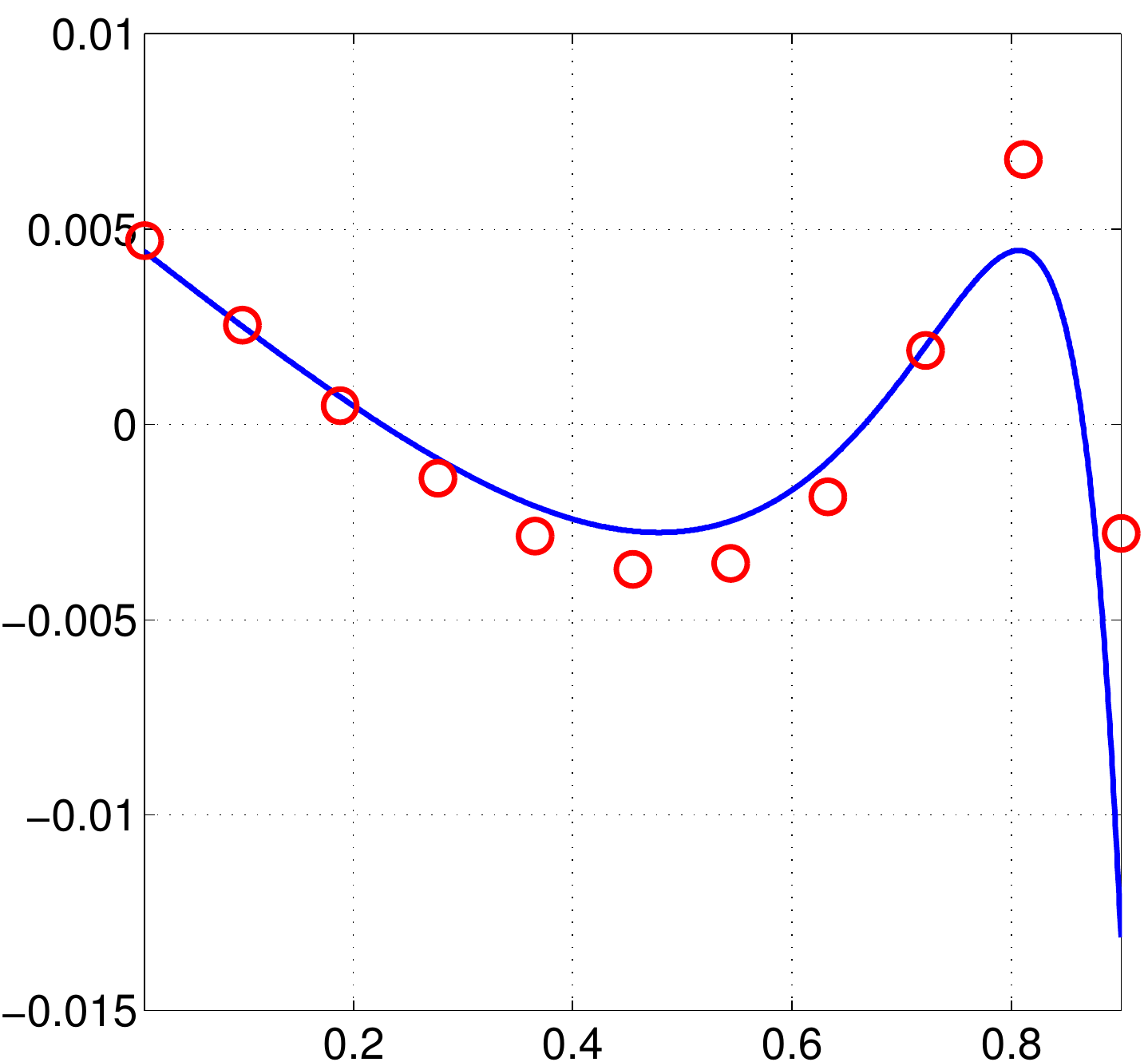}
}\\
\subfloat[$\sigma_5v_5(s)$]{
\includegraphics[width=0.3\linewidth]{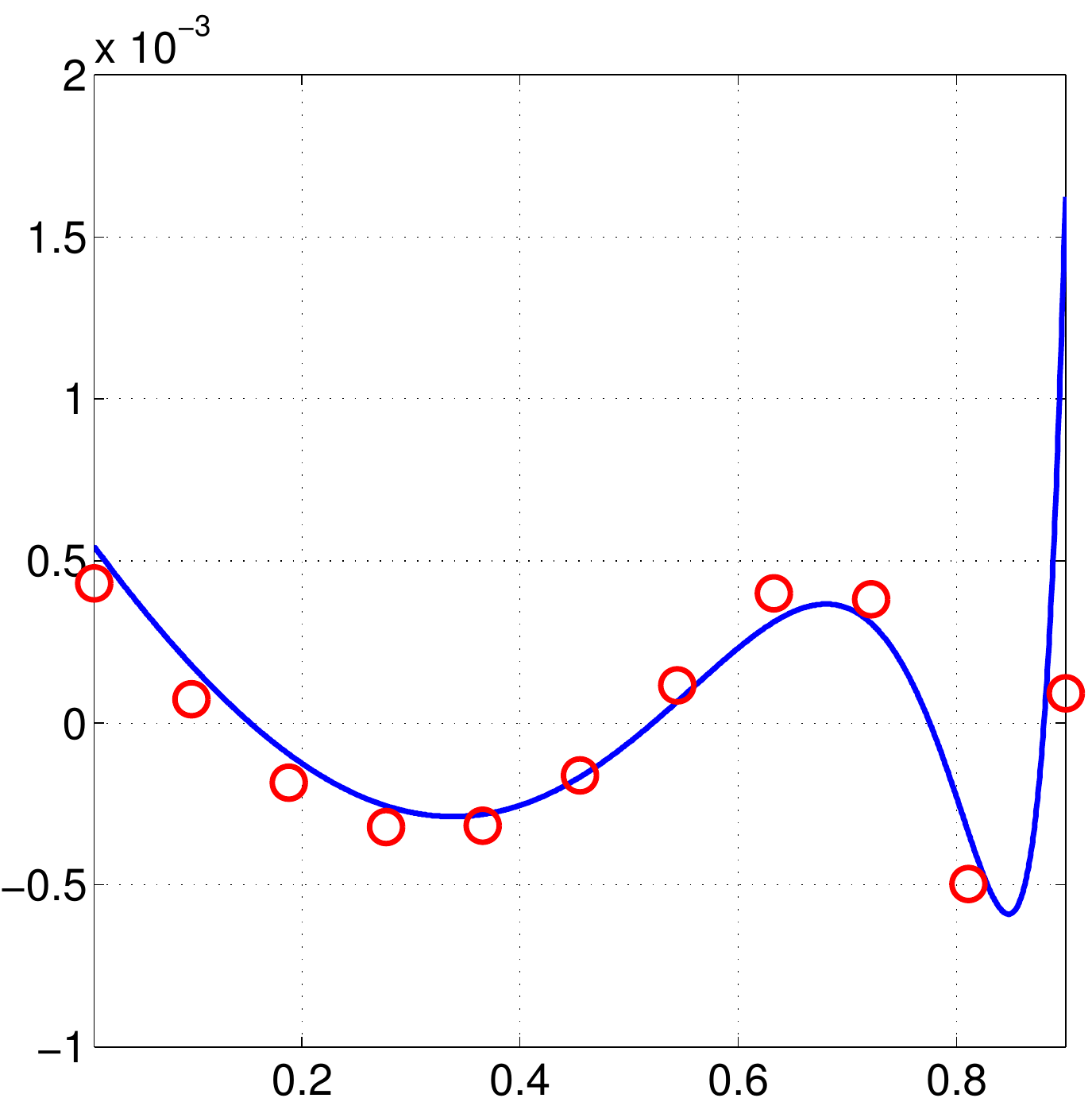}
}\;
\subfloat[$\sigma_6v_6(s)$]{
\includegraphics[width=0.3\linewidth]{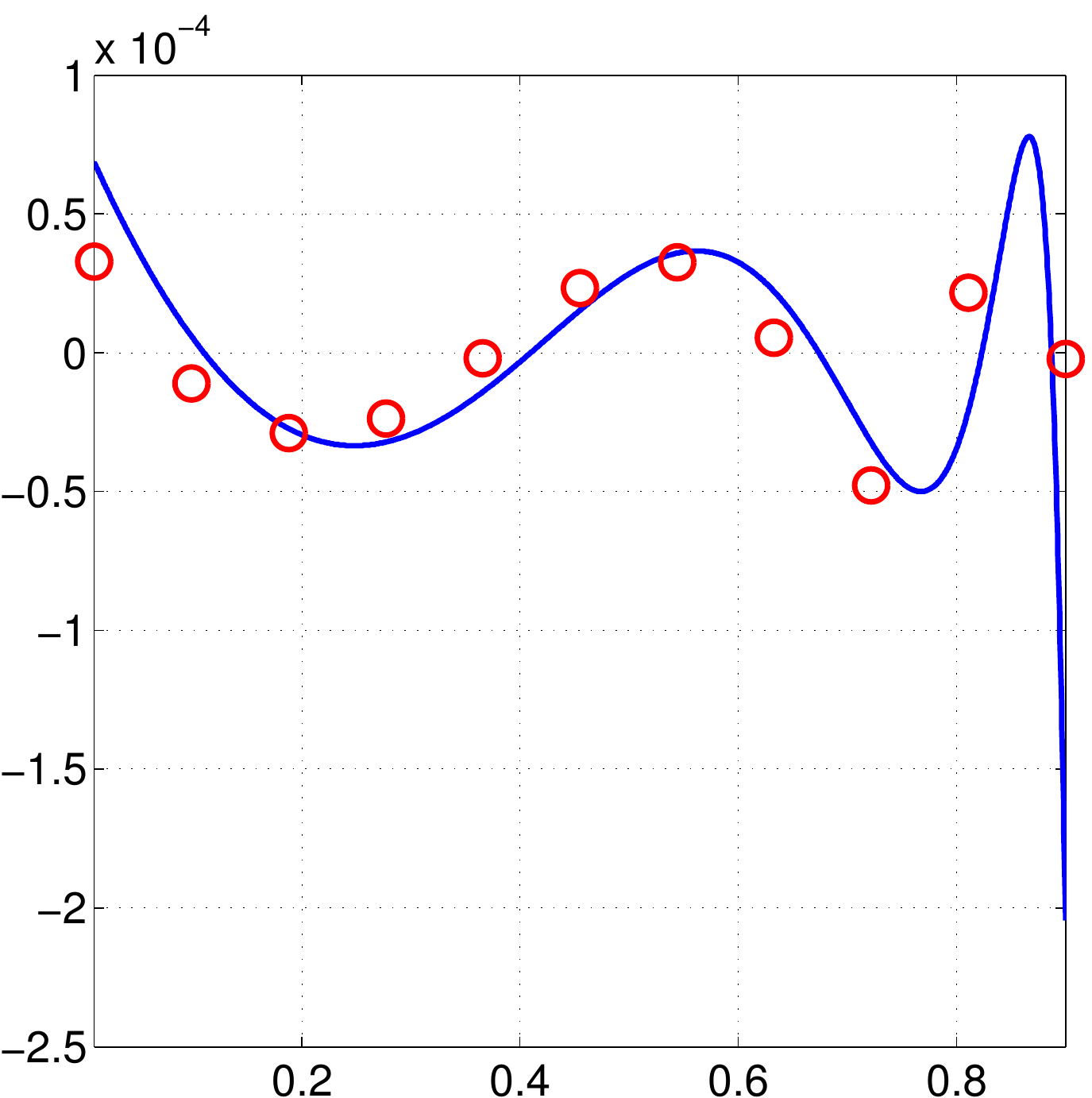}
}\;
\subfloat[$\sigma_7v_7(s)$]{
\includegraphics[width=0.3\linewidth]{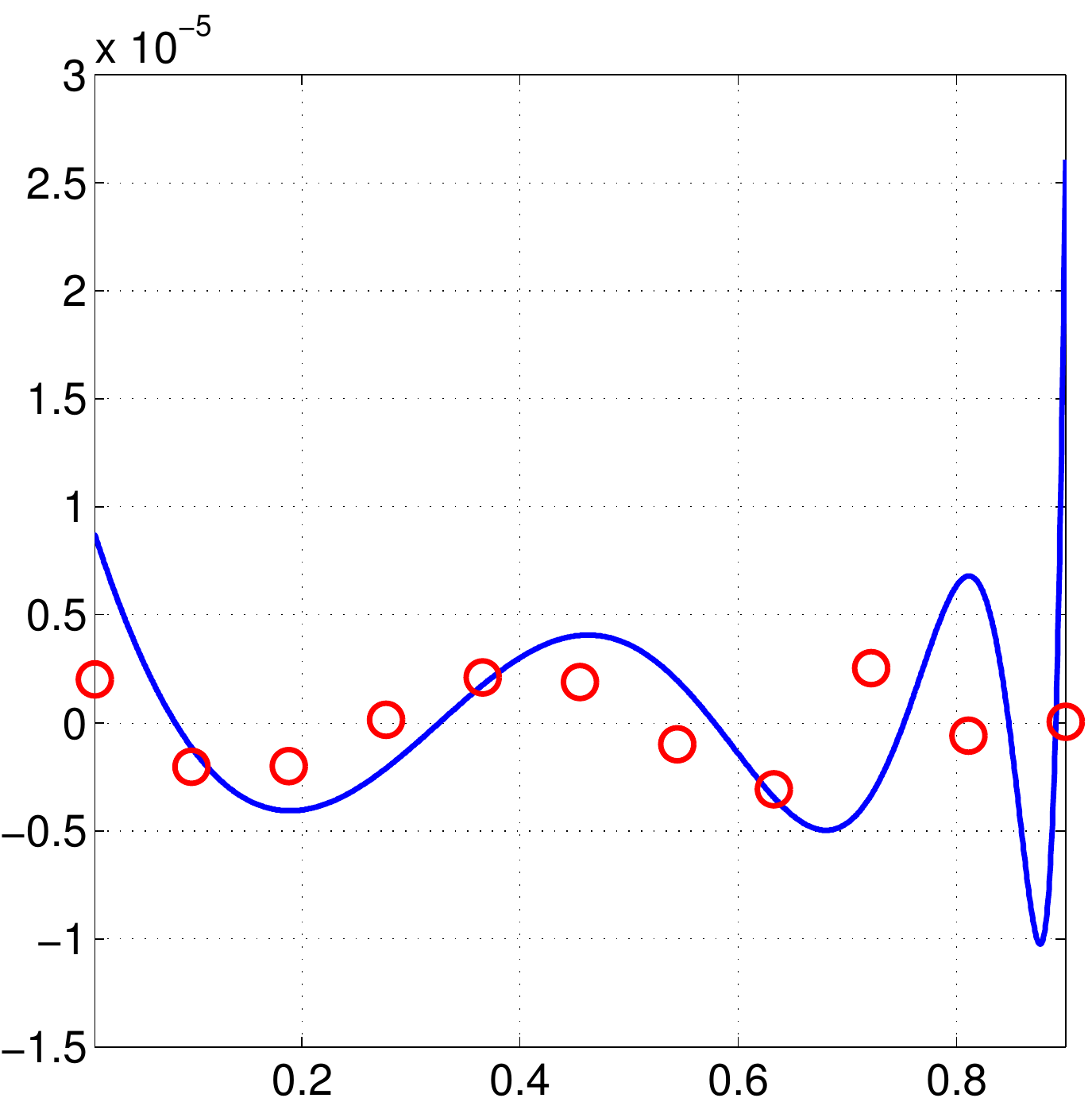}
}
\caption{The top left figure shows the solution $f(x,s)$ to the
  boundary value problem \eqref{eq:bvp}. The figure to its right shows the
  singular values of a finely resolved $\mF$ with 1999 columns (blue
  x's) and a coarsely resolved $\mF$ with 11 columns (red o's). Each
  are scaled by the maximum singular value from each set. The
  remaining figures show the approximations of the singular functions
  $v_1(s)$ through $v_7(s)$ scaled by the respective singular values
  for $\mF$ with 1999 columns (blue lines) and 11 columns (red
  o's). (Colors are visible in the electronic version.) }
\label{fig:svecs1}
\end{figure}

In the next sections, we will exploit the observation of non-uniformly
increasing oscillations in the right singular vectors to devise a
heuristic for the ROM.

\subsection{Constructing the reduced-order model}
Recall that the goal is to approximate $f(x,s)$ for some input $s$
that was not used to compute a training run. We can use the existence
of the SVE to justify the following approach.  Since we treat the
components of the right singular vectors $\mV$ as evaluations of the
singular functions $v_k(s)$, we can interpolate between the singular
vector components to approximate the singular functions at new values
of $s$. More precisely, define
\begin{equation}
\label{eq:interp}
\tilde{v}_k(s) \;=\; \sI(s;\,v_k(s_1),\dots,v_k(s_N))
\end{equation}
where $\sI$ is an interpolation operator that takes a value of $s$ and
the components of the singular vector as arguments. The form of the
interpolant may depend on the selection of the points $s_j$. For
example, if these points are the Chebyshev points or the nodes of a
Gaussian quadrature rule, then high order global polynomial
interpolation is possible. If the points are uniformly spaced, then
one may use piecewise polynomials or radial basis functions.

Unfortunately, the increasingly oscillatory character of the functions
$v_k(s)$ as $k$ increases combined with the fixed discretization $s_j$
causes concern for any chosen interpolation procedure as $k$
approaches $N$. In other words, the smoothness of $v_k(s)$ decreases
as $k$ increases, which diminishes confidence in the interpolation
accuracy.  Therefore, we seek to divide the right singular vectors
into two groups: those that are smooth enough to accurately
interpolate and those that are not. Specifically, we seek an $R=R(s)$
with $R\leq N$ such that for $k\leq R$ we have confidence in the
accuracy of the interpolant $\tilde{v}_k(s)$. We treat the remaining
interpolations with $k>R$ as unpredictable, and we model them with a
random variable. We will discuss the choice of $R$ in the next
section.

Given $R$, we model the PDE output at the space-time coordinate $x_i$
for the new parameter value $s$ as
\begin{equation}
\label{eq:rommodel}
\tilde{f}(x_i,s) \;=\; 
\sum_{k=1}^R \sigma_k\,u_k(x_i)\,\tilde{v}_k(s)
\;+\;  \sum_{k=R+1}^N \sigma_k\,u_k(x_i)\,\eta_{k-R},
\end{equation}
where $\eta_k$ are uncorrelated random variables with mean zero and
variance one; these represent the uncertainty in the interpolation
procedure for increasingly oscillatory functions. Under this
construction, the vector of values $f(x_i,s)$ is a random vector with
mean and covariance,
\begin{equation}
\label{eq:meancovar}
\begin{aligned}
\Exp{\tilde{f}(x_i,s)} &= \sum_{k=1}^R
\sigma_k\,u_k(x_i)\,\tilde{v}_k(s), \\ 
\Cov{\tilde{f}(x_i,s)}{\tilde{f}(x_j,s)}
&= \sum_{k=R+1}^N \sigma_k^2\,u_k(x_i)\,u_k(x_j).
\end{aligned}
\end{equation}
The reduced-order model we propose is the mean of this random vector,
\begin{equation}
\label{eq:gpmean}
f(x_i,s) \;\approx\; \Exp{\tilde{f}(x_i,s)}.
\end{equation}
The diagonal components of the covariance matrix provide a measure of
confidence for the reduced-order model at each $x_i$ similar to the
predication variance of a Gaussian process regression
model~\cite{Rasmussen2006}.

Next we show the reduced-order model is equivalent to applying the
interpolation procedure independently to the rows of a low rank
approximation of the matrix $\mF$. To set up the notation, partition
\begin{equation}
\mU=\bmat{\mU_1&\mU_2},\quad
\mSigma = \bmat{\mSigma_1 & \\ & \mSigma_2},\quad
\mV = \bmat{\mV_1 & \mV_2},
\end{equation}
where $\mU_1$, $\mSigma_1$, and $\mV_1$ contain $R$ columns. Then
\begin{equation}
\begin{aligned}
\mF &= \mU\mSigma\mV^T\\
&=\mU_1\mSigma_1\mV_1^T + \mU_2\mSigma_2\mV_2^T\\
&= \mF_1 + \mF_2\\
&= \bmat{
f^{(1)}(x_1,s_1) & \cdots & f^{(1)}(x_1,s_N) \\
\vdots & \ddots & \vdots \\
f^{(1)}(x_M,s_1) & \cdots & f^{(1)}(x_M,s_N)
} + \bmat{
f^{(2)}(x_1,s_1) & \cdots & f^{(2)}(x_1,s_N) \\
\vdots & \ddots & \vdots \\
f^{(2)}(x_M,s_1) & \cdots & f^{(2)}(x_M,s_N)
}.
\end{aligned}
\end{equation}
Then we have the following proposition.

\begin{prop}
\label{thm:interp}
If $\sI$ from \eqref{eq:interp} is a linear operation,  then 
\begin{equation}
\begin{aligned}
\Exp{\tilde{f}(x_i,s)} &=
\sI\left(s;\,f^{(1)}(x_i,s_1),\dots,f^{(1)}(x_i,s_N)\right),\\ 
\Cov{\tilde{f}(x_i,s)}{\tilde{f}(x_j,s)}
&= \sum_{k=R+1}^N f^{(2)}(x_i,s_k)\,f^{(2)}(x_j,s_k).
\end{aligned}
\end{equation}
\end{prop}

\begin{proof}
For a function $g=g(s)$ with evaluations $g(s_j)$ the linear
interpolation can be written
\begin{equation}
\sI(s;\,g(s_1),\dots,g(s_N)) \;=\; \sum_{j=1}^N w_j\,g(s_j)
\end{equation}
for some set of weights $w_j=w_j(s)$. Then,
\begin{equation}
\begin{aligned}
\Exp{\tilde{f}(x_i,s)} &= \sum_{k=1}^R \sigma_k\,u_k(x_i)\,\tilde{v}_k(s) \\
&= \sum_{k=1}^R \sigma_k\,u_k(x_i)\, \left(\sum_{j=1}^N w_j\,v_k(s_j)\right) \\
&= \sum_{j=1}^N w_j\,\left(\sum_{k=1}^R \sigma_k\,u_k(x_i)\,v_k(s_j)\right) \\
&= \sum_{j=1}^N w_j\, f^{(1)}(x_i,s_j) \\
&= \sI\left(s;\,f^{(1)}(x_i,s_1),\dots,f^{(1)}(x_i,s_N)\right),
\end{aligned}
\end{equation}
as required. The covariance expression is easily proved using the
linear algebra notation. Define the $N\times N$ matrix $\mC_{ij} =
\Cov{\tilde{f}(x_i,s)}{\tilde{f}(x_j,s)}$. Then by the orthogonality
of the columns of $\mV_2$,
\begin{equation}
\mC \;=\; \mU_2\mSigma_2^2\mU_2^T \;=\;
\mU_2\mSigma_2\mV_2^T\mV_2\mSigma_2^T\mU_2^T \;=\; \mF_2\mF_2^T.
\end{equation}
as required.
\end{proof}

\subsection{Choosing $R$}
\label{sec:chooseR}
We must still choose $R$ that determines the split between smooth and
non-smooth singular vectors. We will exploit the observation of the
oscillating singular vectors from Section \ref{sec:oscillations} and
make use of Assumption \ref{bigass}. We define the following
\emph{variation metric},
\begin{equation}
\label{eq:deftau}
\tau(r,s) \;=\; \sum_{k=1}^r \left| \frac{\mV_{j+1,k} - \mV_{j,k}}{\Delta s}
\right|,\qquad\mbox{ for $s_j \leq s < s_{j+1}$.}
\end{equation}
By Assumption \ref{bigass}, $\tau$ is an increasing function of $r$ up
to some $R=R(s)$. Loosely, if $\tau$ is too large, then we have
entered the range of $k$ where interpolations of $v_k(s_j)$ are not to
be trusted. We will quantify this with a threshold $\bar{\tau}$.
Given $\bar{\tau}$, we choose $R=R(s,\bar{\tau})$ to be the largest
$r$ such that $\tau(r,s)\leq \bar{\tau}$.

To determine the appropriate threshold $\bar{\tau}$, we use a set of
PDE evaluations $f_\ell =f(x,s_\ell)$ with $\ell=1,\dots,L$ for
testing, where $s_\ell$ is not in the training set (i.e.,
$s_\ell\not=s_j$ for any $\ell$ or $j$). We choose a set of candidate
thresholds $\bar{\tau}_m$.  For each testing models and each candidate
threshold, we compute the relative error
\begin{equation}
\label{eq:testerr}
\sE(s_\ell,\bar{\tau}_m) = \left[
\sum_{i=1}^M \left(
f(x_i,s_\ell) - \sum_{k=1}^R \sigma_k\,u_k(x_i)\,\tilde{v}_k(s_\ell)
\right)^2 \middle/ 
\sum_{i=1}^M
f(x_i,s_\ell)^2 
\right]^{1/2}
\end{equation}
where $R=R(s_\ell,\bar{\tau}_m)$. These errors can be visualized, and
the final threshold is chosen so that the error in the testing set is
relatively small.

We demonstrate this process using the boundary value problem from
\eqref{eq:bvp}. The training models consist of solutions computed at
eleven equally spaced values of the parameter $s$ in the range
$[0.1,0.9]$.  We compute a test model at the midpoint of each interval
$[s_j,s_{j+1}]$, where $s_j$ was used to compute the training
models. The range of the variation metric $\tau$ from
\eqref{eq:deftau} for these testing sites is roughly 0.1 to 42.8. We
choose 20 candidate thresholds $\bar{\tau}_m$ in this range and
compute the error in the reduced-order model at the testing sites (see
\eqref{eq:testerr}) for each candidate threshold. These errors are
displayed in Figure \ref{fig:error}. We want to choose the split $R$
such that the ROM uses the fewest left singular vectors with
the maximal accuracy; using fewer singular vectors reduces the 
computational work and produces a simpler model. 
The errors in Figure \ref{fig:error} show that
the reduced-order model is as accurate as possible for each testing
site after the fourth candidate threshold, which is roughly
$\bar{\tau}=6.84$.
For this threshold, Table \ref{tab:thresh} displays
the split $R$ between the right singular vectors that admit an
accurate interpolant and those that are modeled with a random variable
for each of the testing sites. Notice that the number of smooth right
singular vectors is smaller for testing sites near the boundary
$s=0.9$ of the parameter domain, which is precisely what we would have
expected.

\begin{figure}
\begin{minipage}[b]{0.49\textwidth}
  \centering
  \includegraphics[width=0.8\linewidth]{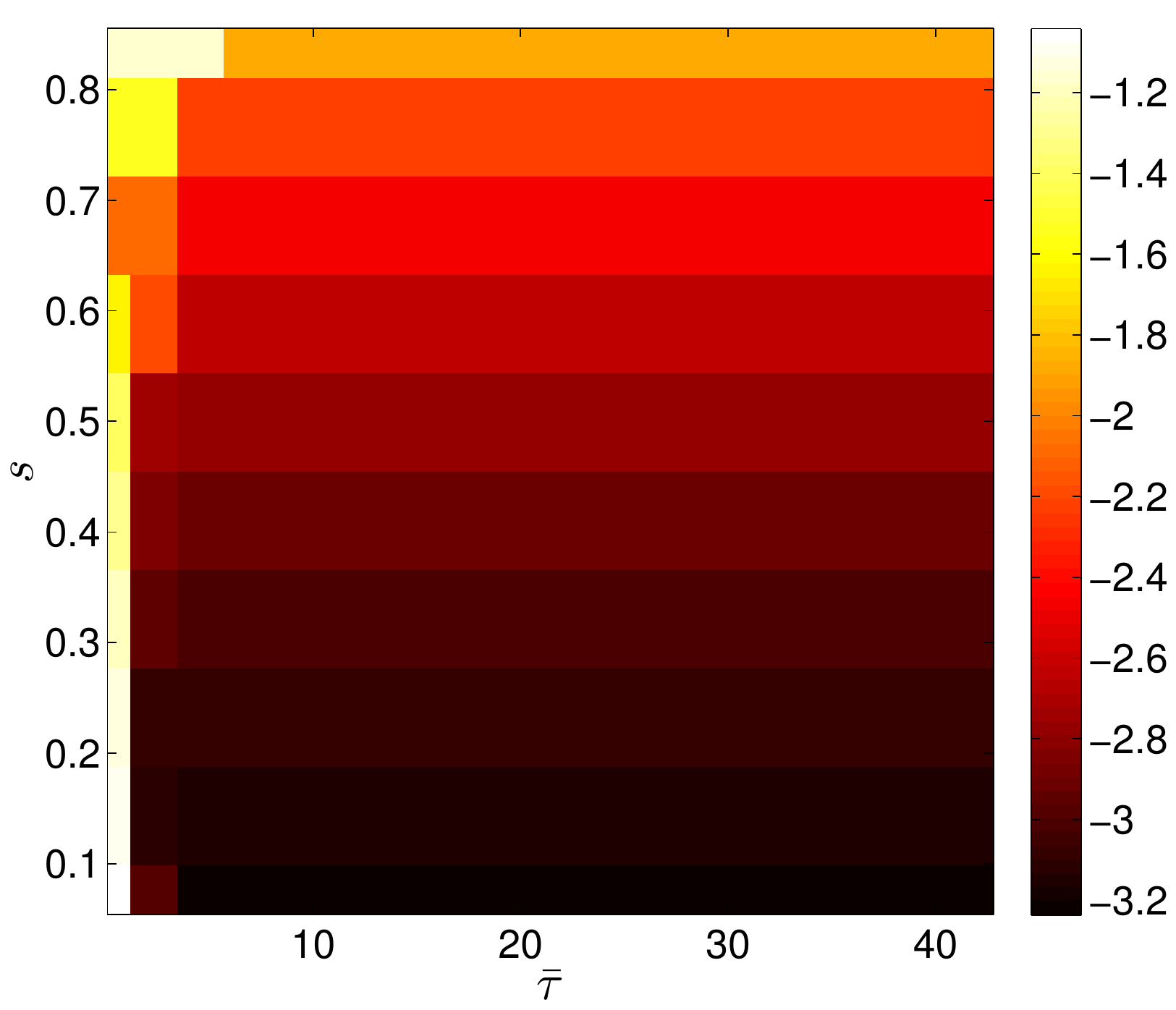}%
  \captionof{figure}{The log of the relative error in the mean prediction of the
    ROM as a function of $s$ and $\bar{\tau}$. (Colors are visible in the
    electronic version.)}
  \label{fig:error}
\end{minipage}
\hfill
\begin{minipage}[b]{0.49\textwidth}
\centering
\begin{tabular}{c|c|c}
  $s$ & $R(s,\bar{\tau})$ & $\sE(s,\bar{\tau})$ \\
  \hline
  0.0545 & 4 & 0.0006\\
  0.1435 & 5 & 0.0007\\
  0.2325 & 6 & 0.0008\\
  0.3215 & 6 & 0.0010\\
  0.4105 & 5  & 0.0013\\
  0.4995 & 6 & 0.0017\\
  0.5885 & 4 & 0.0023\\
  0.6775 & 5 & 0.0034\\
  0.7665 & 3 & 0.0059\\
  0.8555 & 2 & 0.0129
\end{tabular}
\captionof{table}{The split and the corresponding ROM error for $\bar{\tau}=6.84$
  and different values of $s$.}
\label{tab:thresh}
\end{minipage}
\end{figure}

We can compare the splitting strategy based on $R=R(s)$ with a
standard truncation strategy based on the magnitudes of the singular
values of $\mF$. The mean of the random vector \eqref{eq:gpmean} is
equivalent to interpolating a truncated SVD approximation of the data
matrix $\mF$, as shown in Proposition \ref{thm:interp}. However, the
magnitudes of the singular values provide no insight into the
uncertainty in the interpolation procedure. Our splitting strategy
chooses a different truncation for the mean \eqref{eq:gpmean} for each
$s$ based on the capability of the interpolation procedure to
accurately approximate the right singular functions $v_k(s)$ at the
point $s$. The singular values that are not in the mean contribute to
the covariance-based confidence measure from \eqref{eq:meancovar}. A
global truncation based on the singular values would create the same
prediction variance for every $s$, and it would always be on the order
of the largest truncated singular value. In other words, it provides
no information on how the confidence in the prediction changes as $s$
varies.

\section{Implementation in Hadoop}
\label{sec:implement}
Constructing the ROM for highly resolved simulations (i.e., large $M$)
requires significant data processing. We have implemented the
construction in the MapReduce framework, which enables us to take
advantage of Hadoop for large-scale distributed data processing.  To
construct the reduced-order model, the outputs from the high-fidelity
simulations are sent to and stored in the Hadoop cluster. Our
implementation then proceeds in three steps:
\begin{enumerate}
\item Create the tall-and-skinny matrix $\mF$ from the simulation data.
\item Compute the singular value decomposition of $\mF$.
\item Generate the coefficients of the reduced-order model and evaluate
  solutions from the ROM.
\end{enumerate}
In what follows, we give a very brief overview of the MapReduce
framework, and then we describe each step of the implementation in
Hadoop.

\subsection{MapReduce/Hadoop}
Google devised MapReduce because of the frustration programmers
experienced as they constantly juggled the complexity of developing
distributed, fault-tolerant data computational
algorithms~\cite{Dean2004-MapReduce}. Early data-intensive computing
at Google was a complex mix of ad hoc scripts. Their solution was the
MapReduce computation model: a simple, general interface for a common
template behind their data analysis tasks that hides the details of
the parallel implementations from the programmer.  Due to its
generality, the MapReduce computation model has also been an effective
paradigm for parallelizing tasks on GPUs~\cite{He-2008-MARS},
multi-core systems~\cite{Talbot-2011-pheonix++}, and traditional HPC
clusters~\cite{Plimpton-2011-MRMPI}.

The MapReduce model consists of two elements inspired by functional
programming: a \Map operation to transform the input into a key/value
pair and a \Reduce operation to process information with the same
key. The user provides both functions, which cannot have any side
effects. A MapReduce implementation executes the \Map function on the
entire dataset in parallel; see Figure~\ref{fig:mapreduce}. A
canonical MapReduce dataset is a terabyte-sized text file split into
individual lines. In this case, each \Map function receives only a few
thousand lines from the enormous file, which it processes and sends
via \Shuffle to the appropriate \Reduce function. The \Shuffle
operation groups \Map outputs by the key, and the \Reduce function
processes all outputs with the same key---e.g., counting the number of
times a word appears in a large collection of documents.

Google's implementation of MapReduce is proprietary. An alternative,
open source implementation named Hadoop has become the industry
standard for large-scale data processing. More information on Hadoop
can be found at the Cloudera website~\cite{Hadoop2010}. The Hadoop
Distributed File System (HDFS) is a fault-tolerant, replicated, block
file system designed to run using inexpensive, consumer grade hard
disk drives on a large set of nodes.

\begin{figure}[t]
\includegraphics[width=\linewidth]{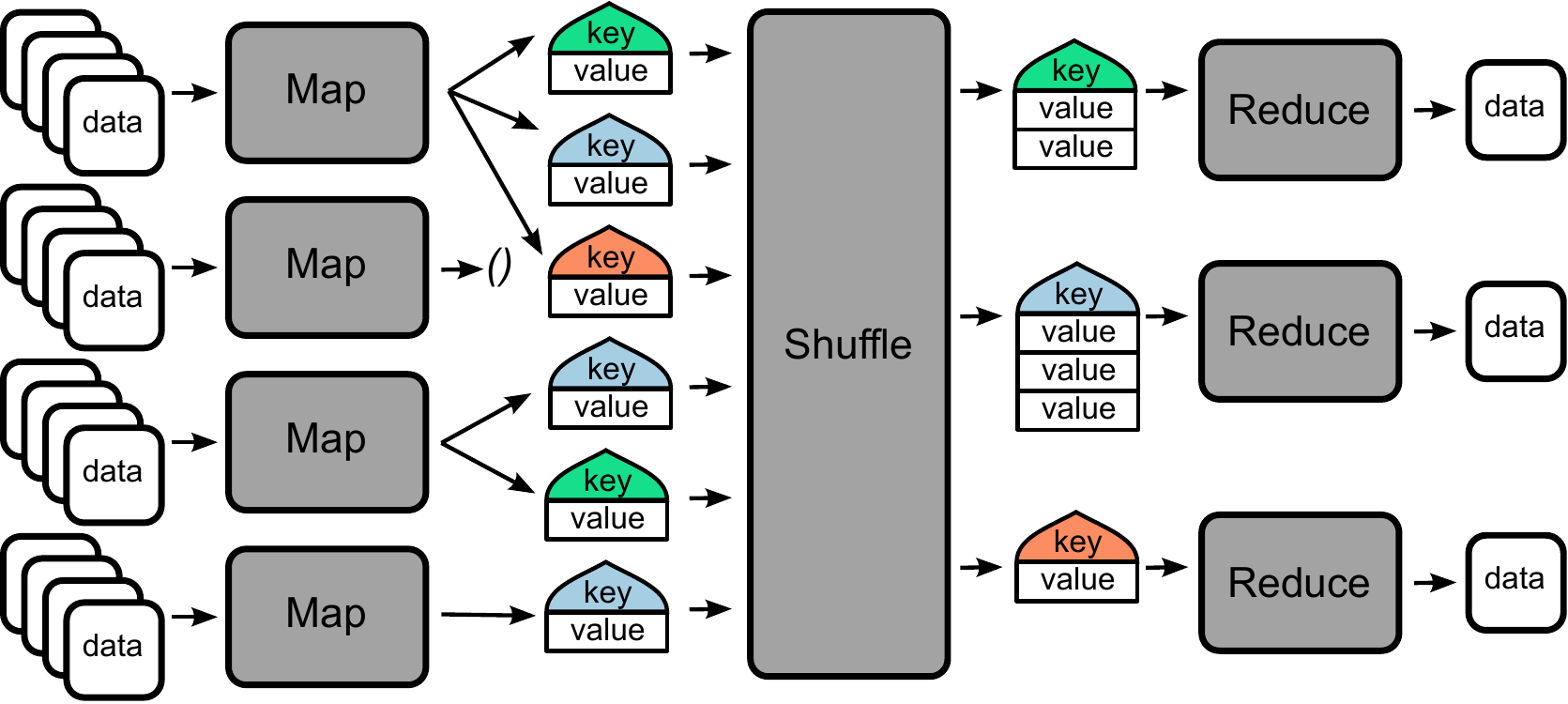}
\caption{The MapReduce system by Google is inspired by a functional
  programming paradigm and consists of three phases: transform data
  (\Map), aggregate results (\Shuffle), and compute and store
  (\Reduce). All \Map and \Reduce functions are independent, which
  allows the system to schedule them in parallel. }
\label{fig:mapreduce}
\end{figure}  

\subsection{Assembling the matrix from simulation outputs}

The first step in the construction of the ROM is to reorganize the
data into the tall-and-skinny matrix $\mF$.  This step is particularly
communication intensive as it requires reorganizing the data from the
columns (the natural outputs of the simulations) to rows for the
tall-and-skinny SVD routine. To do this in Hadoop we create a text
file where each line is a path to a file containing the outputs of one
simulation stored in HDFS. Then the \Map function reads simulation
data from HDFS and outputs the data keyed on the row of matrix
$\mF$. The \Reduce function aggregates all the entries in the row and
outputs the realized row.  The outcome of this first MapReduce
iteration is the matrix $\mF$ stored by rows on the distributed file
system.  More explicit descriptions of the functions are given in
Figure~\ref{fig:mr-assemble}

\begin{figure}
\begin{minipage}[t]{0.48\linewidth} \footnotesize
\textbf{Map}(\emph{key}$=$ simulation id, \emph{value}$=$ empty)\\
[1ex] Read simulation data based on the simulation id \\ 
Emit each point in a simulation as a record where the \emph{key} is the row of
the matrix $\mF$---constructed from the spatial location and time step---and 
the \emph{value} contains both the column id
in the matrix---given by the value of the parameter $s$---and the
value from the simulation data. \end{minipage}\hfill%
\begin{minipage}[t]{0.48\linewidth} \footnotesize
\textbf{Reduce}(\emph{key}$=$ row id, \emph{values}$= $\{column id, $F_{ij}$\})\\[1ex]
Read all of the values, and emit the combined row as a record where the \emph{key} is the row id and the \emph{value} is the array $\vf_i^T$. 
\end{minipage}
\caption{Map and Reduce functions to assemble the matrix $\mF$ from simulation data.}
\label{fig:mr-assemble}
\end{figure}

From a matrix perspective, each mapper processes a subset of the
entries of $\mF$. For example, assume the disjoint index sets $\Omega_1$,
$\Omega_2$, $\Omega_3$, and $\Omega_4$ contain the indices of $\mF$. 
Then the following diagram shows four mappers processing the simulation
data:
\[ 
 \underbrace{
  \begin{array}{l} 
    \{ F_{ij} \mid (i,j) \in \Omega_1 \} 
      \xrightarrow[\text{Map}]{} 
    \{ ( \underbrace{\;\,\, \vphantom{\mF_{ij}}i\;\,\, }_{\mathclap{\text{Key}}},   
         \underbrace{j, F_{ij}}_{\text{Value}} ) \mid (i,j) \in \Omega_1 \} \\
    \{ F_{ij} \mid (i,j) \in \Omega_2 \} 
      \xrightarrow[\text{Map}]{} 
    \{ ( \underbrace{\;\,\, \vphantom{\mF_{ij}}i\;\,\, }_{\mathclap{\text{Key}}},   
         \underbrace{j, F_{ij}}_{\text{Value}} ) \mid (i,j) \in \Omega_2 \} \\
    \{ F_{ij} \mid (i,j) \in \Omega_3 \} 
      \xrightarrow[\text{Map}]{} 
    \{ ( \underbrace{\;\,\, \vphantom{\mF_{ij}}i\;\,\, }_{\mathclap{\text{Key}}},   
         \underbrace{j, F_{ij}}_{\text{Value}} ) \mid (i,j) \in \Omega_3 \} \\
    \{ F_{ij} \mid (i,j) \in \Omega_4 \} 
      \xrightarrow[\text{Map}]{} 
    \{ ( \underbrace{\;\,\, \vphantom{\mF_{ij}}i\;\,\, }_{\mathclap{\text{Key}}},   
         \underbrace{j, F_{ij}}_{\text{Value}} ) \mid (i,j) \in \Omega_4 \} \\
   \end{array}         
}_{\text{Map Stage}} 
\quad
\underbrace{%
  \begin{array}{l} 
   (1, \{ F_{1,:} \} ) 
     \xrightarrow[\text{Red.}]{} 
   (1, \vf_1^T) \\
   (2, \{ F_{2,:} \} ) 
     \xrightarrow[\text{Red.}]{} 
   (2, \vf_2^T) \\
   \qquad \vdots \\
      (m, \{ F_{m,:} \} ) 
     \xrightarrow[\text{Red.}]{} 
   (m, \vf_m^T)
  \end{array}
}_{\text{Reduce Stage}} \] 
Hadoop assigns the reducers randomly to
nodes of the cluster. If we have four nodes, then the output from
those four reducers will all be stored together, 
\[ 
\mF = \bigl\{ \{ \mF_1 \} , \{ \mF_2 \} , \{ \mF_3 \} ,  \{ \mF_4 \} \bigr\}. 
\]
where each $\mF_i$ is a random subset of rows of the original
matrix. Each of these blocks $\mF_i$ also stores the id's of each row
it contains.


\subsection{TSQR and SVD in Hadoop}
\label{sec:tssvd}

Once we have the matrix $\mF$ stored by rows on disk, we compute its
tall-and-skinny QR (TSQR)
factorization~\cite{Benson-preprint-direct-tsqr}. The basis for the
MapReduce TSQR algorithm is the communication-avoiding QR
factorization~\cite{Demmel-2012-CAQR}. The strategy is to divide the
tall-and-skinny matrix $\mF$ into many smaller, tall matrices to be
decomposed independently via the QR factorization. This process is
repeated on the new matrix formed by all the $\mR$ factors computed in
the previous step until there is only a single $\mR$ left. This
algorithm was shown to have superior numerical stability to a large
Householder-style procedure~\cite{Mori-2012-allreduce}. The HDFS
stores the matrix in small chunks according to its internal splitting
procedure. Each \Map function reads a small submatrix and computes
a QR factorization. To record the association between
this QR factorization and all other QR factorizations computed
in the Map stage, we create a small tag $\phi$ that's a
universally unique identifier. The \Map function then writes the $\mQ$
factor back to disk with this small tag $\phi$.  Finally, it outputs
the $\mR$ factor and the same tag $\phi$ with key $0$. All map
functions output their $\mR$ factor with that same key. Because of
this, the outputs all go to the same reducer.
This single reducer was not a limitation for our applications, but a
recursive procedure in~\cite{Benson-preprint-direct-tsqr} is possible
if the \Reduce becomes burdensome on a single node.  The diagram below
and the description that follows demonstrates the procedure when $\mF$
is split into four blocks.
\[ 
 \underbrace{
  \begin{array}{l} 
    \{ \mF_1 \} \xrightarrow[\text{Map}]{\text{QR}} \{ (\phi_1, \mQ_1) \}, (\;\,\, 0 \,\,\; , \mR_1, \phi_1) \\
    \{ \mF_2 \} \xrightarrow[\text{Map}]{\text{QR}} \{ (\phi_2, \mQ_2) \}, (\;\,\, 0 \,\,\; , \mR_2, \phi_2) \\
    \{ \mF_3 \} \xrightarrow[\text{Map}]{\text{QR}} \{ (\phi_3, \mQ_3) \}, (\;\,\, 0 \,\,\; , \mR_3, \phi_3) \\
    \{ \mF_4 \} \xrightarrow[\text{Map}]{\text{QR}} \underbrace{\{ (\phi_4, \mQ_4) \}}_{\text{To Disk}}, (\underbrace{\;\,\, \vphantom{\mR_4 \phi_r}0 \;\,\, }%
    _{\mathclap{\text{Key}}}, \underbrace{\mR_4, \phi_4}_{\text{Value}})  \\
  \end{array}
}_{\text{Map Stage}} 
\quad
\underbrace{
  (0, \underbrace{\kbordermatrix{
  \\
  \phi_1 & \mR_1 \\
  \phi_2 & \mR_2 \\
  \phi_3 & \mR_3 \\
  \phi_4 & \mR_4 }}_{\text{Values}} )
  \xrightarrow[]{\text{QR}}
  \underbrace{\{ \mR \}}_{\text{Disk}}, \begin{array}{l}
    (\phi_1, \mQ_{1,1}) \\
    (\phi_2, \mQ_{2,1}) \\
    (\phi_3, \mQ_{3,1}) \\
    (\phi_4, \mQ_{4,1}) 
   \end{array}
}_{\text{Reduce Stage}} 
\]
There are two types of outputs represented in the diagram above: (i)
those surrounded by curly braces $\{ \ldots \}$ are written to disk
for processing in the future and (ii) those surrounded by parentheses
$( \ldots )$ are used in the next stage.  The tag $\phi_i$ uniquely
identifies each \Map function.  The output from this first MapReduce
job is the matrix $\mR$. The identifiers and factors $(\phi_i,
\mQ_{i,1})$ are fed to the next MapReduce job that computes the matrix
$\mU$ of left singular vectors.  To form $\mU$, we follow the R-SVD
procedure described in~\cite{chan1982improved}.  We compute the SVD of
the small matrix $\mR = \mU_R \mSigma \mV^T$ on one node and store
$\mSigma$ and $\mV$ on disk. With another
\Map and \Reduce, we distribute the matrix $\mU_R$ to all tasks and 
combine it with both
the input $(\phi_i, \mQ_i)$ and the stored output
$(\phi_i, \mQ_{i,1})$ from the last stage. The
shuffle moves all data with the same key to the same reducer, which
uses the tags $\phi_i$ to align the blocks of $\mQ_i$ computed in
the first map stage with those output in the first reduce stage. Then
in the reduce, we read all outputs with the same tag $\phi_i$ and
compute the products $\mQ_i \mQ_{i,1} \mU_R$ to get $\mU_i$. The
picture is:
\[ \underbrace{
  \text{Distribute } \mU_R
}_{\text{Launch phase}} \; 
\underbrace{
  \begin{array}{l}
   (\phi_1, \mQ_1) \xrightarrow[\text{Map}]{\text{Iden.}}  
     (\underbrace{\;\,\, \vphantom{\mR_4 \phi_r}\phi_1 \;\,\, }%
    _{\mathclap{\text{Key}}}, \underbrace{\mQ_1}_{\text{Value}}) \\
   (\phi_1, \mQ_{1,1}) \xrightarrow[\text{Map}]{\text{Iden.}}  
     (\underbrace{\;\,\, \vphantom{\mR_4 \phi_r}\phi_1 \;\,\, }%
    _{\mathclap{\text{Key}}}, \underbrace{\mQ_{1,1}}_{\text{Value}}) \\
    \qquad \vdots \\
      (\phi_4, \mQ_4) \xrightarrow[\text{Map}]{\text{Iden.}}  
     (\underbrace{\;\,\, \vphantom{\mR_4 \phi_r}\phi_1 \;\,\, }%
    _{\mathclap{\text{Key}}}, \underbrace{\mQ_4}_{\text{Value}}) \\
   (\phi_4, \mQ_{4,1}) \xrightarrow[\text{Map}]{\text{Iden.}}  
     (\underbrace{\;\,\, \vphantom{\mR_4 \phi_r}\phi_1 \;\,\, }%
    _{\mathclap{\text{Key}}}, \underbrace{\mQ_{4,1}}_{\text{Value}}) \\ 
  \end{array}
}_{\text{Map Stage}}
\; 
 \begin{array}{l}
    (\phi_1, \{ \mQ_1, \mQ_{1,1} \}) \xrightarrow[\text{Red.}]{}  \\
     \quad	\{ (\text{row ids}, \mQ_1 \mQ_{1,1} \mU_R) \} = \{ \mU_1 \}  \\
    (\phi_2, \{ \mQ_2, \mQ_{2,1} \}) \xrightarrow[\text{Red.}]{}  \\
     \quad  \{ (\text{row ids}, \mQ_2 \mQ_{2,1} \mU_R) \} = \{ \mU_2 \}  \\
    (\phi_3, \{ \mQ_3, \mQ_{3,1} \}) \xrightarrow[\text{Red.}]{}  \\
     \quad	\{ (\text{row ids}, \mQ_3 \mQ_{3,1} \mU_R) \} = \{ \mU_3 \}  \\
    (\phi_4, \{ \mQ_4, \mQ_{4,1} \}) \xrightarrow[\text{Red.}]{}  \\
     \quad  \{ (\text{row ids}, \mQ_4 \mQ_{4,1} \mU_R) \} = \{ \mU_4 \}  \\    
  \end{array}    	
 \]
This is a numerically stable computation of $\mU$ in the SVD of $\mF$
stored in HDFS. For more details about the \Map and \Reduce functions
see~\cite{Benson-preprint-direct-tsqr}.  The codes for computing the
TSQR and SVD can be found at \url{github.com/arbenson/mrtsqr}.

\subsection{Evaluating the reduced-order model in MapReduce}
\label{sec:mrrom}
Next we describe the procedure for evaluating the ROM in MapReduce for
a given parameter value $s$. If one needs to evaluate the ROM at many
values of $s$, this can be done in parallel with our existing codes.

There are two steps involved in evaluating the ROM. The first step is
evaluating the interpolated function $\tilde{v}_k(s)$ at $s$. Since
$\mV$ is small, this step is executed on a single
node.  The second step is estimating the ROM prediction and its
variance via
\begin{equation} \label{eq:f-and-var}
\Exp{\tilde{f}(x_i, s)} \; = \; \sum_{k=1}^{R} u_{k}(x_i) \,\sigma_k \,
\tilde{v}_k(s), 
\qquad 
\Var{\tilde{f}(x_i, s)} \; = \sum_{k=R + 1}^{N} \sigma_k^2 \,u_k(x_i)^2.
\end{equation}
Recall that $R=R(s)$ is the splitting of the singular value expansion
at the point $s$ described in \ref{sec:chooseR}. Further, recall, that
the matrix $\mU$ computed in the SVD of $\mF$ holds the coefficients
$u_k(x_i)$. We can evaluate the ROM at all points $x_i$ by computing
the matrix-vector product:
\begin{equation} 
\vf(s) = \mU(:,1\mathrm{:}R) \,\tilde{\vv}(s),
\end{equation}
where 
\begin{equation}
\vf(s) = \bmat{\Exp{\tilde{f}(x_1,s)} \\ \vdots \\ \Exp{\tilde{f}(x_M,s)}},
\qquad
\tilde{\vv}(s) = \bmat{ \sigma_1\tilde{v}_1(s) \\ \vdots \\ \sigma_R\tilde{v}_{R}(s)}.
\end{equation}
Since the matrices $\mF$ and $\mU$ are tall, we can compute such
matrix-vector products in Hadoop by distributing the small vector $\tilde{\vv}$
to a \Map function that will compute a subset of the entries of the matrix-vector
product. A subsequent \Reduce collects each submatrix-vector
product into a single file. Viewed schematically for $\mU$ stored in four blocks,
\begin{equation} 
\underbrace{
  \text{Distribute } \tilde{\vv}(s)
}_{\text{Launch phase}}
  \quad
\underbrace{  
  \begin{array}{l} 
    \{ \mU_1 \} \xrightarrow[\text{Map}]{\text{}} 
      ( \makebox[2em][c]{$s$}, \vf_1(s),  \Var{\vf_1(s)}  )  \\
    \{ \mU_2 \} \xrightarrow[\text{Map}]{\text{}} 
      ( \makebox[2em][c]{$s$}, \vf_2(s),  \Var{\vf_2(s)}  )  \\
    \{ \mU_3 \} \xrightarrow[\text{Map}]{\text{}} 
      (\makebox[2em][c]{$s$}, \vf_3(s),  \Var{\vf_3(s)}  )  \\
    \{ \mU_4 \} \xrightarrow[\text{Map}]{\text{}} 
      ( \underbrace{\makebox[2em][c]{$s$}}_{\mathclap{\text{Key}}}, 
        \underbrace{\vf_4(s),  \Var{\vf_4(s)}}_{\text{Value}} )  \\
  \end{array}
}_{\text{Map Stage}} 
\end{equation}
\begin{equation}
\underbrace{
\begin{array}{c} 
  \Bigg( s, 
     \begin{array}{ll} 
      \vf_1(s), \vf_2(s), \vf_3(s), \vf_4(s), \\
      \Var{\vf_1(s)}, \Var{\vf_2(s)}, \\ \Var{\vf_3(s)} , \Var{\vf_4(s)} 
      \end{array}
    \Bigg) \xrightarrow[\text{Reduce}]{\text{Join}} (\vf(s), \Var{\vf(s)} )
\end{array}    
}_{\text{Reduce Stage}}
\end{equation}
Although we illustrate this function with a single interpolation point
$s$, which results in a single reduce with one key, our
implementations are designed to handle around one-to-two thousand points $s$
simultaneously. Thus, we actually distribute $s$ and $\tilde{\vv}(s)$
for all of the one-to-two thousand points simultaneously. In this case, we would
have one reducer for each value $s$. 
Our codes for manipulating the simulations and performing the
interpolation are found at \url{github.com/dgleich/simform} in the
branch \texttt{simform-sisc}.

\begin{figure}
\begin{minipage}[t]{0.48\linewidth}\footnotesize%
\textbf{Map}(\emph{key}$=$ row id, \emph{value}$=$ $\vu_i^T$) \\ For
each $s$ and $\tilde{\vv}(s)$ that were distributed, emit the index of
the value of $s$ as the \emph{key} and the value as $\vu_i^T
\tilde{\vs}$ and $\Var{\vu_i^T \tilde{\vv}(s)}$ based on
equation~\eqref{eq:f-and-var}.
\end{minipage}
\hfill
\begin{minipage}[t]{0.48\linewidth}\footnotesize%
\textbf{Reduce}(\emph{key}$=$ $s$, \emph{values}$=$ ROM
evaluation)\\ Assemble the ROM predictions and confidence measure into
a single output and store that on disk.
\end{minipage}
\caption{Map and Reduce functions to compute our interpolants}
\end{figure}


\section{Numerical experiment}
\label{sec:exp}
In this section we apply the model reduction method to a parameter
study with a large-scale heat transfer model in random heterogeneous
media. In what follows, we describe the physical model, the parameter
study, and the construction of the ROM. We compare the ROM's
predictions for the quantity of interest with a standard response
surface approach. We close this section with some remarks on the
computational issues we encountered working with four terabytes of
data from the simulations.

\subsection{Heat transfer model}
We consider a partial differential equation model for bulk conductive
heat transfer of a temperature field $T=T(x,t)$,
\begin{equation}
\frac{\partial}{\partial t} (\rho c_p T) \;=\; \nabla\cdot(\kappa \,\nabla T), 
\qquad (x,y,z)\in\mathcal{D},\quad t\in [0,t_f].
\end{equation}
The spatial domain $\mathcal{D}$ is a rectangular brick of size
$20\times 10\times 10\,\mathrm{cm}^3$ centered at the origin.
The brick contains two materials roughly corresponding to foam
($\rho=319\,\mathrm{kg/m}^{3}$) and stainless steel
($\rho=7900\,\mathrm{kg/m}^3$) chosen for their contrast in thermal
conductivity. The values for temperature-dependent specific heat $c_p$
and conductivity $\kappa$ are shown in Tables \ref{t:cap} and
\ref{t:cond} in the appendix. The distribution of these two materials
within the brick is realized by the following procedure: (i) set the
material in the entire brick to steel, (ii) choose 128 locations
within the brick uniformly at random, (iii) for each location, find
all points in the brick within a given radius $s$ and set the material
to foam. The result is a steel brick with randomly distributed,
potentially overlapping foam bubbles of radius $s$.

A given brick begins ($t=0$) at room temperature
$T=298\,^{\circ}\mathrm{K}$ with a Dirichlet boundary condition of
$1000\,^{\circ}\mathrm{K}$ on the face $x=10\,\mathrm{cm}$. The
temperature field advances to the final time $t_f=2000\,\mathrm{s}$,
and the quantities of interest are measured at the face opposite the
prescribed boundary condition (the \emph{far face}) at
$x=-10\,\mathrm{cm}$. We are interested in two quantities: (i) the
average temperature on the far face and (ii) the proportion of the
temperature on the far face that exceeds $475\,^{\circ}\mathrm{K}$.

The finite element simulation uses Sandia Labs' SIERRA Multimechanics
Module: Aria~\cite{sand07-2734} with a regular mesh of $256\times
128\times 128$ elements constructed with CUBIT~\cite{Blacker1994}. One
simulation takes approximately four hours on a 8-core node of Sandia's
Red Sky capacity cluster (Dual Xeon 5500-series, 2.93 GHz, 2GB per
core).  The simulation outputs contain the full temperature field at
nine uniformly spaced times between $t=1200\,\mathrm{s}$ and
$t_f=2000\,\mathrm{s}$, which are stored in the Exodus II binary file
format. Each simulation output file is approximately 500MB. Two
representative temperature fields are shown in Figure \ref{fig:brick}.

\begin{figure}[t]
\centering
\subfloat[]{
\includegraphics[width=0.45\linewidth]{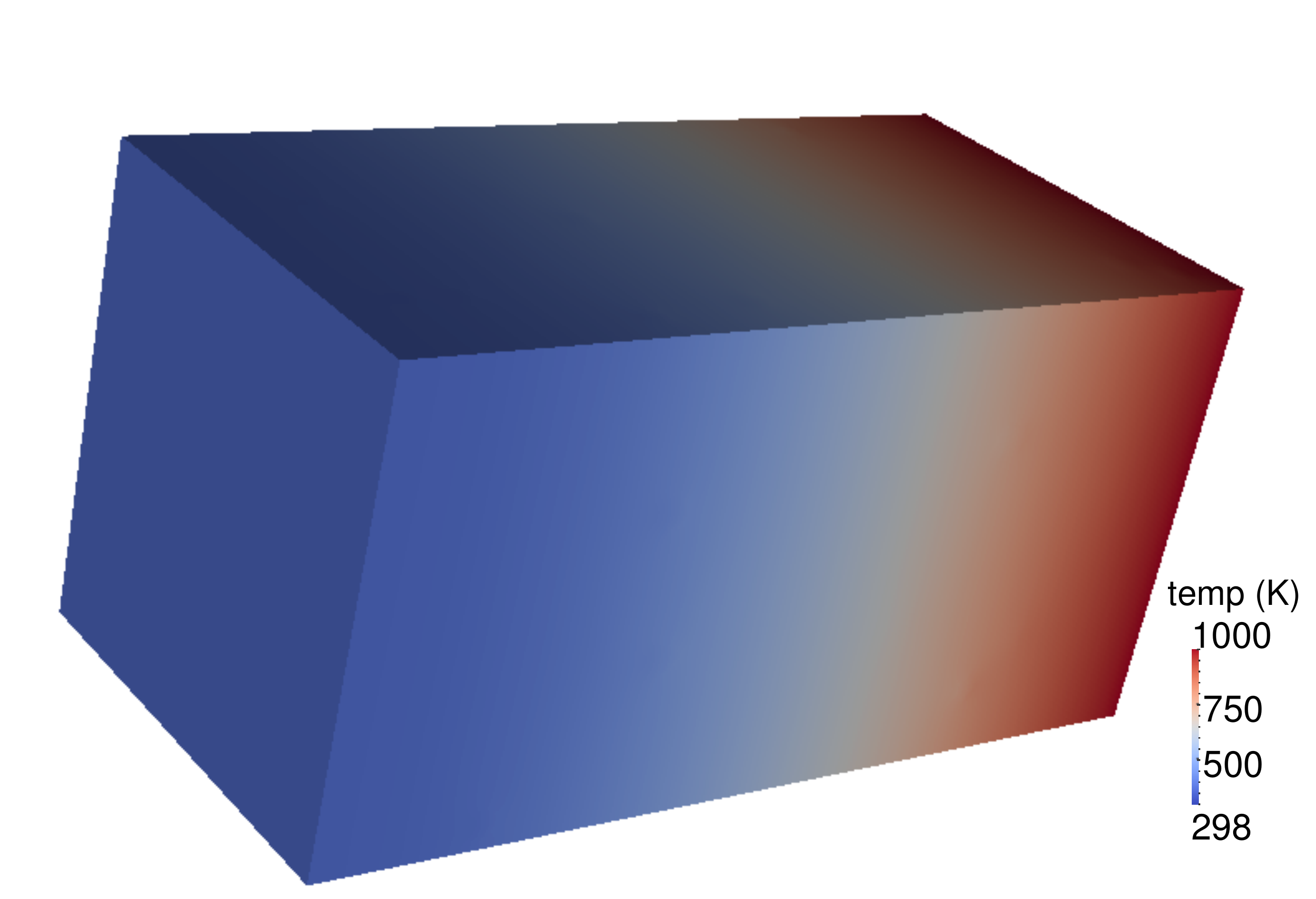}
}\;
\subfloat[]{
\includegraphics[width=0.45\linewidth]{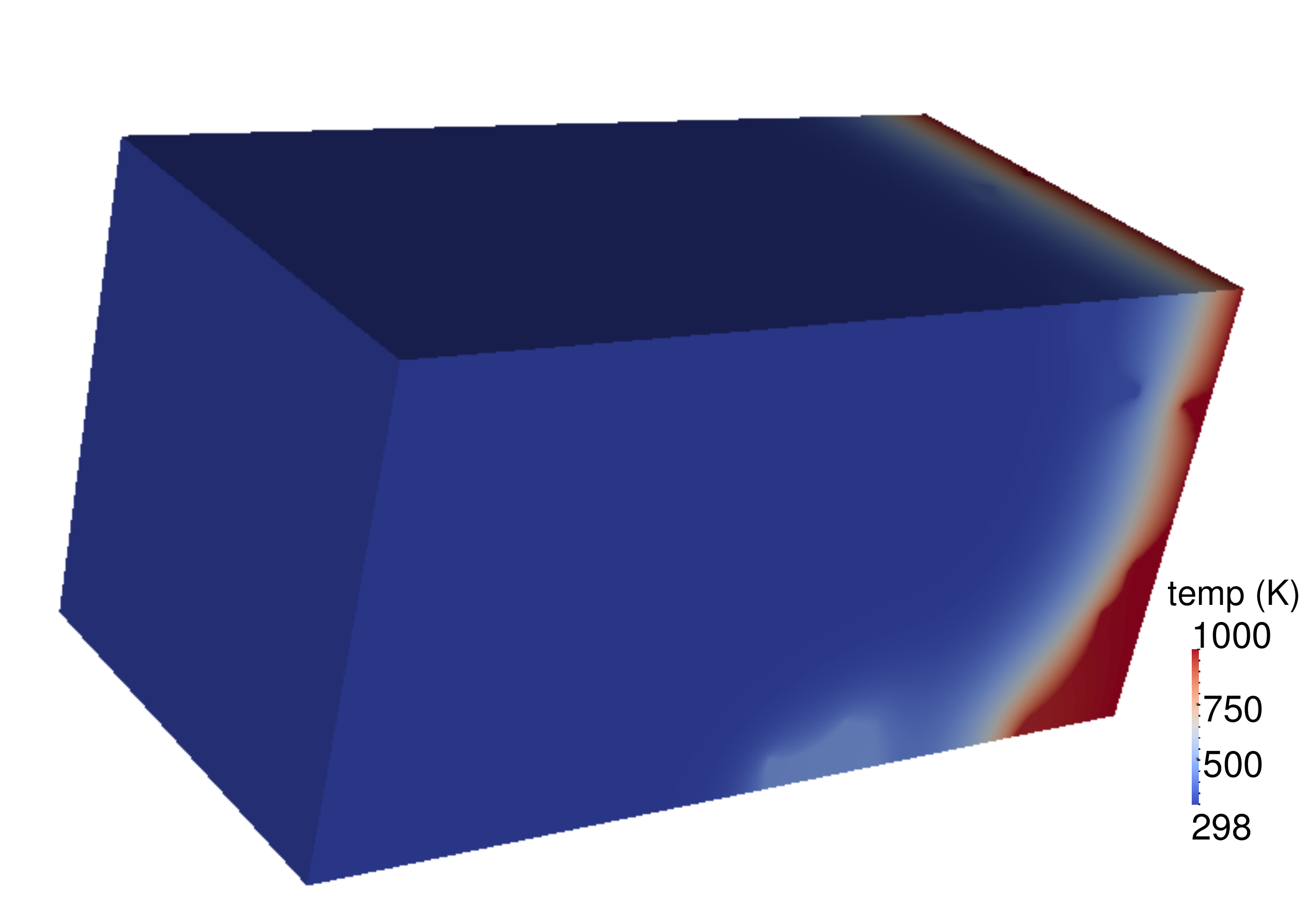}
}
\caption{Representative temperature fields at time
  $t_f=2000\,\mathrm{s}$ for bubble radius $s=0.390\,\mathrm{cm}$
  (left) and $s=2.418\,\mathrm{cm}$ (right). Distortions in the
  temperature field due to the presence of bubbles are clearly visible
  in the latter. (Colors are visible in the electronic version.) }
\label{fig:brick}
\end{figure} 

\subsection{Parameter study}
\label{sec:uqstudy}
We use the heat transfer model to study the effects of the bubble
radius parameter $s$ on the temperature distribution on the far face
via the quantities of interest. Intuitively, as $s$ increases, more of
the brick becomes foam, and we expect a lower temperature on the far
face $x=-10\,\mathrm{cm}$ due to foam's lower conductivity.

To address the variability in random media, we choose 128 random
realizations of the 128 locations for the bubble centers. For a given
radius $s$, we run 128 simulations---one for each realization of the
bubble locations. We use these simulations to compute Monte Carlo
estimates of the mean of each quantity of interest.  Note that there
are 384 random variables (three components per location)
characterizing the locations of the bubbles, so Monte Carlo estimates
are the only feasible option.

For each realization of the bubble locations, we run 64 simulations
varying $s$ uniformly between 0.039 cm and 2.496 cm. This results in a
total of 8192 simulations---128 bubble realizations $\times$ 64 values
for $s$. For each simulation, we compute the two quantities of
interest: (i) the average temperature over the far face, and (ii) the
proportion of the far face temperature that exceeds
$475\,^{\circ}\mathrm{K}$. We then approximate the mean over the
bubble realizations with Monte Carlo. Finally, we plot the estimates
of the mean as a function of the radius $s$. With 128
realizations, the 95\% confidence intervals are within 1\% of the
mean, so we do not plot them. These results are shown in Figure
\ref{fig:uq}. As expected, the propagation of the temperature
decreases as $s$ increases. However, the decrease is qualitatively
different for the two quantities of interest: the average far face
temperature decreases smoothly as a function of $s$ while the
proportion of the temperature above a threshold has a dramatic change
as a function of $s$. In the next section, we test the ability of the
reduced-order model to reproduce these results.

\begin{figure}[t]
\centering
\subfloat[]{
\includegraphics[width=0.45\linewidth]{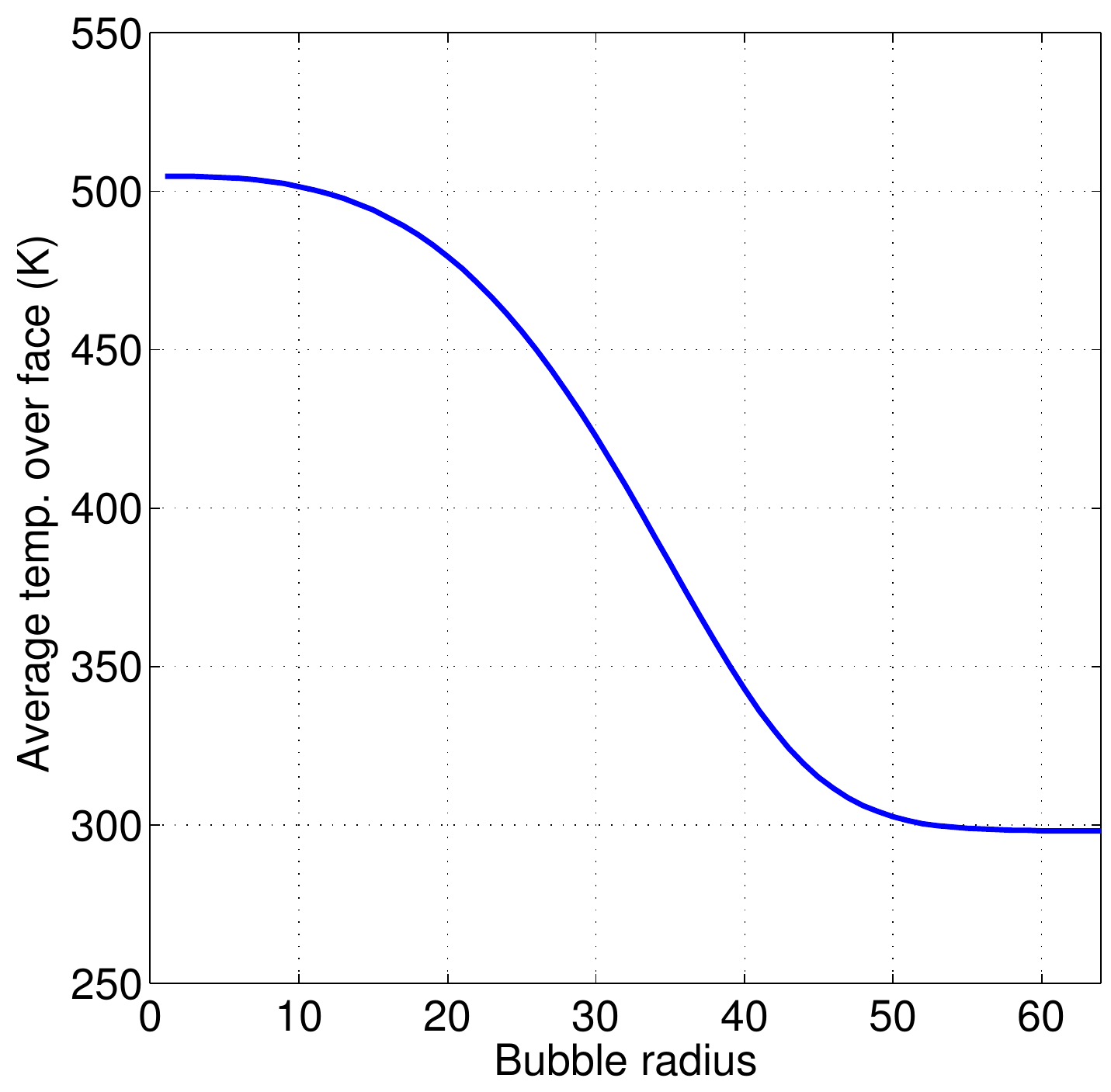}\label{fig:meantemp}
}\;
\subfloat[]{
\includegraphics[width=0.45\linewidth]{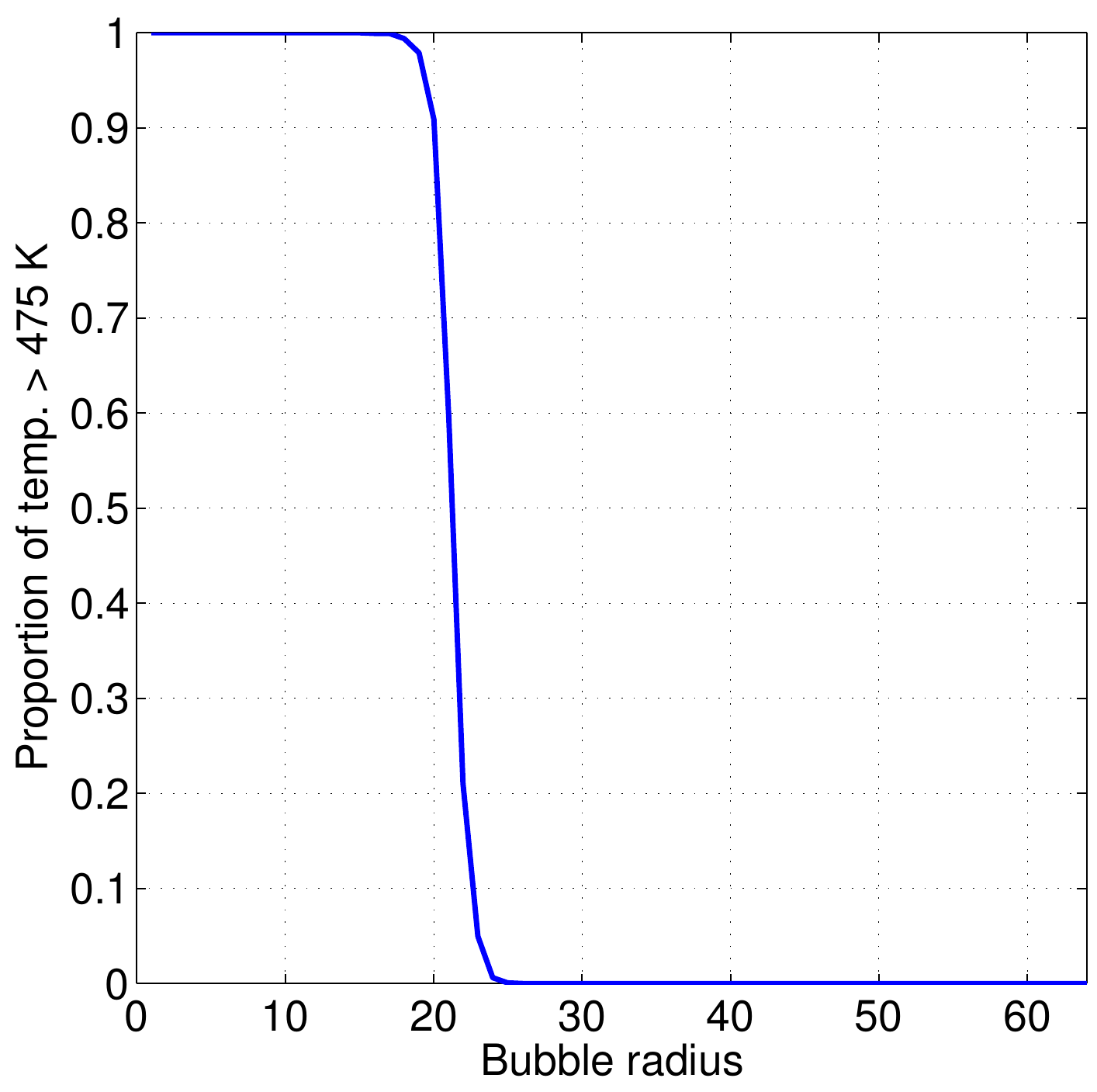}\label{fig:proptemp}
}
\caption{Monte Carlo estimates of the mean of the average far face
  temperature as a function of the bubble radius (left). The mean of
  the proportion of the temperature that exceeds $475\,^{\circ}$ K as
  a function of the bubble radius. (Colors are visible in the
  electronic version.) }
\label{fig:uq}
\end{figure} 

\subsection{Approximating the SVE}
Before testing the ROM, we use all of the simulation data to study the
components of the SVE of the parameterized temperature field. In the
notation of Section \ref{sec:rom}, $f=f(x,s)$ now carries the
following interpretation: $f$ is the temperature; $x$ contains the
three spatial coordinates, the time coordinate, and an identifier for
the realization of the random bubble locations; and $s$ is the bubble
radius. The matrix $\mF$ from \eqref{eq:fmat} has 64 columns
corresponding to the 64 values of the bubble radius.  Each column
contains 128 independent simulations---one for each realization of the
bubble locations. This translates to 4926801024 rows (257 $x$ nodes
$\times$ 129 $y$ nodes $\times$ 129 $z$ nodes $\times$ 9 times
$\times$ 128 bubble locations), which is approximately 2.3 terabytes
of data.

The singular values normalized by the largest singular value and the
first eight right singular vectors scaled by the singular values are
shown in Figure \ref{fig:tempsvd} with the blue x's. The rapid decay
of the singular values indicates the tremendous correlation amongst
components of the temperature fields as the radius varies. More
importantly, we see the more rapid increase in the oscillations of the
right singular vectors for larger values of $s$. This reflects the
fact that temperature fields with similar bubble radii have larger
differences for larger values of the radius $s$. This also means we
expect a more accurate ROM for smaller values of $s$.

We computed this SVD and all other MapReduce-based analyses of this
dataset on a 10-node Hadoop cluster at Stanford University's Institute
for Computational and Mathematical Engineering. Each node in the
cluster has 6 2TB hard drives, one Intel Core i7-960, and 24 GB of
RAM. The 2.3 TB matrix took approximately 12 hours. Below, we discuss
the time required for additional pre- and post-processing work.

\subsection{Contruction and validation of the ROM}
\label{sec:romexample}
We use a subset of the simulations as the training set for the
ROM. In particular, we choose $s_j = 0.039\,j$ for $j=1,5,9,\dots,61$
as the values of $s$ whose corresponding simulations are used to build
the ROM. Thus, the matrix $\mF$ for constructing the ROM contains 16
columns and the same roughly five billion rows, i.e., the same
locations in the domain of space, time, and the locations of the
bubble centers. This matrix contains approximately 600 GB of data, and
the SVD step took approximately 8 hours on the Hadoop cluster.
The singular values normalized by the largest singular value and the
components of the first eight right singular vectors scaled by their
respective singular values of $\mF$ are plotted in Figure
\ref{fig:tempsvd} with red o's.

\begin{figure}[ht]
\centering
\subfloat[Singular values]{
\includegraphics[width=0.3\linewidth]{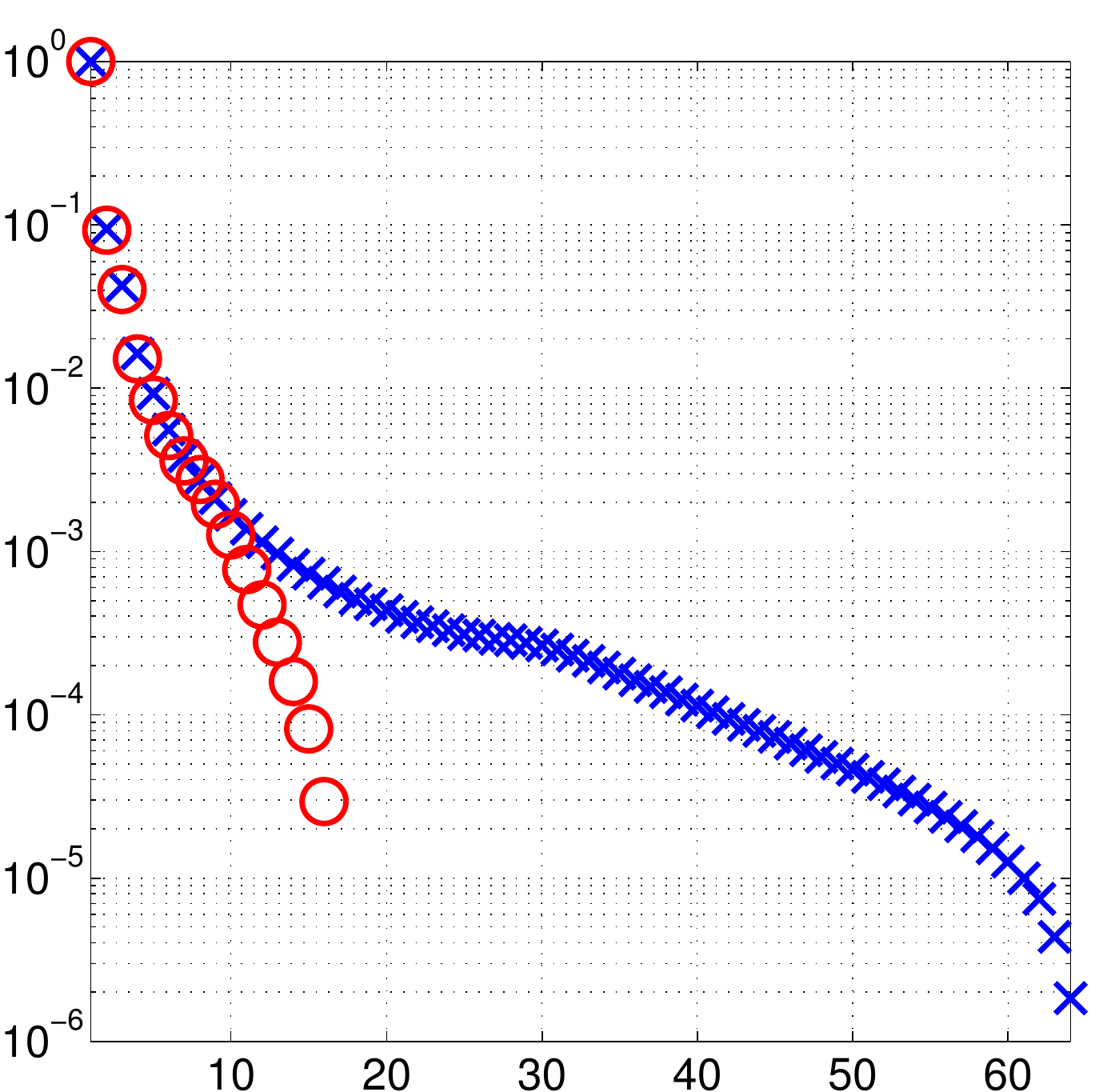}
}\;
\subfloat[$\sigma_1 v_1(s)$]{
\includegraphics[width=0.3\linewidth]{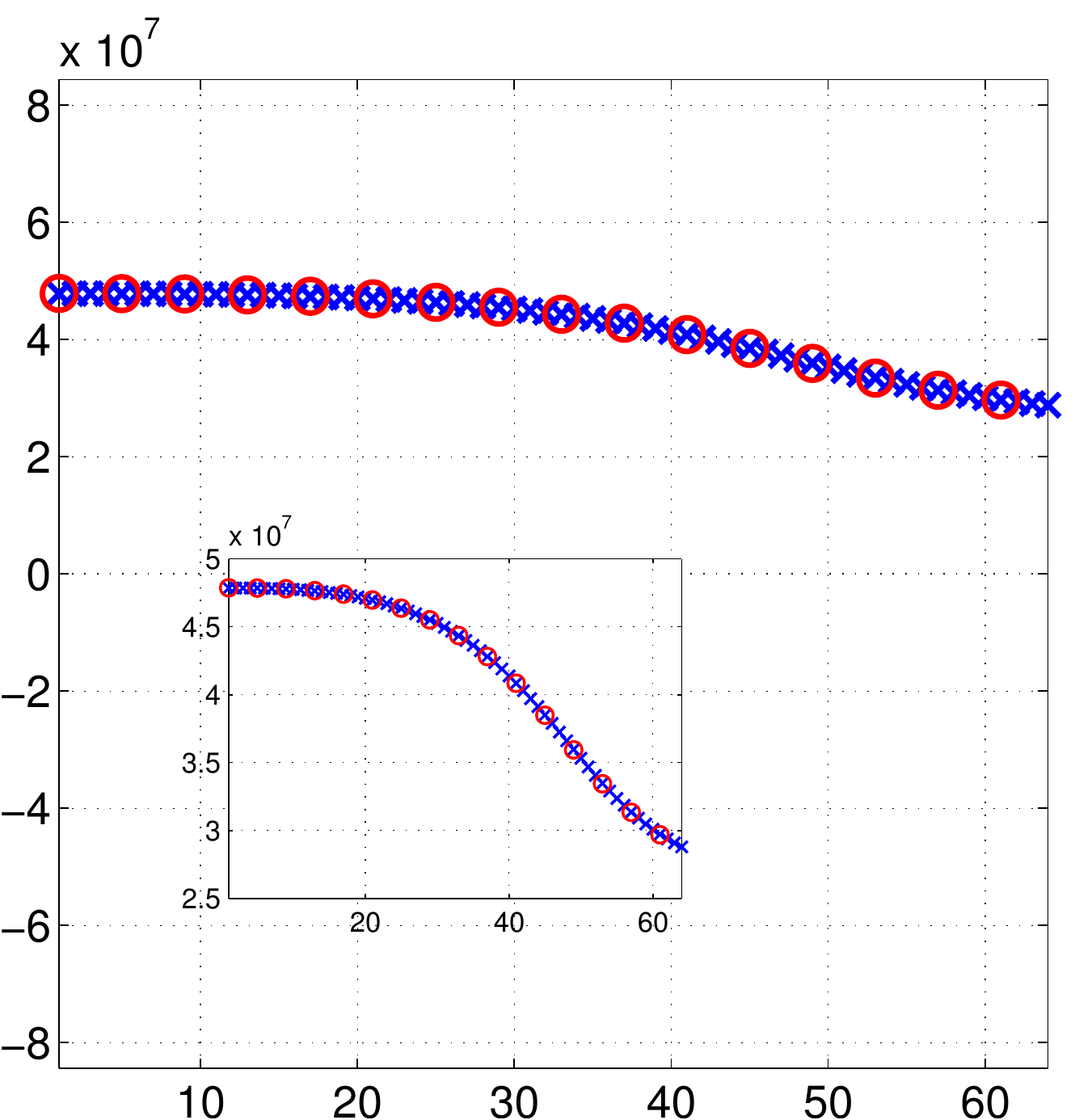}\label{fig:svals-rm}
}\;
\subfloat[$\sigma_2 v_2(s)$]{
\includegraphics[width=0.3\linewidth]{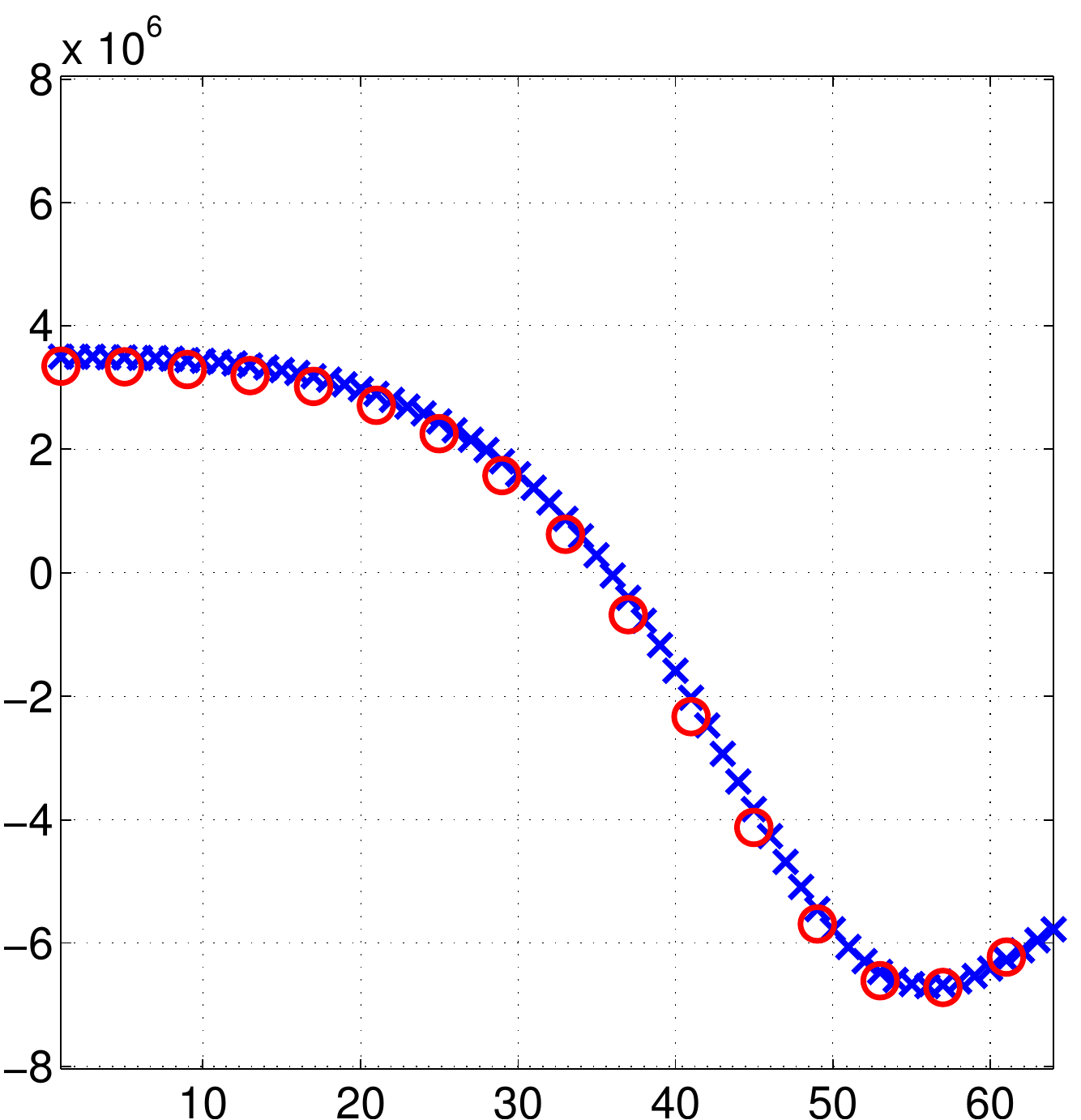}
}\\
\subfloat[$\sigma_3 v_3(s)$]{
\includegraphics[width=0.3\linewidth]{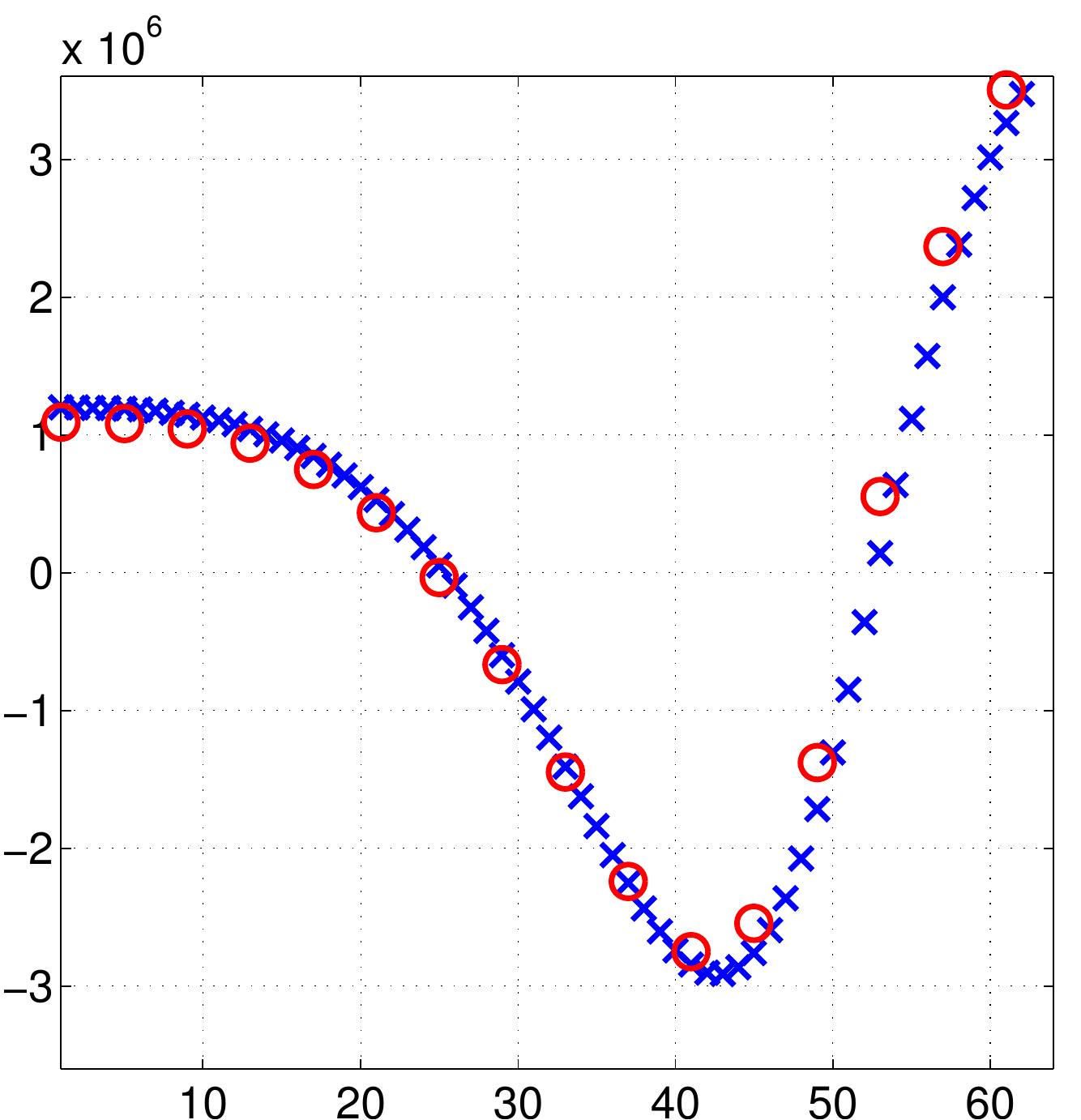}
}\:
\subfloat[$\sigma_4 v_4(s)$]{
\includegraphics[width=0.3\linewidth]{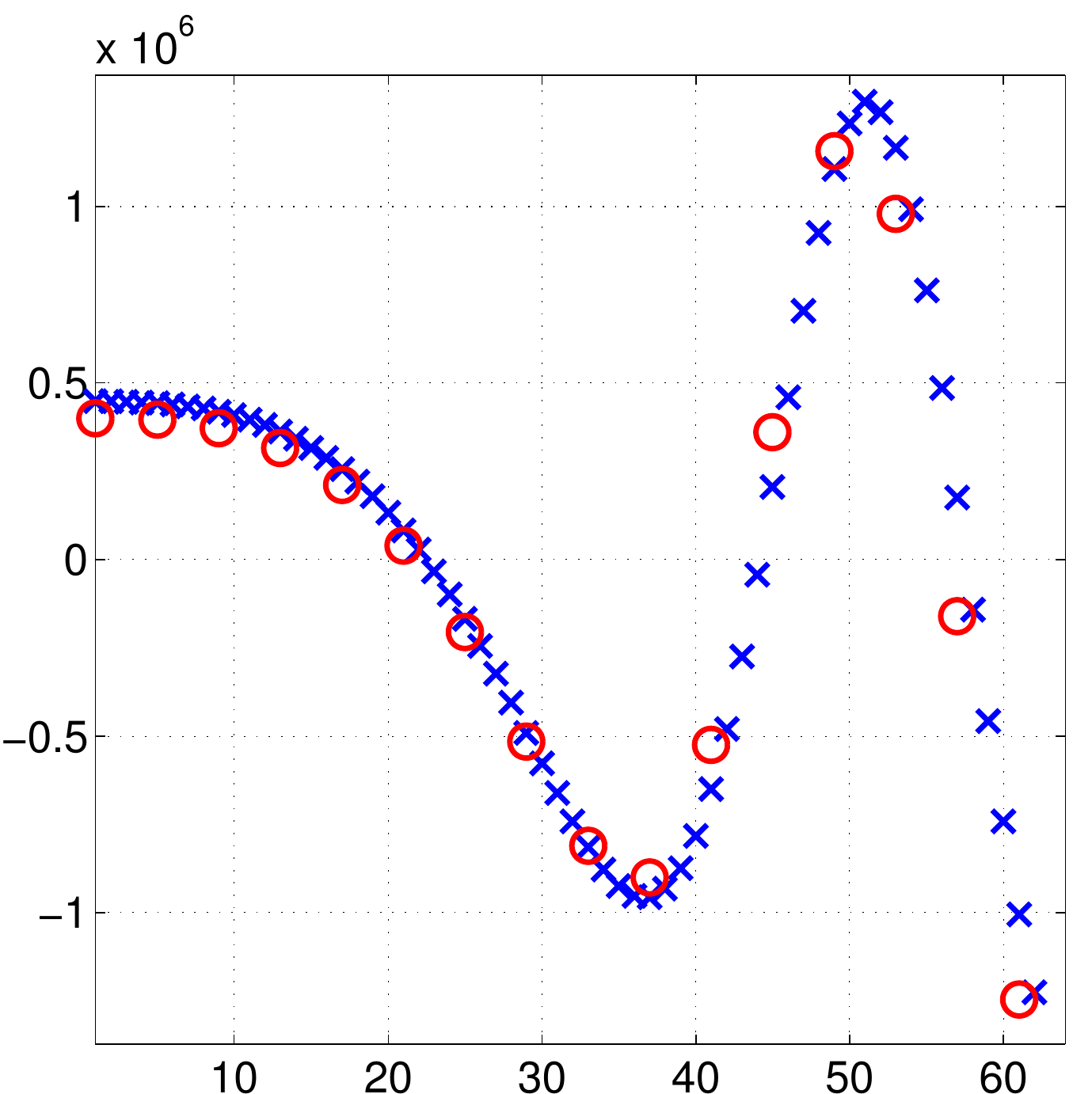}
}\;
\subfloat[$\sigma_5 v_5(s)$]{
\includegraphics[width=0.3\linewidth]{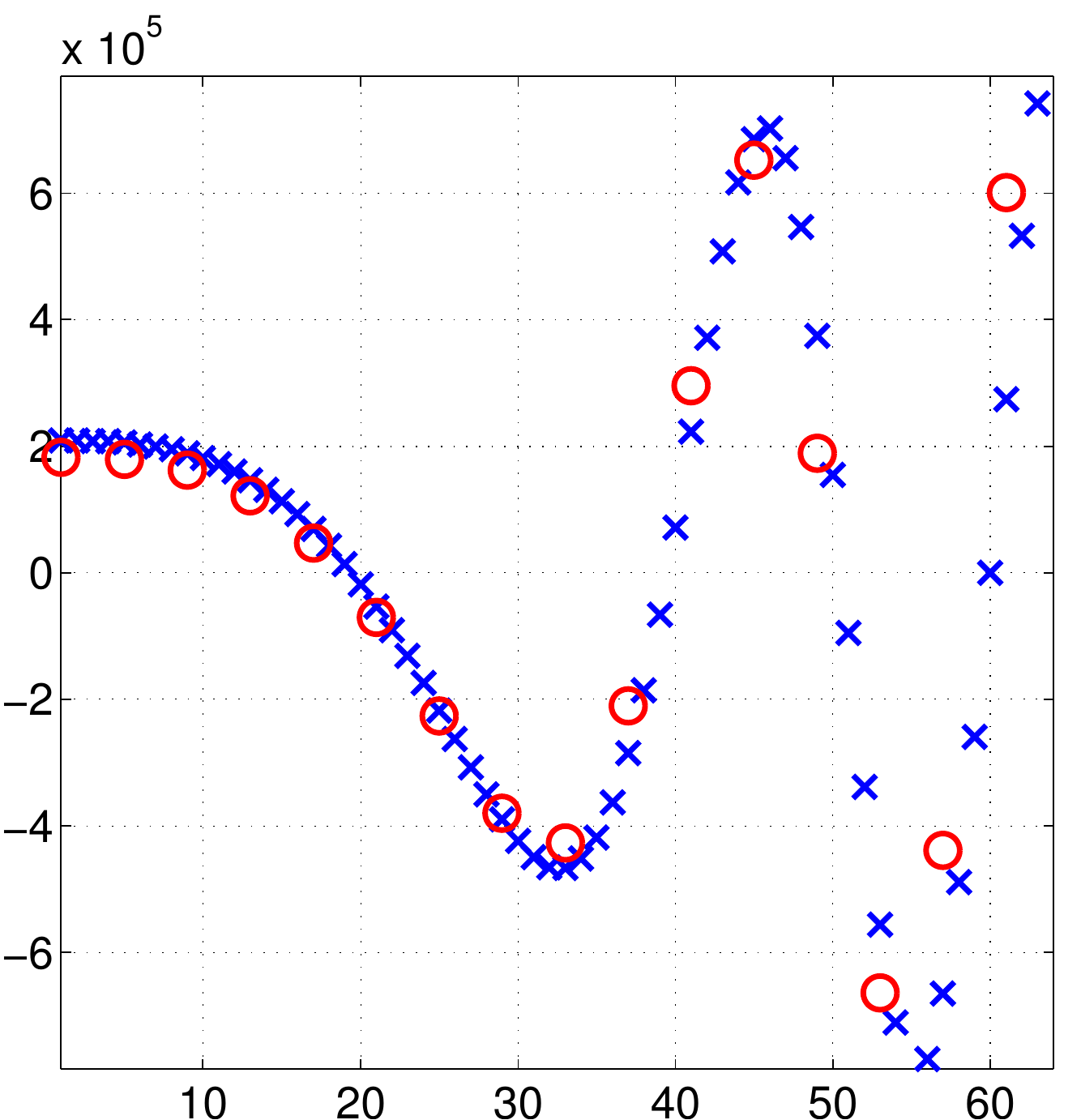}
}\\
\subfloat[$\sigma_6 v_6(s)$]{
\includegraphics[width=0.3\linewidth]{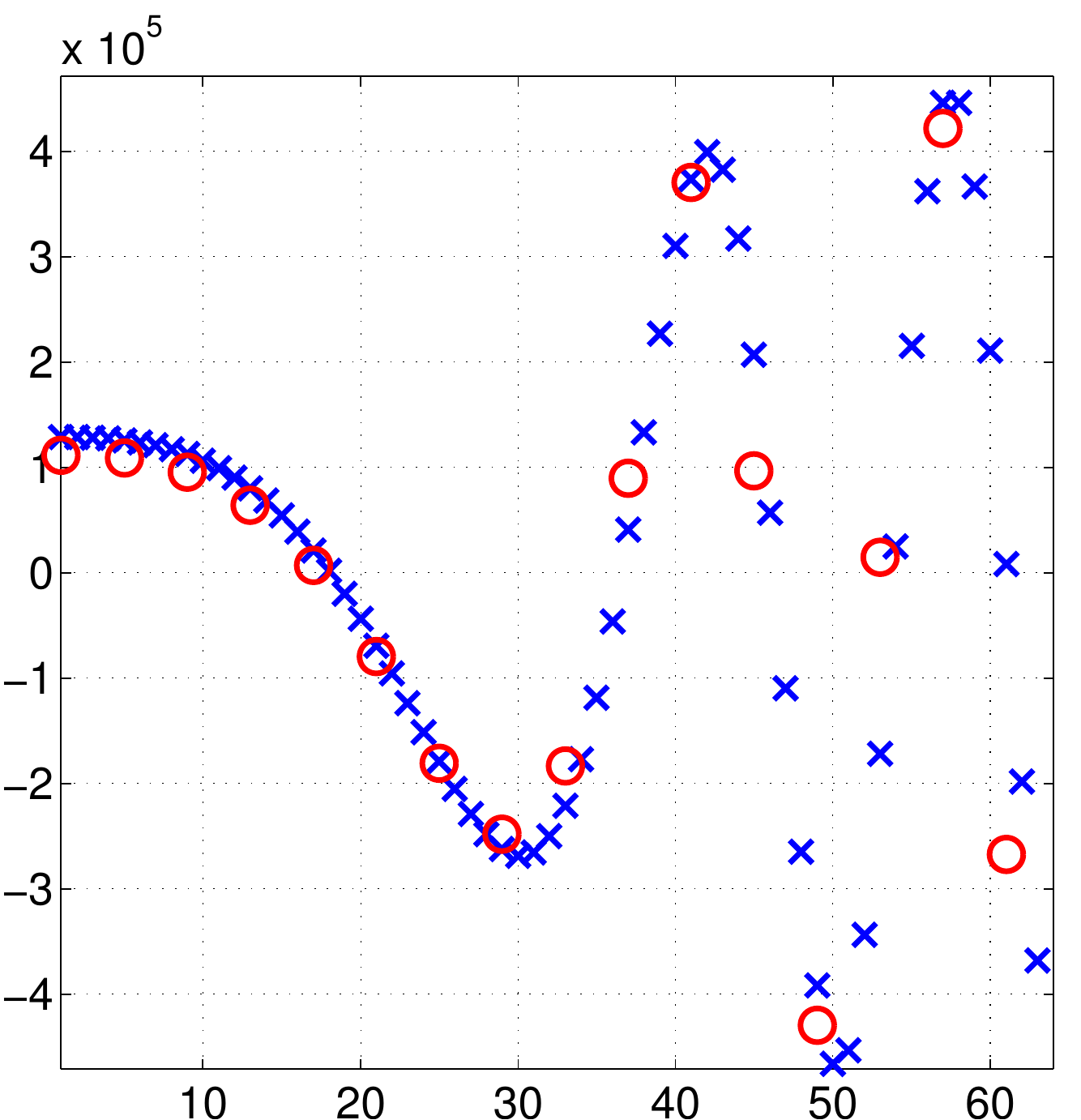}
}\;
\subfloat[$\sigma_7 v_7(s)$]{
\includegraphics[width=0.3\linewidth]{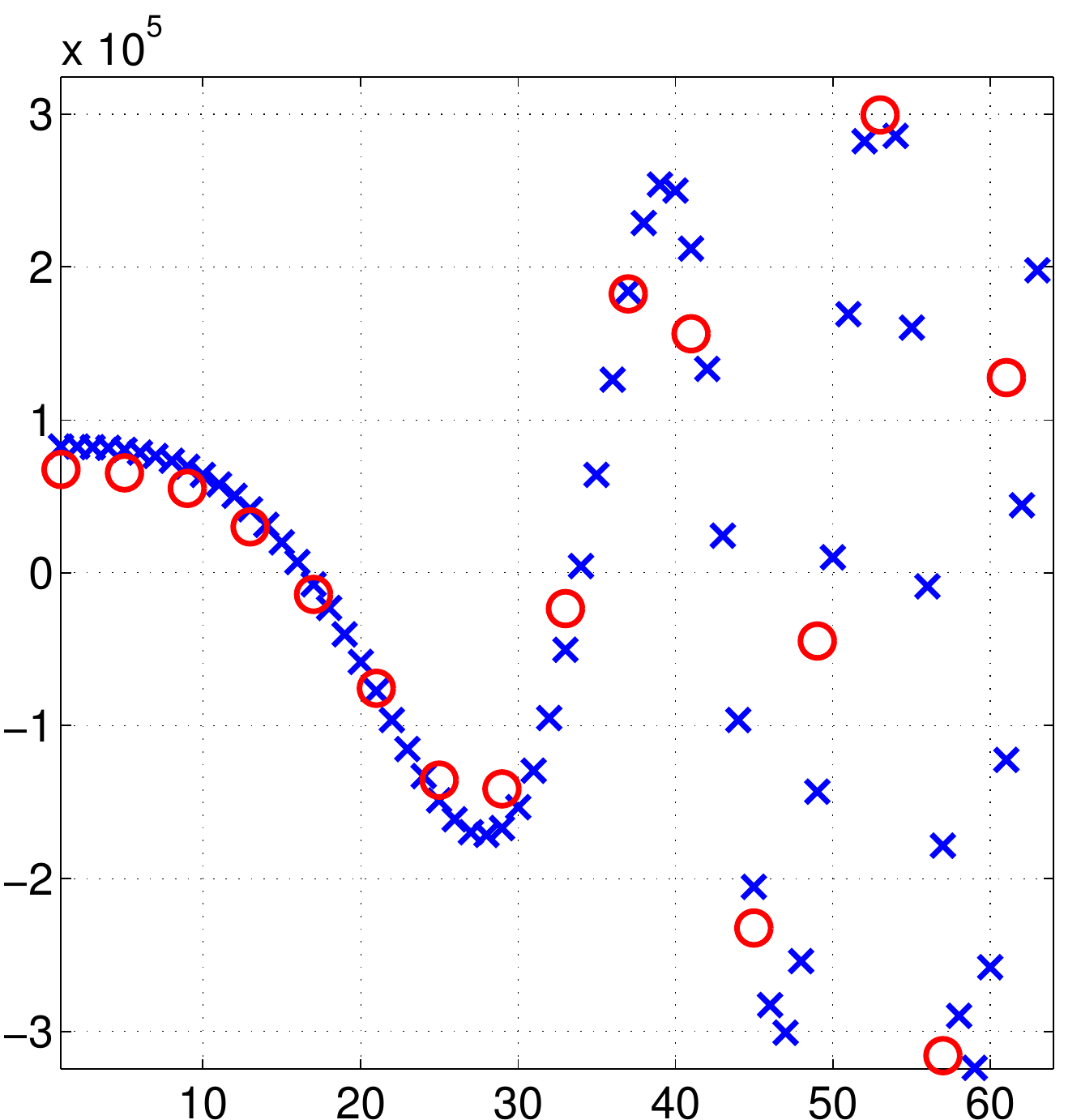}
}\;
\subfloat[$\sigma_8 v_8(s)$]{
\includegraphics[width=0.3\linewidth]{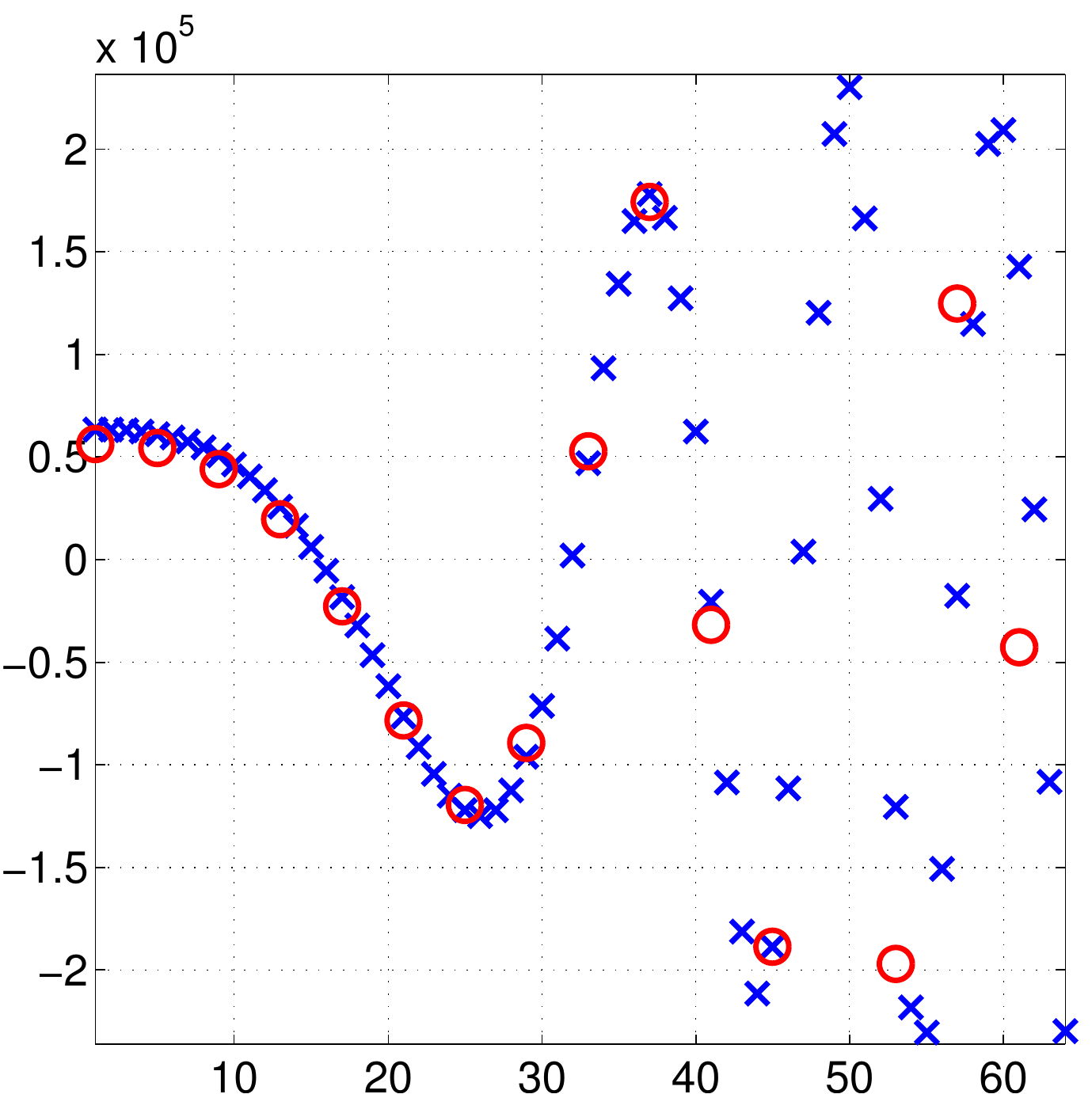}
}
\caption{The top left figure shows the singular values of a finely
  resolved $\mF$ with 64 columns (blue x's) and a coarsely resolved
  $\mF$ with 16 columns (red o's). Each are scaled by the maximum
  singular value from each set. The remaining figures show the
  approximations of the singular functions $v_1(s)$ through $v_8(s)$
  scaled by the respective singular values for $\mF$ with 64 columns
  (blue x's) and 16 columns (red o's). (Colors are visible in the
  electronic version.) }
\label{fig:tempsvd}
\end{figure} 

To choose the threshold $\bar{\tau}$ that defines the splitting
described in Section \ref{sec:chooseR}, we use a subset of the
simulations as a testing set. In particular, we choose $s_j =
0.039\,j$ with $j=3,7,11,\dots,59$ as the values of $s$ whose
simulations we will use for testing. Note that these correspond to the
midpoints of the intervals defined by the values of $s$ used for
training. Figure \ref{fig:error-train} shows the errors as a function
of the bubble radius $s$ and the variation threshold
$\bar{\tau}$. After $\bar{\tau}=0.55$, the approximation does not
improve with more terms (i.e., larger $R$ in \eqref{eq:rommodel}), so
we choose $\bar{\tau}=0.55$ since we want a ROM with the fewest terms.
Table \ref{tab:thresh-train} displays the splitting $R$ and the
associated error $\sE$ for the different values of $s$ using variation
threshold $\bar{\tau}$.

\begin{figure}
\begin{minipage}[b]{0.49\textwidth}
  \centering \includegraphics[width=0.9\linewidth]{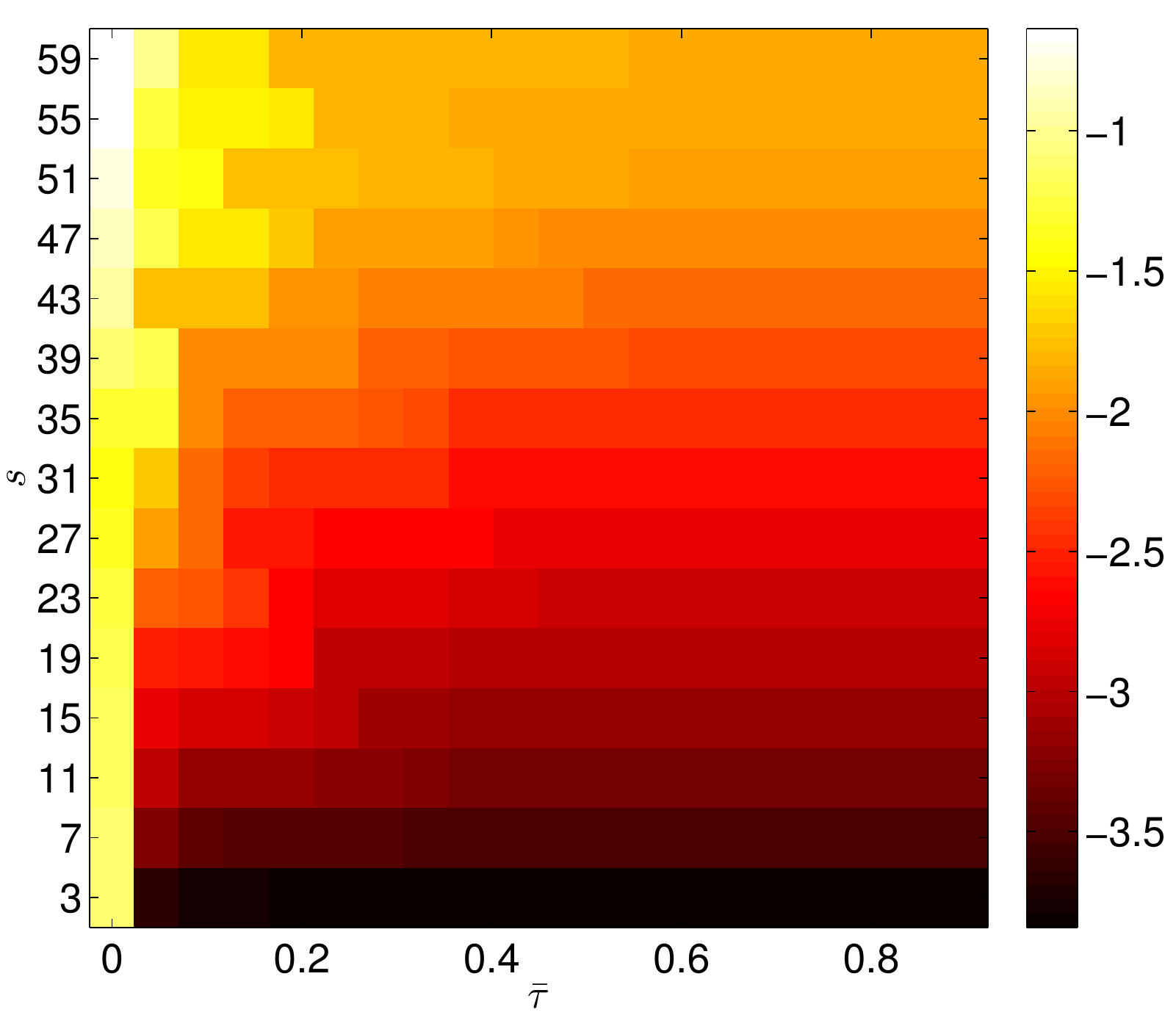}%
  \captionof{figure}{The log of the relative error in the mean
    prediction of the ROM as a function of $s$ and the threshold
    $\bar{\tau}$. (Colors are visible in the electronic version.)}
  \label{fig:error-train}
\end{minipage}
\hfill
\begin{minipage}[b]{0.49\textwidth}
\centering
\begin{tabular}{c|c|c}
$s$ & $R(s,\bar{\tau})$ & $\sE(s,\bar{\tau})$\\
\hline
0.08 & 16 & 1.00e-04\\ 
0.23 & 15 & 2.00e-04\\
0.39 & 14 & 4.00e-04\\
0.55 & 13 & 6.00e-04\\
0.70 & 13 & 8.00e-04\\
0.86 & 12 & 1.10e-03\\
1.01 & 11 & 1.50e-03\\
1.17 & 10 & 2.10e-03\\ 
1.33 & 9 & 3.10e-03\\
1.48 & 8 & 4.50e-03\\ 
1.64 & 8 & 6.50e-03\\ 
1.79 & 7 & 8.20e-03\\
1.95 & 7 & 1.07e-02\\
2.11 & 6 & 1.23e-02\\ 
2.26 & 6 & 1.39e-02
\end{tabular}
\captionof{table}{The split and the corresponding ROM error for
  $\bar{\tau}=0.55$ and different values of $s$.}
\label{tab:thresh-train}
\end{minipage}
\end{figure}

Finally, we visually compare the error in the ROM with the space-time
varying confidence measure.  Figure \ref{fig:romvar} displays the ROM
error and the confidence measure at the final time $t_f$ for one
realization of the bubble locations and two values of the bubble
radius, $s=0.39\,\mathrm{cm}$ and $s=1.95\,\mathrm{cm}$. Both measures
are larger near the bubble boundaries and larger near the face
containing the heat source. Visualizing these measures enables such
qualitative observations and comparisons.

\begin{figure}[t]
\centering
\subfloat[Error, $s=0.39$ cm]{
\includegraphics[width=0.45\linewidth]{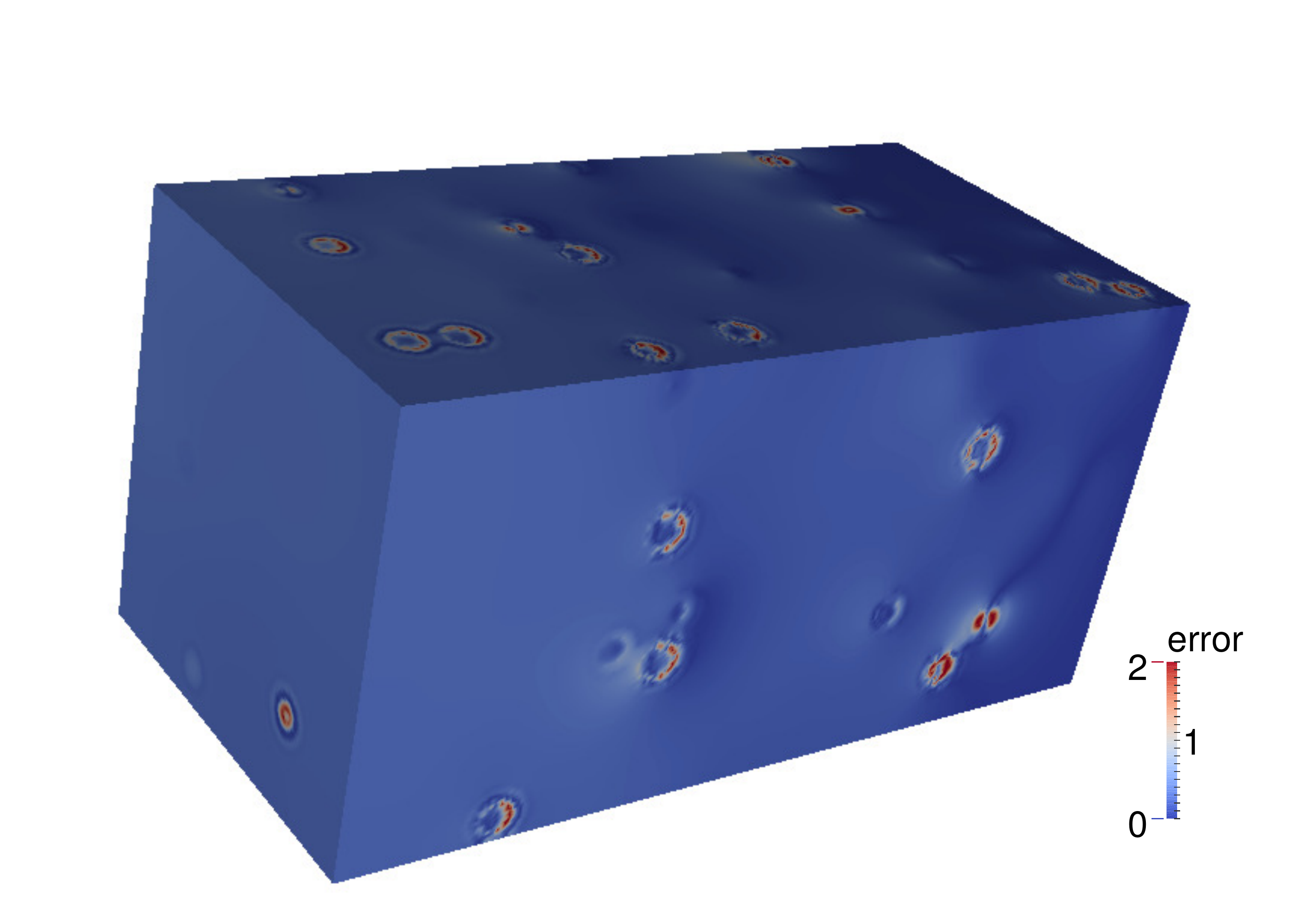}
}\;
\subfloat[Std, $s=0.39$ cm]{
\includegraphics[width=0.45\linewidth]{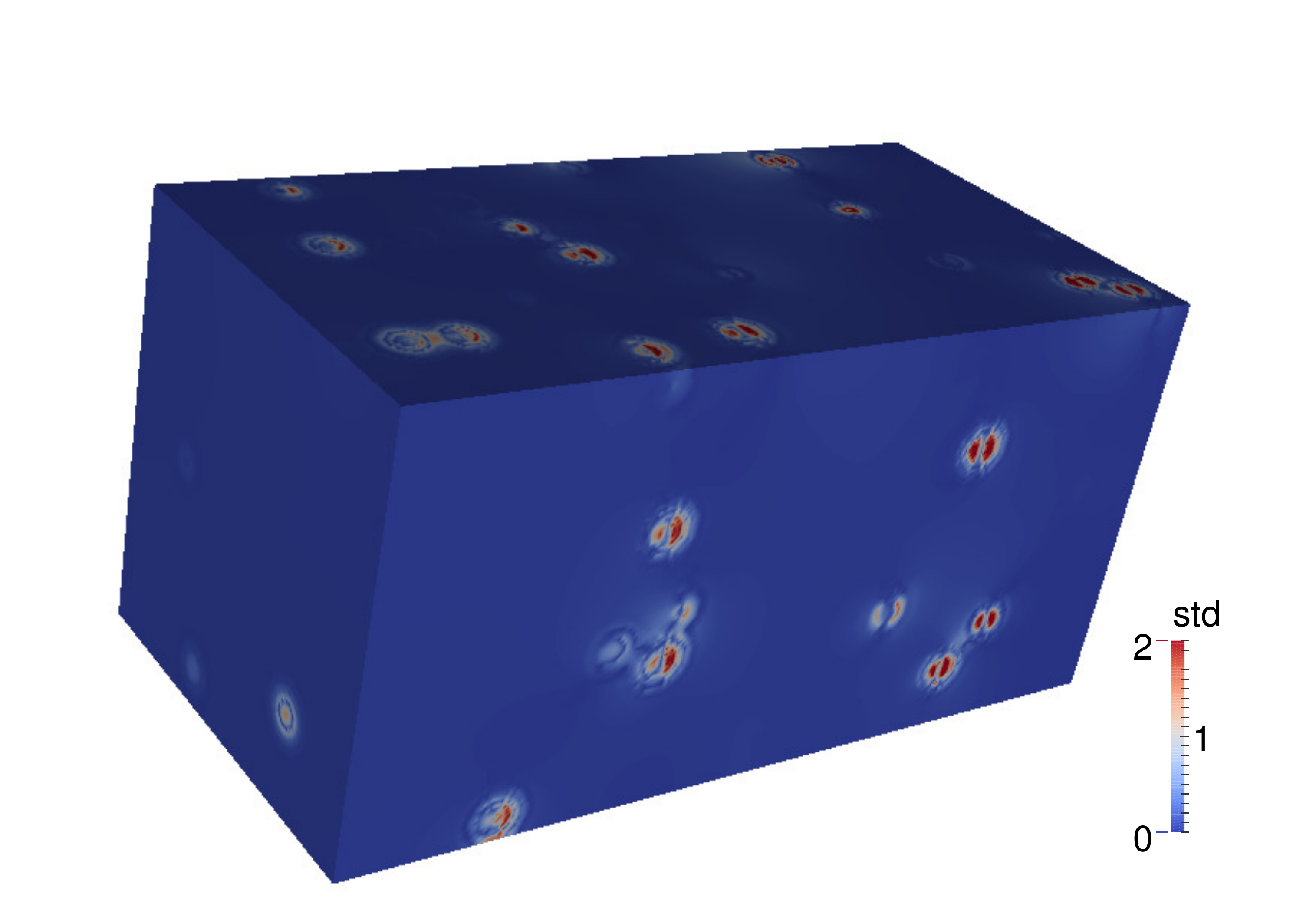}
}\\
\subfloat[Error, $s=1.95$ cm]{
\includegraphics[width=0.45\linewidth]{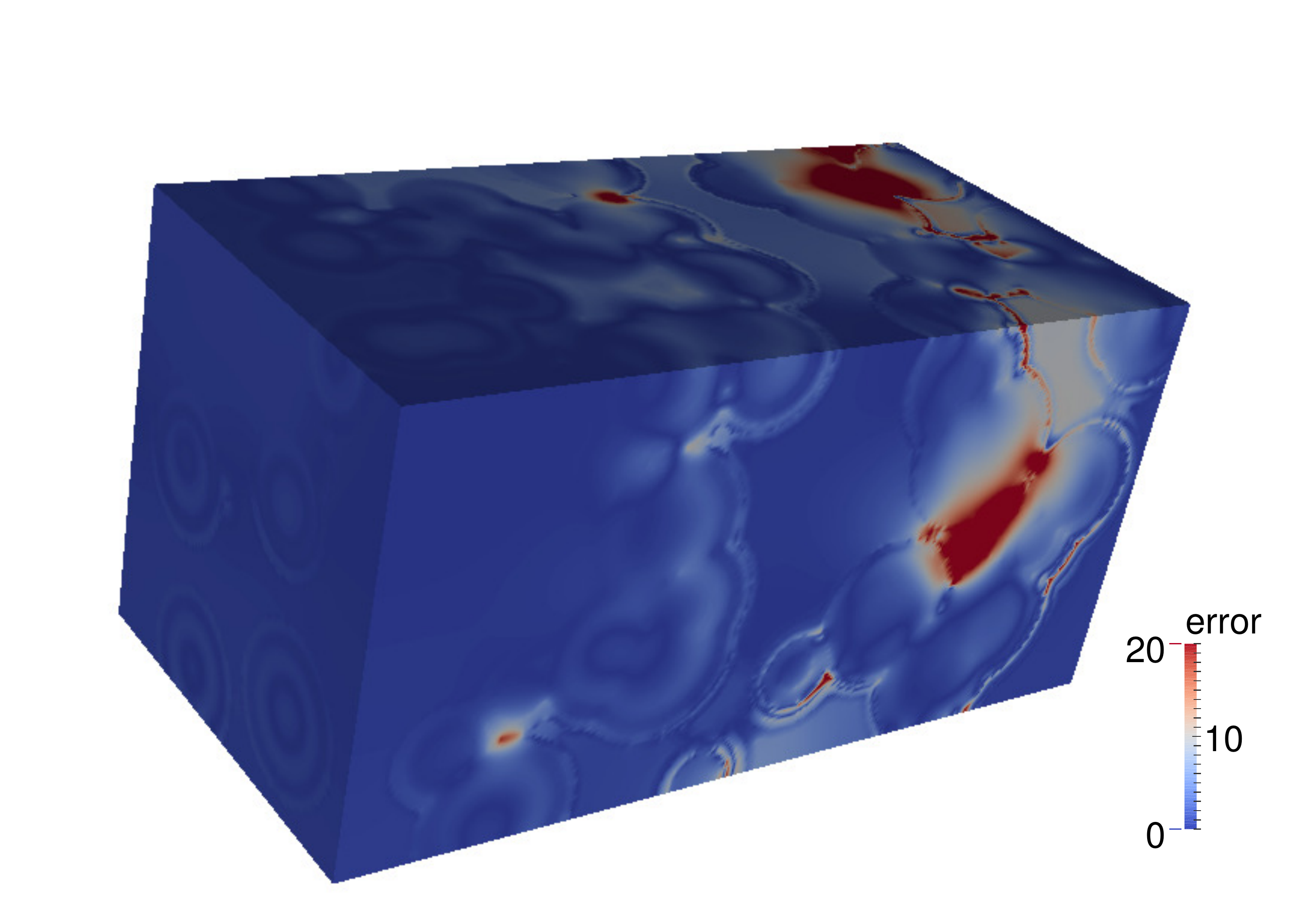}
}\;
\subfloat[Std, $s=1.95$ cm]{
\includegraphics[width=0.45\linewidth]{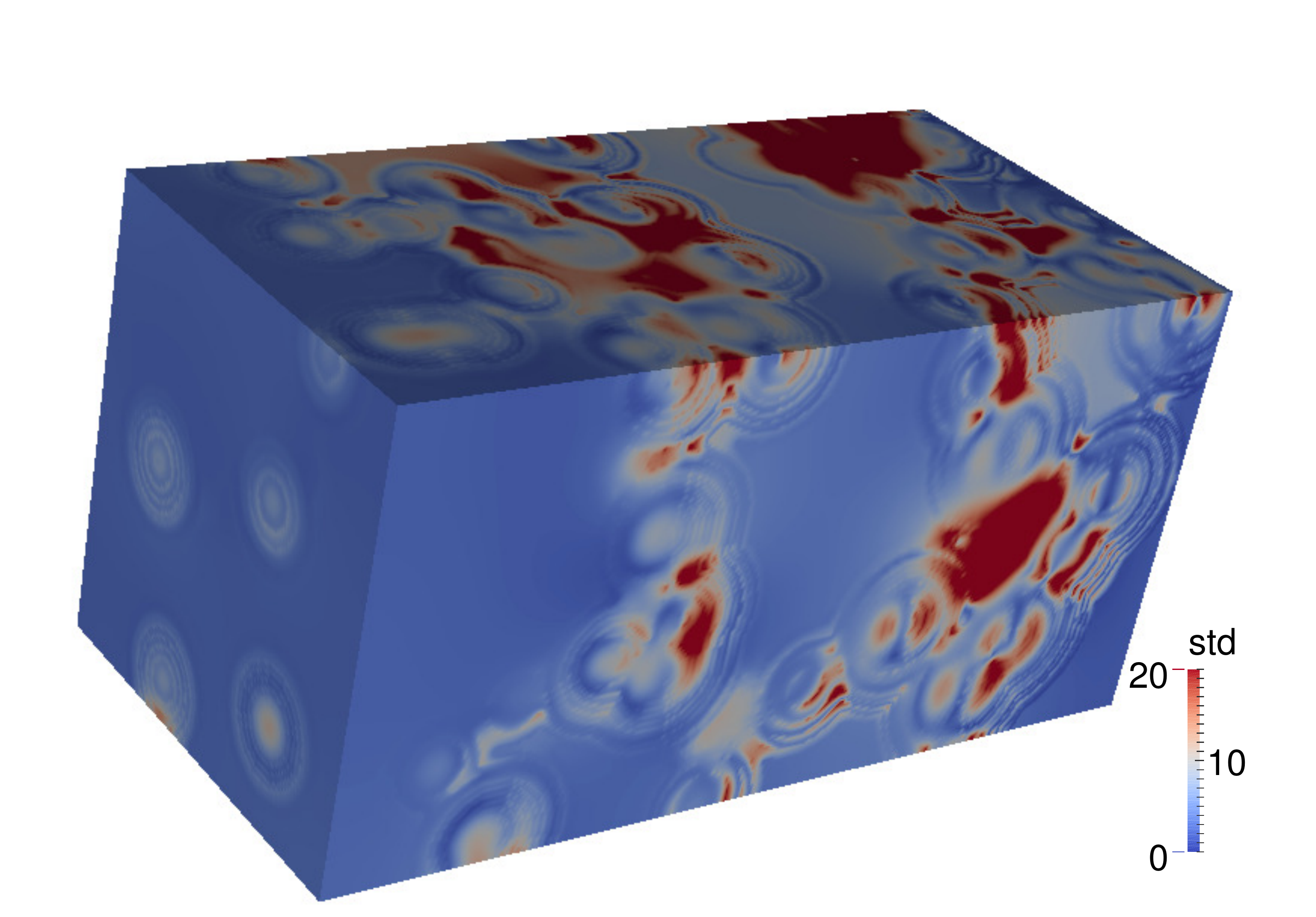}
}
\caption{Absolute error in the reduce order model compared with the prediction
  standard deviation for one realization of the bubble locations at
  the final time for two values of the bubble radius, $s=0.39$ and
  $s=1.95$ cm. (Colors are visible in the electronic version.)}
\label{fig:romvar}
\end{figure} 

\subsection{Comparison with a response surface}
\label{sec:compare}
One question that arises frequently in the context of reduced-order
modeling is, if one is only interested in a scalar quantity of
interest from the full PDE solution, then what is the advantage of
approximating the full solution with a reduced-order model? Why not
just use a scalar response surface to approximate the quantity of
interest as a function of the parameters?  To address this question,
we compare two approaches for the parameter study in Section
\ref{sec:uqstudy}:
\begin{enumerate}
\item Use a response surface to interpolate the means of each of the
  two quantities of interest over a range of bubble radii.  We use the
  quantities of interest at bubble radii $s_j=0.039\,j$ for
  $j=3,7,11,\dots,59$ to decide the form of the response surface:
  piecewise linear, nearest neighbor, cubic spline, or piecewise cubic
  Hermite interpolation (PCHIP).  The response surface form with the
  lowest testing error is constructed from the mean quantities of
  interest for bubble radii $s_j=0.039\,j$ for
  $j=1,5,9,\dots,61$---which are the same values whose simulations are
  used to construct the ROM.  The response surface prediction is then
  computed for $j=1,2,3,\dots,61$.
\item Use the ROM to approximate the temperature field on the far face
  at the final time for each realization of the bubble location. Then
  compute the two quantities of interest for each approximated far
  face temperature distribution, and compute a Monte Carlo
  approximation of the mean (i.e., a simple average).
\end{enumerate}
The results of this study are shown in Figure \ref{fig:rscompare}. For
the first quantity of interest (the average temperature over the far
face), the cubic spline response surface approach adequately captures
the behavior as a function of the bubble radius due to the relative
smoothness of the response. However, the PCHIP response surface
approximation of the second quantity of interest (the proportion of
far face temperature that exceeds $475\,^{\circ}\mathrm{K}$) is
substantially less accurate than the ROM due to the sharp transition
and the low resolution in the parameter space. A global polynomial or
radial basis function approximation would fare worse (e.g., exhibit
Gibbs oscillations near the transition) due to the global nature of
the basis.  We conclude that the choice of ROM versus response surface
depends on the quantity of interest. Broadly speaking, if the quantity
of interest is a smooth function of the parameters, then a response
surface is likely sufficient. However, if the quantity of interest is
a highly nonlinear or discontinuous function of the full PDE solution,
then computing such a quantity from the ROM approximation will yield
better results. While it was not pertinent to this example, the
ROM-based approach also provides the 95\% confidence bounds for the
Monte Carlo estimates, since the temperature distribution for each
realization of the bubble locations is approximated.

\begin{figure}[t]
\centering
\subfloat[]{
\includegraphics[width=0.45\linewidth]{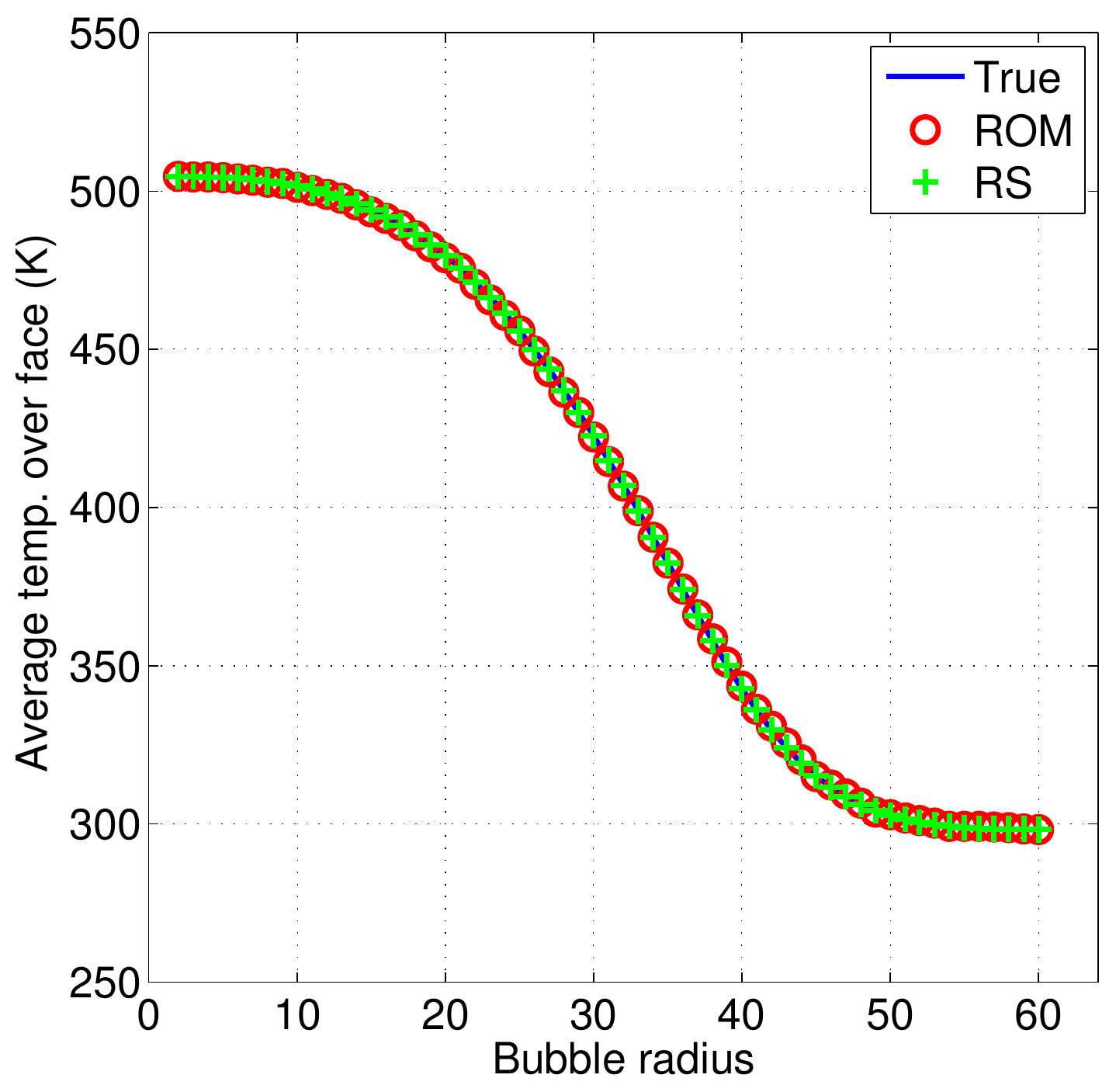}
}\;
\subfloat[]{
\includegraphics[width=0.45\linewidth]{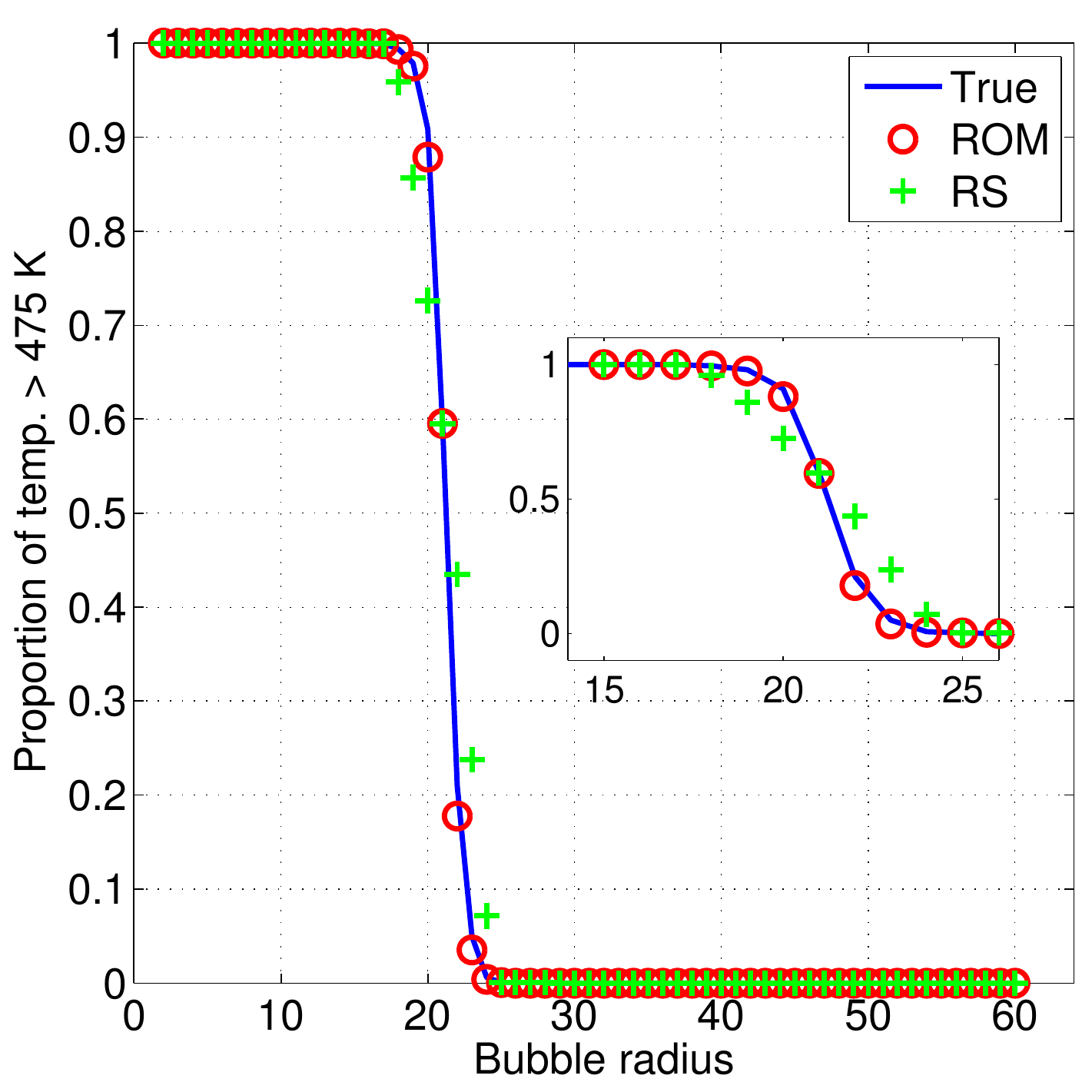}
}
\caption{Comparing ROM with response surface in the UQ study. (Colors are
  visible in the electronic version.) }
\label{fig:rscompare}
\end{figure} 

\subsection{Computational considerations}
We end this section with a few notes on the experience of running 8192
large-scale simulations, transferring them to a Hadoop cluster, and
building the reduced-order model.

Each heat transfer simulation took approximately four hours using
eight processors on Sandia's Red Sky. Communication times were
negligible, but some runs required multiple tries due to occasional
network failures. Also, some runs had to be duplicated due to node
failures and changes in the code paths. Each mesh with its own
conductivity field took approximately twenty minutes to construct
using Cubit after substantial optimizations. 
Unreliable network data transmissions and bursty data write patterns
(e.g., one hundred jobs on Red Sky simultaneously transferring data to
the Stanford cluster) forced us to write custom codes to validate the
data transfer.

Working with the simulation data involved a few pre- and
post-processing steps, such as interpreting 4TB of Exodus II files
from Aria. The preprocessing steps took approximately 8-15 hours. We
collected precise timing information, but we do not report it as these
times are from a multi-tenant, unoptimized Hadoop cluster where other
jobs with sizes ranging between 100GB and 2TB of data sometimes ran
concurrently. Also, during our computations, we observed failures in
hard disk drives and issues causing entire nodes to fail. Given that
the cluster has 40 cores, these calculations consumed at most
2400 cpu-hours---compared with the 262144 cpu-hours it took to
compute 8192 heat transfer simulations on Red Sky. Thus, evaluating
the ROM was about 100 times faster than computing a full simulation.

We did not compare our Hadoop implementation with an MPI
implementation. The dominant bottleneck in the evaluation of the ROM
is the data I/O involved in processing 4TB of simulation data into a
2.3TB matrix, computing its SVD, writing it to disk, computing the
interpolants, and writing those outputs back to Exodus II files. We
expect that any MPI implementation would take at least 3-4 hours based
on pure I/O considerations (assuming around 1GB/sec sustained data
transfer speeds).  It would also be substantially more complicated to
implement. We used Hadoop primarily for the ease of use of its
programming model.


Although MapReduce is an appealing paradigm for
computing factorizations of tall matrices, the Hadoop
MapReduce ecosystem has not developed simple tools for working with
large databases of spatio-temporal data.  For instance, writing ad hoc
utilities to extract data from Exodus II files and utilities for
simply stated queries like, \emph{retrieve all values of temperature
  with $x=-0.1$ and $t=2000$}, occupied much of the second and third
authors' time (see the \texttt{mr-exodus2farface.py} code in the
online code for this particular script).  We see many opportunities
for high-impact software designed to address the mundane details of
manipulating and analyzing large databases of spatio-temporal
simulation data.

\section{Summary \& conclusions}
\label{sec:conclusion}

We presented a method for building a reduced-order model of the
solution of a parametererized partial differential equation. The
method is based on approximating the factors of a singular value
expansion of the solution using the elements of a singular value
decomposition of a matrix whose columns are spatio-temporally
discretized solutions at different parameter values. The SVD step
compares to reduced basis methods, which project the governing
equations with a subset of the left singular vectors to create small
system whose solution yields the coefficients of the reduced-order
model. In contrast, our method interpolates the right singular vectors
in the parameter space to compute the coefficients of the ROM.  By
examining the gradient of the right singular vectors as the index
increases, we determine a separation between factors we can accurately
interpolate and those whose oscillations are too rapid to be
represented on the parameter grid. This separation yields a mean
prediction and a prediction covariance for each point in the
spatio-temporal domain---similar to Gaussian process regression
models.  We use Hadoop/MapReduce to implement the ROM including the
communication-avoiding, tall-and-skinny SVD, which enables the
computation to scale to outputs from large-scale high-fidelity models.

We tested the model reduction method on a parameter study of
large-scale heat transfer in random media. We compared the results of
the ROM with a standard response surface method for approximating the
scalar quantities of interest, and we found that while the cheaper
response surface was appropriate for a smooth quantity of interest,
the ROM was better at approximating a quantity of interest with a
sharp transition in the parameter space. The 8192 heat transfer
simulations used in the study generated approximately 4 TB of data. In
the course of the study, we applied the MapReduce-based SVD
computation to a matrices with approximately 600 GB and 2.2 TB of
data. We found that existing MapReduce tools for working with such
large-scale simulation data lack robustness and generality. There is
an opportunity in computational science to create better tools to
further simulation-based scientific exploration.

\section{Acknowledgments}
We thank the anonymous reviewers for helpful comments and suggestions.
We also thank Margot Gerritsen at Stanford's Institute for
Computational and Mathematical Engineering for procurement of and
access to the Hadoop cluster. We thank Austin Benson at Stanford for
his superb code development for the TSQR and TSSVD. We thank Joe
Ruthruff at Sandia for his efforts developing the infrastructure to
run the Aria cases. Finally, we thank David Rogers at Sandia and
acknowledge the funding of Sandia's Computer Science Applied Research
(CSAR) and the Advanced Simulation and Computing (ASC) programs.  Sandia is a multiprogram laboratory operated by Sandia Corporation, a Lockheed Martin Company, for the United States Department of Energy under contract DE-AC04-94AL85000.

\section{Appendix}
The tables with material properties for foam and steel.

\begin{table}[H]
\centering
\begin{tabular}{c|c}
\multicolumn{2}{c}{Foam ($\rho=319\,\mathrm{kg/m}^{3}$)} \\
\hline
$T$ (K) & $\kappa$ (W/mK) \\
\hline
303 & 0.0486  \\
523 & 0.0706 \\
\end{tabular} \quad
\begin{tabular}{c|c}
\multicolumn{2}{c}{Steel ($\rho=7900\,\mathrm{kg/m}^{3}$)} \\
\hline
$T$ (K) & $\kappa$ (W/mK) \\
\hline
273 & 13.4 \\
373 & 16.3 \\
773 & 21.8 \\
973 & 26.0 \\
\end{tabular}
\caption{Thermal conductivity $\kappa$.  Linear interpolation is used
  between specified values.  Constant values are used outside of the
  bounds.}
\label{t:cond}
\end{table}

\begin{table}[H]
\centering
\begin{tabular}{c|c}
\multicolumn{2}{c}{Foam ($\rho=319\,\mathrm{kg/m}^{3}$)}  \\
\hline
$T$ (K) & $c_p$ (J/kgK) \\
\hline
296 & 1269 \\
323 & 1356 \\
373 & 1497 \\
423 & 1843 \\
473 & 1900 \\
523 & 2203 \\
\end{tabular}\quad
\begin{tabular}{c|c}
\multicolumn{2}{c}{Steel ($\rho=7900\,\mathrm{kg/m}^{3}$)} \\
\hline
$T$ (K) & $c_p$ (J/kgK) \\
\hline
273 & 502 \\
673 & 565 \\
\end{tabular}
\caption{Specific heat capacity $c_p$.  Linear interpolation is used between
  specified values.  Constant values are used outside of the bounds.}
\label{t:cap}
\end{table}

\bibliographystyle{siam}
\bibliography{mrmr}

\begin{thebibliography}{10}

\bibitem{audouze2009reduced}
{\sc C~Audouze, F~De~Vuyst, and PB~Nair}, {\em Reduced-order modeling of
  parameterized pdes using time-space-parameter principal component analysis},
  International journal for numerical methods in engineering, 80 (2009),
  pp.~1025--1057.

\bibitem{audouze2013nonintrusive}
{\sc Christophe Audouze, Florian De~Vuyst, and Prasanth~B Nair}, {\em
  Nonintrusive reduced-order modeling of parametrized time-dependent partial
  differential equations}, Numerical Methods for Partial Differential
  Equations,  (2013).

\bibitem{Babuska2007}
{\sc Ivo Babu{\v{s}}ka, Fabio Nobile, and Ra{\'u}l Tempone}, {\em A stochastic
  collocation method for elliptic partial differential equations with random
  input data}, SIAM Journal on Numerical Analysis, 45 (2007), pp.~1005--1034.

\bibitem{Benner2013}
{\sc Peter Benner, Serkan Gugercin, and Karen Willcox}, {\em A survey of model
  reduction methods for parametric systems}, tech. report, Max Planck Institute
  Magdeburg, 2013.

\bibitem{Benson-preprint-direct-tsqr}
{\sc Austin Benson, David~F. Gleich, and James Demmel}, {\em Direct
  tall-and-skinny {QR} factorizations in mapreduce architectures}, arXiv, cs.DC
  (2012), p.~1301.1071.

\bibitem{Blacker1994}
{\sc Ted~D Blacker, WJ~Bohnhoff, and TL~Edwards}, {\em Cubit mesh generation
  environment. volume 1: Users manual}, tech. report, Sandia National Labs.,
  Albuquerque, NM (United States), 1994.

\bibitem{blackford1997scalapack}
{\sc L~Susan Blackford}, {\em ScaLAPACK user's guide}, vol.~4, Society for
  Industrial and Applied Mathematics, 1997.

\bibitem{Bui2003}
{\sc T.~Bui-Thanh, K.~Willcox, and M.~Damodaran}, {\em Applications of proper
  orthogonal decomposition for inviscid transonic aerodynamics}, tech. report,
  MIT, 2003.
\newblock http://hdl.handle.net/1721.1/3694.

\bibitem{Bui2008}
{\sc Tan Bui-Thanh, Karen Willcox, and Omar Ghattas}, {\em Model reduction for
  large-scale systems with high-dimensional parametric input space}, SIAM
  Journal on Scientific Computing, 30 (2008), pp.~3270--3288.

\bibitem{carlberg2013gnat}
{\sc Kevin Carlberg, Charbel Farhat, Julien Cortial, and David Amsallem}, {\em
  The {GNAT} method for nonlinear model reduction: effective implementation and
  application to computational fluid dynamics and turbulent flows}, Journal of
  Computational Physics,  (2013).

\bibitem{chan1982improved}
{\sc Tony~F Chan}, {\em An improved algorithm for computing the singular value
  decomposition}, ACM Transactions on Mathematical Software (TOMS), 8 (1982),
  pp.~72--83.

\bibitem{Hadoop2012-0202-cdh3}
{\sc Cloudera}, {\em Hadoop version 0.20.2 in cloudera hadoop distribution
  version cdh3u4}.
\newblock \url{http://www.cloudera.com}, 2012.

\bibitem{Constantine2011}
{\sc Paul~G Constantine and David~F Gleich}, {\em Tall and skinny qr
  factorizations in mapreduce architectures}, in Proceedings of the second
  international workshop on {MapReduce} and its applications, ACM, 2011,
  pp.~43--50.

\bibitem{Constantine2012}
{\sc Paul~G Constantine and Qiqi Wang}, {\em Residual minimizing model
  interpolation for parameterized nonlinear dynamical systems}, SIAM Journal on
  Scientific Computing, 34 (2012), pp.~A2118--A2144.

\bibitem{Dean2004-MapReduce}
{\sc Jeffrey Dean and Sanjay Ghemawat}, {\em {MapReduce}: Simpliﬁed data
  processing on large clusters}, in Proceedings of the 6th Symposium on
  Operating Systems Design and Implementation (OSDI2004), 2004, pp.~137--150.

\bibitem{Demmel-2012-CAQR}
{\sc James Demmel, Laura Grigori, Mark Hoemmen, and Julien Langou}, {\em
  Communication-optimal parallel and sequential qr and lu factorizations}, SIAM
  Journal on Scientific Computing, 34 (2012), pp.~A206--A239.

\bibitem{goss2008inlet}
{\sc Jennifer Goss and Kamesh Subbarao}, {\em Inlet shape optimization based on
  pod model reduction of the euler equations}, AIAA, 5809 (2008), p.~2008.

\bibitem{Hansen1988}
{\sc PC~Hansen}, {\em Computation of the singular value expansion}, Computing,
  40 (1988), pp.~185--199.

\bibitem{Hansen2010}
{\sc Per~Christian Hansen}, {\em Discrete inverse problems: insight and
  algorithms}, vol.~7, Society for Industrial and Applied Mathematics, 2010.

\bibitem{Hansen2006}
{\sc Per~Christian Hansen, Misha~Elena Kilmer, and Rikke~H{\o}j Kjeldsen}, {\em
  Exploiting residual information in the parameter choice for discrete
  ill-posed problems}, BIT Numerical Mathematics, 46 (2006), pp.~41--59.

\bibitem{He-2008-MARS}
{\sc Bingsheng He, Wenbin Fang, Qiong Luo, Naga~K. Govindaraju, and Tuyong
  Wang}, {\em Mars: a mapreduce framework on graphics processors}, in
  Proceedings of the 17th international conference on Parallel architectures
  and compilation techniques, PACT '08, New York, NY, USA, 2008, ACM,
  pp.~260--269.

\bibitem{lee2011reduced}
{\sc Kyunghoon Lee, Taewoo Nam, Christopher Perullo, and Dimitri~N Mavris},
  {\em Reduced-order modeling of a high-fidelity propulsion system simulation},
  AIAA journal, 49 (2011), pp.~1665--1682.

\bibitem{lumley2012turbulence}
{\sc John~L Lumley, Gahl Berkooz, and Clarence~W Rowley}, {\em Turbulence,
  coherent structures, dynamical systems and symmetry}, Cambridge University
  Press, 2012.

\bibitem{ly2001modeling}
{\sc Hung~V Ly and Hien~T Tran}, {\em Modeling and control of physical
  processes using proper orthogonal decomposition}, Mathematical and computer
  modelling, 33 (2001), pp.~223--236.

\bibitem{Mori-2012-allreduce}
{\sc Daisuke Mori, Yusaku Yamamoto, and Shao-Liang Zhang}, {\em Backward error
  analysis of the allreduce algorithm for householder qr decomposition}, Japan
  Journal of Industrial and Applied Mathematics, 29 (2012), pp.~111--130.

\bibitem{sand07-2734}
{\sc P.K. Notz, S.R. Subia, M.M. Hopkins, H.K. Moffat, and D.R. Noble}, {\em
  Aria 1.5: User manual}, Tech. Report SAND2007-2734, Sandia National
  Laboratories, Albuquerque, NM 87185 and Livermore, CA 94551, Apr. 2007.

\bibitem{ostrowski2005estimation}
{\sc Ziemowit Ostrowski, Ryszard~A Bia{\l}ecki, and Alain~J Kassab}, {\em
  Estimation of constant thermal conductivity by use of proper orthogonal
  decomposition}, Computational Mechanics, 37 (2005), pp.~52--59.

\bibitem{Plimpton-2011-MRMPI}
{\sc Steven~J. Plimpton and Karen~D. Devine}, {\em Mapreduce in mpi for
  large-scale graph algorithms}, Parallel Computing, 37 (2011), pp.~610--632.

\bibitem{Rasmussen2006}
{\sc Carl~Edward Rasmussen}, {\em Gaussian processes for machine learning},
  (2006).

\bibitem{Talbot-2011-pheonix++}
{\sc Justin Talbot, Richard~M. Yoo, and Christos Kozyrakis}, {\em Phoenix++:
  modular mapreduce for shared-memory systems}, in Proceedings of the second
  international workshop on MapReduce and its applications, MapReduce '11, New
  York, NY, USA, 2011, ACM, pp.~9--16.

\bibitem{Hadoop2010}
{\sc Various}, {\em Hadoop version 0.20, cloudera cdh3}.
\newblock \url{http://hadoop.apache.org}, \url{http://cloudera.com}, 2010.

\bibitem{Xiu2005}
{\sc Dongbin Xiu and Jan~S Hesthaven}, {\em High-order collocation methods for
  differential equations with random inputs}, SIAM Journal on Scientific
  Computing, 27 (2005), pp.~1118--1139.

\end{thebibliography}

\end{document}